\documentclass[11pt]{article}
\usepackage[latin9]{inputenc}
\usepackage{geometry}
\usepackage{color}
\usepackage{array}
\usepackage{calc}
\usepackage{multirow}
\usepackage{amsmath}
\usepackage{amsthm}
\usepackage{amssymb}
\usepackage{graphicx}
\PassOptionsToPackage{normalem}{ulem}
\usepackage{ulem}
\usepackage[unicode=true,pdfusetitle,
 bookmarks=true,bookmarksnumbered=false,bookmarksopen=false,
 breaklinks=false,pdfborder={0 0 1},backref=page,colorlinks=true]
 {hyperref}
\hypersetup{
 linkcolor=blue}

\makeatletter

\providecommand{\tabularnewline}{\\}

\numberwithin{equation}{section}
\numberwithin{figure}{section}
\theoremstyle{plain}
\newtheorem{thm}{\protect\theoremname}[section]
\theoremstyle{plain}
\newtheorem{lem}[thm]{\protect\lemmaname}
\theoremstyle{definition}
\newtheorem{defn}[thm]{\protect\definitionname}
\theoremstyle{remark}
\newtheorem{rem}[thm]{\protect\remarkname}
\theoremstyle{plain}
\newtheorem{cor}[thm]{\protect\corollaryname}
\theoremstyle{plain}
\newtheorem{prop}[thm]{\protect\propositionname}
\theoremstyle{definition}
\newtheorem{example}[thm]{\protect\examplename}
\theoremstyle{plain}
\newtheorem{fact}[thm]{\protect\factname}
\theoremstyle{remark}
\newtheorem{claim}[thm]{\protect\claimname}

\usepackage{graphicx}
\usepackage{appendix}

\usepackage{enumitem}
\usepackage[backref=page]{hyperref}

\makeatother

\providecommand{\claimname}{Claim}
\providecommand{\corollaryname}{Corollary}
\providecommand{\definitionname}{Definition}
\providecommand{\examplename}{Example}
\providecommand{\factname}{Fact}
\providecommand{\lemmaname}{Lemma}
\providecommand{\propositionname}{Proposition}
\providecommand{\remarkname}{Remark}
\providecommand{\theoremname}{Theorem}

\begin{document}
\global\long\def\defi{\stackrel{\mathrm{def}}{=}}%
 
\global\long\def\Hom{\mathrm{Hom}}%
 
\global\long\def\wedger{{\textstyle \bigvee^{r}S^{1}} }%
 
\global\long\def\MCG{\mathrm{MCG}}%
\global\long\def\mcg{\MCG}%
 
\global\long\def\trw{{\cal T}r_{w} }%
 
\global\long\def\tr{{\cal T}r }%
\global\long\def\trwl{{\cal T}r_{w_{1},\ldots,w_{\ell}} }%
\global\long\def\wl{w_{1},\ldots,w_{\ell}}%
\global\long\def\wsl{\wl}%
 
\global\long\def\cl{{\cal \mathrm{cl}} }%
\global\long\def\ch{\mathrm{chi}}%
\global\long\def\wg{{\cal \mathrm{Wg}} }%
\global\long\def\moeb{\mathrm{M\ddot{o}b} }%
\global\long\def\F{\mathrm{\mathbf{F}} }%
\global\long\def\id{\mathrm{id}}%
\global\long\def\e{\varepsilon}%
\global\long\def\U{\mathcal{\mathrm{U}}}%
\global\long\def\Aut{\mathrm{Aut}}%
\global\long\def\E{\mathbb{E}}%
 
\global\long\def\T{{\cal T}}%
\global\long\def\tps{\left|{\cal T}\right|_{\ps}}%
\global\long\def\tips{\left|{\cal T}_{\infty}\right|_{\ps}}%
\global\long\def\ps{\mathrm{poly}}%

\global\long\def\match{\mathrm{MATCH}}%
\global\long\def\matchr{\mathrm{\overline{MATCH}}}%
\global\long\def\matchmap{\mathbf{match}}%
\global\long\def\bijs{\mathrm{bijs}}%
\global\long\def\g{\gamma}%
\global\long\def\sur{\surfaces}%
 
\global\long\def\sx{\overline{\sigma_{x}}}%
\global\long\def\os{\overrightarrow{S^{1}}}%

\global\long\def\rank{\mathrm{rank}}%
\global\long\def\chain{\mathbf{Ch}}%
\global\long\def\mod{\mathbf{mod}}%
 
\global\long\def\res{\mathrm{res}}%
\global\long\def\eq{\mathrm{eq}}%
 
\global\long\def\Stab{\mathrm{Stab}}%
\global\long\def\Cells{\mathrm{Cells}}%
\global\long\def\Cold{\mathcal{C}}%
 
\global\long\def\BIJS{\mathcal{BIJS}}%
\global\long\def\cd{\mathrm{cd}}%
\global\long\def\vcd{\mathrm{vcd}}%

\global\long\def\Z{\mathbf{\mathbf{Z}}}%
 
\global\long\def\C{\mathbf{\mathbf{C}}}%
 
\global\long\def\Q{\mathbb{\mathbb{\mathbf{Q}}}}%
\global\long\def\d{\delta}%
 
\global\long\def\G{\Gamma}%
\global\long\def\bij{\mathrm{BIJ}}%
\global\long\def\g{\gamma}%
\global\long\def\surfaces{\mathsf{Surfaces}}%
\global\long\def\sx{\overline{\sigma_{x}}}%
\global\long\def\os{\overrightarrow{S^{1}}}%
\global\long\def\B{\mathcal{B}}%
 
\global\long\def\sing{\mathrm{sing}}%
 
\global\long\def\R{\mathbf{R}}%
\global\long\def\Q{\mathbb{\mathbb{\mathbf{Q}}}}%
 
\global\long\def\image{\mathrm{image}}%
 
\global\long\def\sing{\mathrm{sing}}%
 
\global\long\def\Rep{\mathrm{Rep}}%
\global\long\def\N{\mathcal{N}}%
\global\long\def\thick{\mathbb{D}}%
\global\long\def\P{{\cal P}}%
 
\global\long\def\O{\mathrm{O}}%
\global\long\def\Sp{\mathrm{Sp}}%
 
\global\long\def\CC{\mathbf{C}}%
\global\long\def\End{\mathrm{End}}%
 
\global\long\def\G{\mathbb{G}}%
\global\long\def\trace{\mathrm{trace}}%
 
\global\long\def\sign{\mathrm{sign}}%
 
\global\long\def\k{\kappa}%
 
\global\long\def\m{\mathbf{m}}%
 
\global\long\def\kk{\mathbf{\kappa}}%
\global\long\def\I{\mathcal{I}}%
 
\global\long\def\a{\mathbf{a}}%
\global\long\def\Id{\mathrm{Id}}%
 
\global\long\def\sql{\mathrm{sql}}%
\global\long\def\GL{\mathrm{GL}}%
\global\long\def\cut{\mathrm{cut}}%
 
\global\long\def\A{\mathcal{A}}%
 
\global\long\def\sm{\mathrm{SMATCH}}%
\global\long\def\forget{\mathsf{forget}}%

\title{Matrix Group Integrals, Surfaces, and Mapping Class Groups II:\\
$\O\left(n\right)$ and $\Sp\left(n\right)$}
\author{Michael Magee and Doron Puder\thanks{D.~P.~was supported by the Israel Science Foundation (grant No.~1071/16).}}
\maketitle
\begin{abstract}
Let $w$ be a word in the free group on $r$ generators. The expected
value of the trace of the word in $r$ independent Haar elements of
$\O(n)$ gives a function $\trw^{\O}(n)$ of $n$. We show that $\trw^{\O}(n)$
has a convergent Laurent expansion at $n=\infty$ involving maps on
surfaces and $L^{2}$-Euler characteristics of mapping class groups
associated to these maps. This can be compared to known, by now classical,
results for the GUE and GOE ensembles, and is similar to previous
results concerning $\U\left(n\right)$, yet with some surprising twists.

A priori to our result, $\trw^{\O}(n)$ does not change if $w$ is
replaced with $\alpha(w)$ where $\alpha$ is an automorphism of the
free group. One main feature of the Laurent expansion we obtain is
that its coefficients respect this symmetry under $\Aut(\F_{r})$.

As corollaries of our main theorem, we obtain a quantitative estimate
on the rate of decay of $\tr_{w}^{\O}(n)$ as $n\to\infty$, we generalize
a formula of Frobenius and Schur, and we obtain a universality result
on random orthogonal matrices sampled according to words in free groups,
generalizing a theorem of Diaconis and Shahshahani.

Our results are obtained more generally for a tuple of words $\wl$,
leading to functions $\trwl^{\O}$. We also obtain all the analogous
results for the compact symplectic groups $\Sp(n)$ through a rather
mysterious duality formula.

\tableofcontents{}
\end{abstract}

\section{Introduction\label{sec:Introduction}}

Let $\F_{r}$ be the free group on $r$ generators with a fixed basis
$B=\{x_{1},\ldots,x_{r}\}$. For any word $w\in\F_{r}$ and group
$H$ there is a \emph{word map}
\[
w:H^{r}\to H
\]
defined by substitutions, for example, if $r=2$ and $w=x_{1}^{2}x_{2}^{-2}$
then $w(h_{1},h_{2})=h_{1}^{2}h_{2}^{-2}$. In this paper we consider
the case that $H$ is a compact orthogonal or symplectic group. For
$n\in\Z_{\geq1}$, the orthogonal group $\O(n)=\U\left(n,\mathbb{R}\right)$
is the group of $n\times n$ real matrices $A$ such that $A^{T}A=AA^{T}=\Id_{n}$.
For $n\in\Z_{\geq1}$ the compact symplectic group $\Sp(n)=\U\left(n,\mathbb{H}\right)$
is the group of $n\times n$ quaternion matrices $A$ such that $A^{*}A=AA^{*}=\mathrm{Id}_{n}$
(here $A^{*}$ is the adjoint matrix $\overline{A^{T}}$). However,
we use here the isomorphic description of $\Sp\left(n\right)$ as
$2n\times2n$ complex matrices which are both unitary and complex-symplectic
in the sense that they preserve a skew-symmetric form. Namely, we
use the definition,
\begin{equation}
\Sp\left(n\right)\defi\left\{ A\in M_{2n}\left(\mathbb{C}\right)\,\middle|\,A^{*}A=AA^{*}=\Id_{2n}\mathrm{~and~}A^{T}JA=J\right\} ,\label{eq:Sp definition}
\end{equation}
where $J$ is the matrix
\begin{equation}
J\defi\left(\begin{array}{cc}
0_{n} & \Id_{n}\\
-\Id_{n} & 0_{n}
\end{array}\right).\label{eq:J}
\end{equation}
The isomorphism $U\left(n,\mathbb{H}\right)\cong\U\left(2n\right)\cap\Sp\left(2n,\mathbb{C}\right)$
is given by 
\begin{equation}
A+Bj\mapsto\left(\begin{array}{cc}
A & -B\\
\overline{B} & \overline{A}
\end{array}\right)\label{eq:isom of two forms of Sp(n)}
\end{equation}
where $A,B\in M_{n}\left(\mathbb{C}\right)$ (see \cite[\S 1.2.8]{hall2015matrix}
for more details). Notice, in particular, that all matrices in $\Sp\left(n\right)$,
as defined in (\ref{eq:Sp definition}), have real trace.

We fix words $\wl\in\F_{r}$. Our main object of study is the following
type of matrix integral
\[
\trwl^{G}(n)\stackrel{\mathrm{def}}{=}\int_{G(n)^{r}}\mathrm{tr}(w_{1}(g_{1},\ldots,g_{r}))\cdots\mathrm{tr}(w_{\ell}(g_{1},\ldots,g_{r}))d\mu_{n}(g_{1})\ldots d\mu_{n}(g_{r})
\]
where $G(n)$ is either the orthogonal group $\O(n)$ or compact symplectic
group $\Sp(n)$ defined in (\ref{eq:Sp definition}), and $\mu_{n}$
is the Haar probability measure on the respective group. We write
$G=\O$ or $\Sp$ respectively to refer to the type of group being
studied.

These values, the expected value of the trace and product of traces,
are natural objects of study. First, the measure induced by a word
$w$ on $\O\left(n\right)$ or $\Sp\left(n\right)$ (see Remark \ref{rem:progress on Aners conjecture}
for the precise definition) is completely determined by these values
when $\wl$ are various (positive) powers of $w$. Second, these values
are parallel to well-studied quantities in various models of random
matrices and in free probability, and are related to questions on
representation varieties and more. See the introduction of \cite{MP2}
for more details.

Our paper is centered around writing the functions $\trwl^{G}(n)$
as Laurent series in the variable $(n+1-\alpha_{G})^{-1}$ with positive
radii of convergence\footnote{Of course, since $(n+1-\alpha_{G})^{-1}$ and $n^{-1}$ are related
by $z\mapsto\frac{z}{(z+1-\alpha_{G})}$, which fixes $0$ and is
a local biholomorphism there, Laurent series in $(n+1-\alpha_{G})^{-1}$
give rise to Laurent series in $n^{-1}$.}, where $\alpha_{\O}=2$ and $\alpha_{\Sp}=\frac{1}{2}$ are called
`Jack parameters' in the literature \cite{Novaes}. In $\S$\ref{subsec:Duality-between-}
we discuss the history of these parameters and their necessary role
in the current paper.

In fact, the functions $\trwl^{G}(n)$ are not only meromorphic at
$\infty$, but actually given by rational functions of $n$, with
integer coefficients, when $n$ is sufficiently large. This is a reasonably
straightforward consequence of the Weingarten Calculus and is proved
in Corollaries \ref{cor:orth-rational-function} and \ref{cor:symp-rational-function}
below. See Table \ref{tab:examples} for some examples.

The ability to find a Laurent series at $\infty$ for $\trwl^{G}(n)$
using diagrammatic expansions will not be surprising to experts in
Free Probability Theory and in particular, those who have worked with
the Weingarten Calculus. \emph{However, this is not the main point
of this paper.} Rather, the main point is what we now explain. The
integrals $\trwl^{G}(n)$ have a priori symmetries under $\Aut(\F_{r})$,
the automorphism group of $\F_{r}$, according to the following lemma.
\begin{lem}
If $\alpha$ is an automorphism of $\F_{r}$, then $\trwl^{G}(n)=\tr_{\alpha(w_{1}),\ldots,\alpha(w_{\ell})}^{G}(n)$.
\end{lem}

This is proved in the case $\ell=1$ in \cite[\S 2.2]{MP} and the
case of general $\ell$ can be proved using the same argument. In
particular, if $w\in\F_{r}$ then $\trw^{G}(n)$ is a function of
$n$ that only depends on $w$ up to automorphism; in other words,
it reflects algebraic properties of $w$, not just combinatorial properties.
So a priori, any Laurent series expansion of $\tr_{\wl}^{G}(n)$ will
involve $\Aut(\F_{r})$-invariants of $\wl$ and this brings us to
the true goal of the paper:\emph{ to find Laurent series expansions
at $\infty$ for $\tr_{\wl}^{G}(n)$ in terms of $\Aut(\F_{r})$-invariants
of $\wl$}. Indeed, our main result, Theorem \ref{thm:main-theorem},
as well as its corollaries in Section \ref{subsec:Consequences-of-the},
are all given in terms of \emph{$\Aut(\F_{r})$}-invariants of the
words.

In fact, the expressions for $G=\O$ or $\Sp$ are closely related;
we will prove
\begin{thm}
\label{thm:orth-symp-relation}There is\footnote{We define $N$ precisely in (\ref{eq:N0}).}
$N=N(\wl)\geq0$ such that when $2n\geq N$ we have
\[
\trwl^{\Sp}(n)=(-1)^{\ell}\trwl^{\O}(-2n).
\]
The quantity $\trwl^{\O}(-2n)$ is interpreted using the rational
function form of $\trwl^{\O}$ for $2n\geq N$. In particular, this
identity relates the Laurent series at $\infty$ of $\trwl^{\Sp}(n)$
and $\trwl^{\O}(n)$.
\end{thm}

This type of duality between $\O$ and $\Sp$ has been observed previously
in various contexts, some of which we discuss in $\S$\ref{subsec:Duality-between-}.

\renewcommand{\arraystretch}{2}
\begin{center}
\begin{table*}[t]
\begin{centering}
\begin{tabular}{|c|>{\centering}p{3cm}|c|c|>{\raggedright}p{5.7cm}|}
\hline 
\noalign{\vskip\doublerulesep}
$\ell$ & $\wl$ & $\tr_{\wl}^{\O}\left(n\right)$ & $\chi_{\max}$ & Admissible maps with $\chi\left(\Sigma\right)=\chi_{\max}$\tabularnewline[\doublerulesep]
\hline 
\hline 
\noalign{\vskip\doublerulesep}
\multirow{6}{*}{1} & $x^{2}y^{2}$ & {\Large{}$\frac{1}{n}$} & $-1$ & one $\left(P_{2,1},f\right)$ with $\mcg\left(f\right)=\left\{ 1\right\} $\tabularnewline[\doublerulesep]
\cline{2-5} \cline{3-5} \cline{4-5} \cline{5-5} 
\noalign{\vskip\doublerulesep}
 & $x^{4}y^{4}$ & {\Large{}$\frac{1}{n}$} & $-1$ & one $\left(P_{2,1},f\right)$ with $\mcg\left(f\right)=\left\{ 1\right\} $\tabularnewline[\doublerulesep]
\cline{2-5} \cline{3-5} \cline{4-5} \cline{5-5} 
\noalign{\vskip\doublerulesep}
 & $[x,y]^{2}$ & {\Large{}$\frac{n^{3}+n^{2}-2n-4}{n(n+2)(n-1)}$} & $0$ & one $\left(P_{1,1},f\right)$ with $\mcg\left(f\right)=\left\{ 1\right\} $\tabularnewline[\doublerulesep]
\cline{2-5} \cline{3-5} \cline{4-5} \cline{5-5} 
\noalign{\vskip\doublerulesep}
 & $xy^{3}x^{-1}y^{-1}$  & {\Large{}$0$} & $-2$ & one $\left(P_{3,1},f\right)$ with $\mcg\left(f\right)\cong\Z$\tabularnewline[\doublerulesep]
\cline{2-5} \cline{3-5} \cline{4-5} \cline{5-5} 
\noalign{\vskip\doublerulesep}
 & $xy^{4}x^{-1}y^{-2}$ & {\Large{}$\frac{1}{n}$} & $-1$ & one $\left(P_{2,1},f\right)$ with $\mcg\left(f\right)=\left\{ 1\right\} $\tabularnewline[\doublerulesep]
\cline{2-5} \cline{3-5} \cline{4-5} \cline{5-5} 
\noalign{\vskip\doublerulesep}
 & $xyx^{2}yx^{3}y^{2}$ & {\Large{}$\frac{3n+2}{n(n+2)(n-1)}$} & $-2$ & three $\left(P_{3,1},f\right)$ with $\mcg\left(f\right)=\left\{ 1\right\} $\tabularnewline[\doublerulesep]
\hline 
\noalign{\vskip\doublerulesep}
\multirow{2}{*}{2} & $w,w$ for $w=x^{2}y^{2}$ & {\Large{}$\frac{n^{3}+n^{2}+2n+4}{n(n+2)(n-1)}$} & $0$ & one annulus with $\mcg\left(f\right)=\left\{ 1\right\} $\tabularnewline[\doublerulesep]
\cline{2-5} \cline{3-5} \cline{4-5} \cline{5-5} 
\noalign{\vskip\doublerulesep}
 & $w,w$ for $w=x^{2}y$ & {\Large{}$1$} & $0$ & one annulus with $\mcg\left(f\right)=\left\{ 1\right\} $\tabularnewline[\doublerulesep]
\hline 
\noalign{\vskip\doublerulesep}
3 & $w,w,w$ for $w=x^{2}y^{2}$ & {\Large{}$\frac{3(n^{4}+3n^{3}-2n^{2}+6n+16)}{(n-2)(n-1)n(n+2)(n+4)}$} & $-1$ & three (annulus $\sqcup~P_{2,1}$) with $\mcg\left(f\right)=\left\{ 1\right\} $\tabularnewline[\doublerulesep]
\hline 
\end{tabular}
\par\end{centering}
\caption{Some examples of the rational expression for $\protect\tr_{\protect\wl}^{\protect\O}\left(n\right)$.
All these examples contain words in $\protect\F_{2}$ with generators
$\left\{ x,y\right\} $. The notation $\left[x,y\right]$ is for the
commutator $xyx^{-1}y^{-1}$. We let $\chi_{\max}=\chi_{\max}\left(\protect\wl\right)$
denote the maximal Euler characteristic of a surface in $\protect\sur^{*}\left(\protect\wl\right)$
-- see Page \pageref{chi_max}. See Definitions \ref{def:admissible-maps-and-surfaces}
and \ref{def:MCG f} for the notions of admissible maps and $\protect\mcg\left(f\right)$
appearing in the right column. In that column, $P_{g,b}$ denotes
the non-orientable surface of genus $g$ with $b$ boundary components
(so $\chi\left(P_{g,b}\right)=2-g-b$). We give more details in Section
\ref{subsec:More-examples} and Table \ref{tab:examples-elaborated}.
\label{tab:examples}}
\end{table*}
\par\end{center}

\vspace{-30bp}

\subsection{The main theorem}

Understanding the functions $\trwl^{G}(n)$ involves studying maps
from surfaces to the wedge of $r$ circles, denoted by $\wedger$,
as was also the case in \cite{MP2}, where the corresponding theorems
were proved for the functions $\trwl^{U}(n)$ that arise from compact
unitary groups $\U\left(n\right)=\U\left(n,\mathbb{C}\right)$. One
central difference that appears here is that for unitary groups, only
orientable surfaces featured in the description of $\trwl^{U}(n)$,
whereas the description of $\trwl^{G}(n)$ for $G=\O$ or $\Sp$ involves
non-orientable surfaces as well.

We call the point in $\wedger$ at which the circles are wedged together
$o$. We fix an orientation of each circle that gives us an identification
$\pi_{1}(\wedger,o)\cong\F_{r}$ and we identify the generator $x\in B$
with the loop that traverses the circle $S_{x}^{1}$ corresponding
to $x$ according to its given orientation.
\begin{defn}[Admissible maps]
\label{def:admissible-maps-and-surfaces}Given $\wl\in\F_{r}$, consider
pairs $(\Sigma,f)$ where
\begin{itemize}
\item $\Sigma$ is a compact surface, not necessarily connected, with $\ell$
ordered and oriented boundary components $\delta_{1},\ldots,\delta_{\ell}$
and a marked point $v_{j}$ on each $\delta_{j}$. Note the orientation
of the $j$th boundary component specifies a generator $[\delta_{j}]$
of $\pi_{1}(\delta_{j},v_{j})$. We also require that $\Sigma$ has
no closed components\footnote{So every connected component of a surface $\Sigma$ in an admissible
map is either the orientable $\Sigma_{g,b}$ of genus $g\ge0$ and
with $b\ge1$ boundary components, or the non-orientable $P_{g,b}$
of genus $g\ge1$ and with $b\ge1$ boundary components. Recall that
the Euler characteristics of these surfaces are $\chi\left(\Sigma_{g,b}\right)=2-2g-b$
and $\chi\left(P_{g,b}\right)=2-g-b$.}.
\item $f:\Sigma\to\wedger$ is a continuous map such that $f(v_{j})=o$
and $f_{*}([\delta_{j}])=w_{j}\in\pi_{1}(\wedger,o)\cong\F_{r}$ for
all $1\leq j\leq\ell$.
\end{itemize}
We call such a pair an \emph{admissible map} (for $\wl)$. Consider
two admissible maps $(\Sigma,f)$ and $(\Sigma',f')$ with boundary
components $\delta_{1},\ldots,\delta_{\ell}$ and $\delta'_{1},\ldots,\delta'_{\ell}$
and marked points $\{v_{j}\}$ and $\{v'_{j}\}$ respectively. We
say $(\Sigma,f)$ and $(\Sigma',f')$ are \emph{equivalent} and write
$(\Sigma,f)\approx(\Sigma',f')$ if 
\begin{itemize}
\item there exists a homeomorphism $F:\Sigma\to\Sigma'$ such that $F(v_{j})=v_{j}'$
and $F_{*}([\delta_{j}])=[\delta'_{j}]$ for all $1\leq j\leq\ell$.
Here $F_{*}$ are the maps induced on the relevant fundamental groups.
This condition says that $F$ preserves the orientation of each boundary
component. And,
\item there is a homotopy between $f'\circ F$ and $f$ on $\Sigma$. This
homotopy is relative to the points $\{v_{j}\}$.
\end{itemize}
The set $\surfaces^{*}(\wl)$ \marginpar{$\protect\surfaces^{*}$}is
defined to be the resulting collection of equivalence classes $[(\Sigma,f)]$
of admissible maps for $\wl$.
\end{defn}

\begin{figure}[t]
\begin{centering}
\includegraphics[scale=1.1]{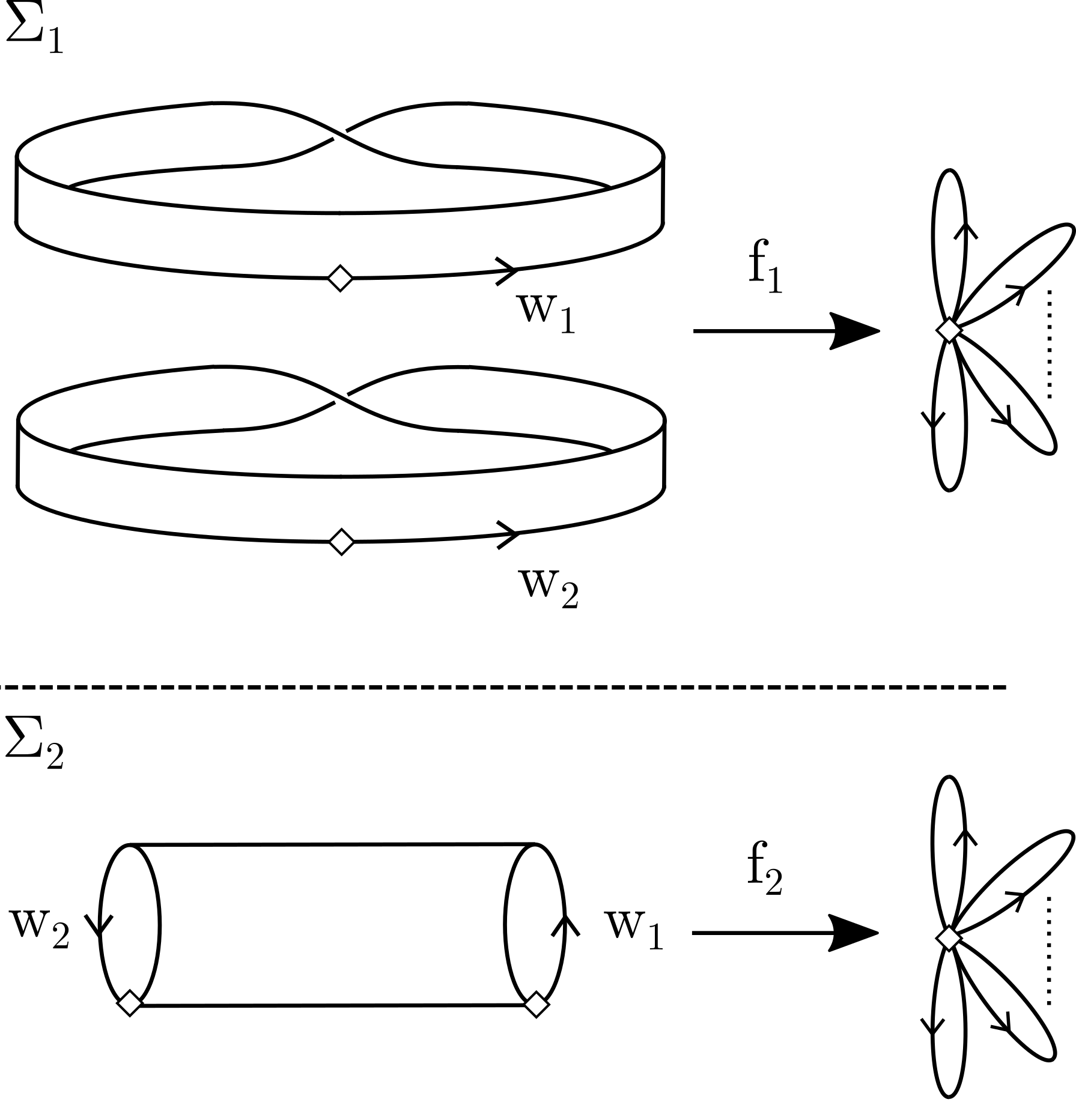}\-\-
\par\end{centering}
\caption{\label{fig:Schematic-illustration}Schematic illustration of two admissible
maps $(\Sigma_{1},f_{1})$ and $(\Sigma_{2},f_{2})$ when $\ell=2$
and there are two words $w_{1},w_{2}$. Base points $v_{j}$ are marked
by diamonds, and an oriented boundary component labeled by $w_{j}$
maps to $w_{j}$ under $(f_{i})_{*}$, for the relevant $i=1,2$.}

\end{figure}

Figure \ref{fig:Schematic-illustration} illustrates the concept of
admissible maps. Definition \ref{def:admissible-maps-and-surfaces}
modifies the definition of $\surfaces(\wl)$ given in \cite[Def.~1.3]{MP2},
by dropping the requirement that $\Sigma$ is orientable, and any
compatibility between the orientations of boundary components of $\Sigma$.
Hence $\surfaces(\wl)\subset\surfaces^{*}(\wl)$. The set $\surfaces^{*}(\wl)$
will index the summation in our Laurent series expansion of $\trwl^{O}(n)$,
but to understand the contribution of a given $[(\Sigma,f)]$ we must
take into account the internal symmetries of this pair.

If $\Sigma$ is a compact surface, the mapping class group of $\Sigma$,
denoted $\MCG(\Sigma)$, is the collection of homeomorphisms of $\Sigma$
that fix the boundary pointwise, modulo homeomorphisms that are isotopic
to the identity through homeomorphisms of this type. Note that for
$\Sigma$ fixed, $\MCG(\Sigma)$ acts on the collection of homotopy
classes $[f]$ of $f$ such that $[(\Sigma,f)]\in\surfaces^{*}(\wl)$.
To take into account the function $f$ in our definition of symmetries,
we make the following definition.
\begin{defn}
\label{def:MCG f}For $[(\Sigma,f)]\in\surfaces^{*}(w_{1},\ldots,w_{\ell})$,
we define \marginpar{$\protect\MCG(f)$} $\MCG(f)$ to be the stabilizer
of $[f]$ in $\MCG(\Sigma)$. This is a well-defined subgroup of $\mcg\left(\Sigma\right)$
up to conjugation, and so a well-defined group up to isomorphism.
\end{defn}

A certain integer invariant of $\MCG(f)$, where $[(\Sigma,f)]\in\surfaces^{*}(\wl)$,
appears in our formula for $\trwl^{G}(n)$. The invariant that appears
is the \emph{$L^{2}$-Euler characteristic, }denoted by $\chi^{(2)}(\MCG(f))$\emph{,}
and defined precisely in \emph{$\S$}\ref{subsec:L2invariants}\emph{.}
This is defined for a class of groups that we will prove in $\S$\ref{subsec:Proof-of-Theorem-MAIN}
contains all $\MCG(f)$ where $[(\Sigma,f)]\in\surfaces^{*}(\wl)$.
Moreover, for certain $[(\Sigma,f)]$, including all those contributing
to the `leading and second-leading order' terms of the Laurent series
expansion of $\trwl^{G}(n)$, we will prove in $\S$\ref{subsec:The--Euler-characteristic-is the usual one}
that $\chi^{(2)}(\MCG(f))$ coincides with a much tamer invariant:
the Euler characteristic of a finite $CW$-complex that is an Eilenberg-Maclane
space of type $K(\MCG(f),1)$. We make this point precise in Theorem
\ref{thm:tame-euler-char}. Some examples are given in Table \ref{tab:examples}.

Another simpler type of Euler characteristic also appears in our formula.
If $[(\Sigma,f)]$ is in\linebreak{}
$\surfaces^{*}(\wl)$ we write $\chi(\Sigma)$ for the usual topological
Euler characteristic of $\Sigma$. Clearly this does not depend on
the representative chosen for $[(\Sigma,f)]$.

We can now state our main theorem.
\begin{thm}
\label{thm:main-theorem}There is\footnote{See (\ref{eq:Mdef}) for the precise definition of $M$.}
$M=M(\wl)\geq0$ such that for $n>M$, $\trwl^{\O}(n)$ is given by
the following absolutely convergent Laurent series in $(n-1)^{-1}$:

\begin{equation}
\trwl^{\O}(n)=\sum_{[(\Sigma,f)]\in\surfaces^{*}(\wl)}(n-1)^{\chi(\Sigma)}\chi^{(2)}(\MCG(f)).\label{eq:main_thm}
\end{equation}
This also gives a Laurent series expansion in $(n+\frac{1}{2})^{-1}$
for $\trwl^{\Sp}(n)$ in view of Theorem \ref{thm:orth-symp-relation}.
\end{thm}

\begin{rem}
We have $\chi\left(\Sigma\right)\le\ell$ for $[(\Sigma,f)]\in\surfaces^{*}(\wl)$
(see Lemma \ref{lem:chi-max-non-pos} and its preceding paragraph).
Hence the powers of $n-1$ that appear in (\ref{eq:main_thm}) are
bounded above.
\end{rem}

\begin{rem}
In the course of the proof of Theorem \ref{thm:main-theorem} we prove
that for each fixed $\chi_{0}$, there are only finitely many $[(\Sigma,f)]\in\surfaces^{*}(\wl)$
such that $\chi(\Sigma)=\chi_{0}$ and $\chi^{(2)}(\MCG(f))\neq0$.
We also prove that each $\chi^{(2)}(\MCG(f))\in\Z$, so the Laurent
series has integer coefficients.
\end{rem}

\begin{rem}
Although Theorem \ref{thm:main-theorem} holds if some of the $w_{i}=1$,
it simplifies the paper to assume that all $w_{i}\neq1$, which we
do from now on.
\end{rem}

\begin{rem}
As claimed above, the statement of Theorem \ref{thm:main-theorem}
indeed gives an expansion of $\trwl^{\O}\left(n\right)$ in terms
of $\mathrm{Aut}\left(\F_{r}\right)$-invariants of the words. Indeed,
if $\alpha\in\mathrm{Aut}\left(\F_{r}\right)$, there is a corresponding
map $g_{\alpha}\colon\left(\wedger,o\right)\to\left(\wedger,o\right)$
such that $\left(g_{\alpha}\right)_{*}=\alpha$, and then every $\left[\left(\Sigma,f\right)\right]\in\surfaces^{*}\left(\wl\right)$
corresponds to $\left[\left(\Sigma,g_{\alpha}\circ f\right)\right]\in\surfaces^{*}\left(\alpha\left(w_{1}\right),\ldots,\alpha\left(w_{\ell}\right)\right)$,
satisfying $\mcg\left(f\right)\cong\mcg\left(g_{a}\circ f\right)$.
For a more detailed argument, see \cite[Lem.~4.4]{brodsky2022word}.
\end{rem}

Theorem \ref{thm:main-theorem} can be viewed as analogous to the
genus expansions of GOE matrix integrals in terms of Euler characteristics
of mapping class groups obtained by Goulden-Jackson \cite{GJ} and
Goulden, Harer, and Jackson \cite{GHJ}. These results extended previous
results of Harer and Zagier \cite{HARERZAGIER} on the GUE ensemble.

The connection between $\O(n)$ matrix integrals and not-necessarily-oriented
surfaces was previously pointed out by Mingo and Popa \cite{MingoPopa}
and Redelmeier \cite{Redelmeier}.

\subsection{Consequences of the main theorem\label{subsec:Consequences-of-the}}

Theorem \ref{thm:main-theorem} has several important corollaries.
Firstly, for a fixed word $w\in\F_{r}$ the rate of decay of $\trw^{G}(n)$
as $n\to\infty$ can be bounded in terms of the \emph{square length
of w }and the \emph{commutator length of $w$.}
\begin{defn}
The \emph{square length of $w\in\F_{r}$, }denoted $\sql(w)$, is
the minimum number $s$ such that $w$ can be written as the product
of $s$ squares, or $\infty$ if it is not possible to write $w$
as the product of squares. The \emph{commutator length of $w$, }denoted
$\cl(w)$, is the minimum number $g$ such that $w$ can be written
as the product of $g$ commutators, or $\infty$ if $w$ cannot be
written as the product of commutators.
\end{defn}

One has the elementary inequality
\begin{equation}
\sql(w)\leq2\cl(w)+1.\label{eq:sql-cl-ineq}
\end{equation}
This follows from the identities 
\[
[a,b]=(ab)^{2}(b^{-1}a^{-1}b)^{2}(b^{-1})^{2},\quad a^{2}[b,c]=(a^{2}ba^{-1})^{2}(ab^{-1}a^{-1}ca^{-1})^{2}(ac^{-1})^{2}
\]
for any $a,b,c\in\F_{r}$. In particular, if $\sql(w)=\infty$ then
$\cl(w)=\infty$.
\begin{cor}
\label{cor:decay-rate}For $G=\O$ or $\Sp$, we have 
\[
\trw^{G}(n)=O\left(n^{1-\min(\sql(w),2\cl(w))}\right)
\]
as $n\to\infty$. We interpret the right hand side as $0$ when $\sql(w)=\infty$.
\end{cor}

\begin{rem}
The inequality (\ref{eq:sql-cl-ineq}) implies that $\trw^{G}(n)=O(n^{1-\sql(w)})$
unless (\ref{eq:sql-cl-ineq}) is an equality. In that case, $\sql(w)=2\cl(w)+1$
and $\trw^{G}(n)=O(n^{2-\sql(w)})$. An analogous result holds for
$G=\U$ the unitary group: $\trw^{\U}\left(n\right)=O\left(n^{1-2\cl\left(w\right)}\right)$
\cite[Cor.~1.8]{MP2}. Yet another analogous result, albeit with quite
a different flavor, holds in the case of the symmetric group. Let
$\tr_{w}^{\mathrm{Sym}}\left(n\right)$ denote the trace of the $\left(n-1\right)$-dimensional
irreducible ``standard'' representation of $\mathrm{Sym}\left(n\right)$.
Then $\tr_{w}^{\mathrm{Sym}}\left(n\right)=C\cdot n^{1-\pi\left(w\right)}+O\left(n^{-\pi\left(w\right)}\right)$,
where $\pi\left(w\right)$ is the smallest rank of a subgroup of $\F_{r}$
containing $w$ as a non-free-generator and $C=C\left(w\right)$ is
some positive integer \cite[Thm~1.8]{PP15}. See also \cite[Thm.~1.11]{MP3}
for similar results in the case of generalized symmetric groups.
\end{rem}

\begin{rem}
Although we have stated Corollary \ref{cor:decay-rate} as a corollary
of Theorem \ref{thm:main-theorem}, the proof is significantly simpler
and does not require any discussion of $L^{2}$-invariants. We explain
the proof of Corollary \ref{cor:decay-rate} in $\S$\ref{subsec:Proof-of-Corollary-decay}
below.
\end{rem}

\begin{rem}
\label{rem:coro-remark}The consequence of Corollary \ref{cor:decay-rate}
that $\trw^{G}(n)=0$ when $w$ cannot be written as the product of
squares can also be proved directly from the definition of $\trw^{G}(n)$
-- see Lemmas \ref{lem:integral-is-zero} and \ref{lem:when is surfaces nonempty}.
\end{rem}

\begin{rem}
When $w=x_{1}^{2}x_{2}^{2}\ldots x_{s}^{2}\in\F_{r}$, $r\geq s$,
a result of Frobenius and Schur \cite{frobenius1906reellen} (for
$s=1$) and a straightforward generalization (see \cite[\S 2]{MP3}
or \cite[Prop.~3.1(3)]{parzanchevski2014fourier}) give
\begin{equation}
\trw^{\O}(n)=\frac{1}{n^{s-1}},\quad\trw^{\Sp}(n)=\frac{(-1)^{s}}{(2n)^{s-1}}.\label{eq:frobenius}
\end{equation}
In fact, analogs of these formulas hold for any compact group $G$.
Thus Theorem \ref{thm:main-theorem} and Corollary \ref{cor:decay-rate}
can be viewed as a generalization of these formulas to arbitrary words
in $\F_{r}$. However, this generalization is not as simple as it
might appear on the surface. For example, when $w=x_{1}^{2}x_{2}^{2}$,
combining Theorem \ref{thm:main-theorem} with (\ref{eq:frobenius})
gives for large $n$
\[
n^{-1}=\sum_{[(\Sigma,f)]\in\surfaces^{*}(x_{1}^{2}x_{2}^{2})}(n-1)^{\chi(\Sigma)}\chi^{(2)}(\MCG(f)).
\]
Since $n^{-1}$ does not agree with any Laurent polynomial of $n-1$
for large $n$, this implies that there are infinitely many different
$\chi_{0}$ such that there is $[(\Sigma,f)]$ in $\surfaces^{*}(x_{1}^{2}x_{2}^{2})$
with $\chi(\Sigma)=\chi_{0}$ and $\chi^{(2)}(\MCG(f))\neq0$. We
analyze these $[(\Sigma,f)]$ in Example \ref{sub:frob-schur-example}.
\end{rem}

\begin{rem}
\label{rem:progress on Aners conjecture}In another direction, in
\cite{MP3} we show using in part Corollary \ref{cor:decay-rate}
that the word $x_{1}^{2}\ldots x_{s}^{2}$ is uniquely determined,
up to automorphisms, by the `word measures' induced by the word on
compact groups, that is, the pushforwards of Haar measures on $G^{r}$
under the word map $w$.
\end{rem}

\begin{cor}
\label{cor:limit-counting-formula}If all $\wl$ are not equal to
$1$, the limit $\lim_{n\to\infty}\trwl^{\O}(n)$ exists, and is an
integer that counts the number of pairs $[(\Sigma,f)]\in\surfaces^{*}(\wl)$
such that all the connected components of $\Sigma$ are annuli or
Möbius bands.

In algebraic terms, this integer is the weighted number of ways to
partition $\wl$ into singletons and pairs, such that the word in
every singleton is a square, and every pair $\{w,w'\}$ has the property
that $w'$ is conjugate to either $w$ or $w^{-1}$. The weight given
to such a partition is 
\[
\prod_{\{w,w'\}}\max\{d\geq1\::\exists u~\mathrm{s.t.}\:w=u^{d}\}
\]
where $\{w,w'\}$ run over the pairs of the partition.
\end{cor}

We prove Corollary \ref{cor:limit-counting-formula} in $\S$\ref{subsec:Proof-of-Corollary-limit-counting}.
It implies, in turn, the following corollary:
\begin{cor}
\label{cor:power-limits}Suppose that $w\neq1$ and $w=u^{d}$ with
$d\geq1$ such that $u\in\F_{r}$ is not a proper power of any other
element of $\F_{r}$. Then for all $\ell\geq1$ and $j_{1},\ldots,j_{\ell}\in\Z$,
the limit
\[
\lim_{n\to\infty}\tr_{w^{j_{1}},\ldots,w^{j_{\ell}}}^{\O}(n)
\]
 exists and only depends on $d$ and the $j_{k}$, not on $u$. Moreover,
this collection of limits determines $d$. The same result holds with
$\O$ replaced by $\Sp$.
\end{cor}

The phenomenon observed in Corollary \ref{cor:power-limits} is also
known to be present for other families of groups including unitary
groups $U(n)$ \cite{MSS07,Radulescu06}, symmetric groups $S_{n}$
\cite{Nica94} (complemented in \cite{Linial2010,hanany2022word}),
and $\GL_{n}(\mathbb{F}_{q})$ \cite{PW}. See also \cite[Thm.~1.3]{dubach2021words}
for a similar phenomenon for Ginibre ensembles. We will prove Corollary
\ref{cor:power-limits} in $\S$\ref{subsec:Proof-of-Corollary-limit-counting}.

In the case $r=1$, so that there is only one letter $x=x_{1}$, and
$w_{k}=x^{j_{k}}$ for $k=1,\ldots,\ell,$ we have 
\[
\tr_{x^{j_{1}},\ldots,x^{j_{\ell}}}^{\O}(n)=\int_{\O(n)}\mathrm{tr}\left(g^{j_{1}}\right)\cdots\mathrm{tr}\left(g^{j_{\ell}}\right)d\mu_{n}(g).
\]
Diaconis and Shahshahani \cite[Thm.~4]{Diaconis1994} prove that in
this case, $\tr_{x^{j_{1}},\ldots,x^{j_{\ell}}}^{\O}(n)$ is \emph{exactly}
the integer described in Corollary \ref{cor:limit-counting-formula}
for any $n$ sufficiently large, and give a closed formula for this
integer. Moreover in\emph{ (ibid.) }Diaconis and Shahshahani use this
fact to prove that for $j\in\mathbb{N}$, the collection of random
variables $\mathrm{tr}(g),\mathrm{tr}\left(g^{2}\right)\ldots,\mathrm{tr}\left(g^{j}\right)$,
where $g$ is chosen according to Haar measure on $\O(n)$, converge
in probability as $n\to\infty$ to independent normal variables with
different centers and variances. Using the method of moments, Corollary
\ref{cor:power-limits} implies that one has the same result if $x_{1}$
is replaced by any non-trivial word $w$ that is not a proper power.
More precisely, we have the following result.
\begin{cor}[Universality for traces of non-powers]
\label{cor:moment-convergence}Let $w\in\F_{r}$, $w\neq1$, and
$w$ not a proper power of another element in $\F_{r}$. For fixed
$\ell\geq1$, consider the real-valued random variables $T_{n}\left(w^{j}\right)\stackrel{\mathrm{def}}{=}\mathrm{tr}\left(w\left(g_{1},\ldots,g_{r}\right){}^{j}\right)$
on the probability space $(\O(n)^{r},\mu_{n}^{r})$ for $j=1,\ldots,\ell$.
The $T_{n}\left(w^{j}\right)$ converge in probability as $n\to\infty$:
\[
\left(T_{n}(w),T_{n}\left(w^{2}\right),\ldots,T_{n}\left(w^{\ell}\right)\right)\xrightarrow[\text{probability}]{n\to\infty}\left(Z_{1},Z_{2},\ldots,Z_{\ell}\right),
\]
where $Z_{1},\ldots,Z_{\ell}$ are independent real normal random
variables, such that when $j$ is odd, $Z_{j}$ has mean 0 and variance
$j$, and when $j$ is even, $Z_{j}$ has mean $1$ and variance $j+1$.
\end{cor}

Some related results were previously obtained by Mingo and Popa in
\cite{MingoPopa} where they obtain `real second order freeness' of
independent Haar elements of $\O(n)$. This concept does not seem
to imply the explicit statement of Corollary \ref{cor:moment-convergence}.

\subsection{Duality between $\protect\Sp$ and $\protect\O$, and the parameters
$\alpha_{G}$\label{subsec:Duality-between-}}

\subsubsection*{Duality between $\protect\Sp$ and $\protect\O$ }

The formula that appears in Theorem \ref{thm:orth-symp-relation}
was suggested by Deligne \cite{DELIGNE} in a private communication
to the second named author of this paper, without the explicit calculation
of the sign $(-1)^{\ell}$. Deligne's reasoning is that one has the
identities of `supergroups'
\[
\O(-2n)=\O(0|2n)=\Sp(n).
\]
We do not know how to make this into a rigorous concise proof of Theorem
\ref{thm:orth-symp-relation} at the moment. The proof we give in
the Appendix relies on a technical combinatorial comparison of the
terms arising in $\tr_{\wl}^{\O}$ and $\tr_{\wl}^{\Sp}$ from the
Weingarten Calculus. 

Formulas similar to Theorem \ref{thm:orth-symp-relation} were observed
by Mkrtchyan \cite{MKRTCHYAN1981} in the early 1980s in the setting
of $\O$ vs $\Sp$ gauge theory. The duality also shows up as a duality
between the GOE and GSE ensembles \cite{MW}. The introduction to
\emph{(ibid.)} also contains an overview of what was known about the
$\O$-$\Sp$ duality at that time. The $\O(-2n)=\Sp(n)$ formula has
more recently been interpreted in a different way, in terms of Casimir
operators, by Mkrtchyan and Veselov in \cite{MV}.

\subsubsection*{The parameters $\alpha_{G}$}

In this section we mention other occurrences of the parameters $\alpha_{G}$
in random matrix theory. We have observed in \cite{MP2} and the current
paper that it is most natural to expand $\tr_{\wl}^{G}(n)$ as a Laurent
series in $(n+1-\alpha_{G})^{-1}$, where $\alpha_{G}=2,1,\frac{1}{2}$
for $G=\O,\U,\Sp$ respectively. It is more common in the literature
that the parameter
\[
\beta_{G}\stackrel{\mathrm{def}}{=}\frac{2}{\alpha_{G}}
\]
appears. Hence $\beta_{G}=1,2,4$ for $G=\O,\U,\Sp$ respectively.
One sees that $1,2,4$ correspond to the dimensions of the real, complex,
and quaternion numbers as real vector spaces, and this is surely the
fundamental source of the parameter $\beta_{G}$. 

One historical role that these parameters play in random matrix theory
is that they give unified expressions for the joint eigenvalue distributions
of Dyson's Orthogonal, Unitary, and Symplectic Circular Ensembles\footnote{The COE of dimension $n$ is the space of symmetric unitary matrices.
This can be identified with $\U(n)/\O(n)$ and as such, has a natural
probability measure coming from Haar measure on $\U(n)$. The CUE
of dimension $n$ is $\U(n)$ with its Haar measure. The CSE of dimension
$2n$ is the space of self-dual unitary matrices, that can be identified
with $\U(2n)/\Sp(n)$ and hence given the probability measure coming
from Haar measure on $\U(2n)$.} (COE/CUE/CSE) introduced by Dyson in \cite{DYSONI}. Results of Weyl
\cite{Weyl} (for the CUE) and Dyson \cite[Thm.~8]{DYSONI} (for the
COE and CSE) say that the joint eigenvalue density of a matrix in
one of these ensembles is given by
\[
C_{\beta}\prod_{k\neq l}\left|\exp(i\theta_{k})-\exp(i\theta_{l})\right|^{\beta}
\]
where $\exp(i\theta_{k})$ are the eigenvalues of the random matrix,
$C_{\beta}$ is a normalizing constant, and $\beta=1$ for COE, $\beta=2$
for CUE, and $\beta=4$ for CSE. 

More recently, and in a context more closely related to the current
paper, the parameters $\alpha_{G}$ appear in the work of Novaes \cite{Novaes}
who calculates Laurent series in $(n+1-\alpha_{G})^{-1}$ for Weingarten
functions on $G(n)$ with $G=\U,\O,\Sp$. The origin of the shifted
parameter in (\emph{loc.~cit.}) is its use of the following result
of Forrester \cite[eq.~3.10]{Forrester}: if $B$ is sampled from
$\U(n),\O(n),$ or $\Sp(n)$ according to Haar measure, then the density
function of the top left $m_{1}\times m_{2}$ submatrix $A$ of $g$
is given by a determinant involving $A$ raised to an exponent that
is an explicit function of $\alpha_{G}$. This is very different to
the appearance of $\alpha_{G}$ in the current paper.

Indeed, the origin of $\alpha_{G}$ in the current paper is \emph{topological}
and based on the following observations:
\begin{itemize}
\item There are exactly two types of connected surfaces with boundary that
have trivial mapping class group: a disc, and a Möbius band (cf.~Lemma
\ref{lem:The-mapping-class-group-of-disc-or-mobius} and Proposition
\ref{prop:dehn-twists-generate-free-abelian}).
\item As a result, when we compute the terms $\chi^{(2)}(\MCG(f))$ that
appear in Theorem \ref{thm:main-theorem}, we need to enumerate surfaces
that are formed by gluing together discs \textbf{and}\emph{ }Möbius
bands.
\item On the other hand, using the Weingarten calculus to expand $\tr_{\wl}^{\O}(n)$
leads to a formula that involves enumerating surfaces that are formed
only by gluing discs (i.e., given as \emph{CW}-complexes). This formula
is given in Proposition \ref{prop:First-Laurent-Expansion}.
\end{itemize}
Hence, one must at some stage pass from a formula involving surfaces
given as \emph{CW}-complexes to a formula involving surfaces formed
by gluing together discs and Möbius bands. This is somewhat surprisingly
accomplished simply by replacing the parameter $n$ by $n-1$ in the
case $G=\O$, and is given by Proposition \ref{prop:Second-Laurent-Expansion}. 

We suspect that all these appearances of the parameter $\alpha_{G}$
(or $\beta_{G}$) in different results in random matrix theory are
all related, but we do not know how to directly explain this relation.

\subsection{Notation and paper organization}

For $n\in\mathbf{N},$ we use the notation $[n]$ for the set $\{1,2,\ldots,n\}$.
If $f$ is a map between topological spaces with base points, then
$f_{*}$ is the induced map between the fundamental groups of the
spaces. We write $\emptyset$ for the empty set. If $w\in\F_{r}$
we write $|w|$ for the word length of $w$ in reduced form. We write
$\log$ for the natural logarithm (base $e$).

The paper is organized as follows. Section \ref{sec:Maps-on-surfaces}
describes how one can construct admissible maps in $\sur^{*}\left(\wl\right)$
from sets of matchings of the letters of $\wl$, and gives some basic
definitions and facts about general elements of $\sur^{*}\left(\wl\right)$.
In Section \ref{sec:A-combinatorial-formula} we discuss the Weingarten
calculus, give a combinatorial formula for $\trwl^{\O}\left(n\right)$
(Theorem \ref{thm:orth-wg-exp}), derive two different Laurent series
expansions of $\trwl^{\O}\left(n\right)$ (Propositions \ref{prop:First-Laurent-Expansion}
and \ref{prop:Second-Laurent-Expansion}), and reduce our main theorem,
Theorem \ref{thm:main-theorem}, to a theorem about a single admissible
map in $\sur^{*}\left(\wl\right)$ (Theorem \ref{thm:Formula-for-l2-euler-char}).
Section \ref{sec:The-transverse-map} introduces the complex of transverse
maps associated with some $\left[\left(\Sigma,f\right)\right]\in\sur^{*}\left(\wl\right)$
and proves it is contractible, following closely with the analogous
result in \cite{MP2}. In Section \ref{sec:The-action-of} we finish
the proof of Theorems \ref{thm:Formula-for-l2-euler-char} and \ref{thm:main-theorem},
and in Section \ref{sec:Remaining-proofs-and} we prove Corollaries
\ref{cor:limit-counting-formula}, \ref{cor:power-limits} and \ref{cor:moment-convergence},
and discuss some concrete examples. Finally, Appendix \ref{sec:symplectic formula}
gives a combinatorial proof of Theorem \ref{thm:orth-symp-relation}.

\subsection*{Acknowledgements}

We thank Pierre Deligne for the formulation of Theorem \ref{thm:orth-symp-relation}.
We also thank Marcel Novaes and Sasha Veselov for helpful discussions
related to this work, and the anonymous referee for several suggestions
that improved the presentation of this work.

\section{Maps on surfaces\label{sec:Maps-on-surfaces}}

In the rest of this paper, we view $\wl\in\F_{r}$ as fixed. For a
given word $w_{j}\in\F_{r}$ we may write

\begin{equation}
w_{j}=x_{i_{1}^{j}}^{\varepsilon_{1}^{j}}x_{i_{2}^{j}}^{\varepsilon_{2}^{j}}\ldots x_{i_{|w_{j}|}^{j}}^{\varepsilon_{|w_{j}|}^{j}},\quad\varepsilon_{u}^{j}\in\{\pm1\},i_{u}^{j}\in[r],\label{eq:combinatorial-word}
\end{equation}
where if $i_{u}^{i}=i_{u+1}^{j}$, then $\varepsilon_{u}^{j}=\varepsilon_{u+1}^{j}$.
In other words, we write each $w_{j}$ in reduced form. Recall that
$B$ is the basis $x_{1},\ldots,x_{r}$. We define the \emph{total
unsigned exponent} \marginpar{total\protect \\
unsigned\protect \\
exponent}of a generator $x_{t}\in B$ in $\wl$ to be 
\[
\sum_{j=1}^{\ell}\#\{1\leq u\leq|w_{j}|\,:\,i_{u}^{j}=t\}.
\]

\subsection{Construction of maps on surfaces from matchings\label{subsec:Construction-of-maps}}

In this section we assume that the total unsigned exponent of $x$
in $\wl$ is even for each $x\in B$, and write\marginpar{$L_{x}$}
$2L_{x}$ for this quantity.

\subsubsection{Matchings and permutations\label{subsec:Matchings-and-permutations}}

A \emph{matching} of the set $[2k]=\{1,\ldots,2k\}$ is a partition
of $[2k]$ into pairs. The collection of matchings of $[2k]$ is denoted
by\marginpar{$M_{k}$} $M_{k}$. We use two ways to identify a matching
in $M_{k}$ with a permutation in $S_{2k}$. In the first we identify
a matching $m$ with a permutation whose cycle decomposition consists
of disjoint transpositions given by the matched pairs of $m$. We
call the resulting permutation $\pi_{m}$.\marginpar{$\pi_{m}$}

In the second, given a matching $m\in M_{k}$ we canonically view
$m$ as an ordered list of ordered pairs $((m_{(1)},m_{(2)}),(m_{(3)},m_{(4)}),\ldots,(m_{(2k-1)},m_{(2k)}))$
with 
\begin{equation}
m_{(1)}<m_{(3)}<\ldots<m_{(2k-1)},\quad m_{(2r-1)}<m_{(2r)},\,r\in[k].\label{eq:m-order1}
\end{equation}
As such, we have an embedding $M_{k}\to S_{2k}$ , $m\mapsto\sigma_{m}$
by sending the matching $m$ to the permutation $\sigma_{m}:i\mapsto m_{(i)}$.
\marginpar{$\sigma_{m}$}

Following Collins and \'{S}niady \cite{CS} we introduce a metric
on $M_{k}$ as follows. For a permutation $\sigma$, write $|\sigma|$
for the minimum number of transpositions that $\sigma$ can be written
as a product of. For matchings $m,m'\in M_{k}$, we define \marginpar{$\rho$}
\begin{equation}
\rho(m,m')=\frac{|\pi_{m}\pi_{m'}|}{2}.\label{eq:rho-def}
\end{equation}
Since both permutations $\pi_{m}$ and $\pi_{m'}$ have the same sign,
$\rho(m,m')\in\Z_{\geq0}$.

\subsubsection{\label{subsec:Markings-of-points-on-wedge}Markings of $\protect\wedger$}

We now describe certain markings of $\wedger$ that will be used in
our construction of maps on surfaces. For any given tuple of positive
integers $\{\k_{x}\}_{x\in B}$ we will mark additional points on
the circles of $\wedger$ as follows. On the circle corresponding
to the each generator $x\in B$ we mark $\k_{x}+1$ distinct ordered
points $(x,0),\ldots,(x,\k_{x})$, placed consecutively along the
circle according to its orientation, and disjoint from $o$. This
is illustrated in the bottom part of Figure \ref{fig:Constructing-a-surface}
for the case $B=\left\{ x,y\right\} $ and $\kappa_{x}=\kappa_{y}=1$.

\subsubsection{\label{subsec:map-on-circle-from-word}Construction of a map on a
circle from a combinatorial word}

Recall our ongoing assumptions that all words are $\neq1$. For each
word $w\neq1$ we construct an oriented marked circle $C(w)$. Write
\begin{equation}
w=x_{j_{1}}^{\varepsilon_{1}}x_{j_{2}}^{\varepsilon_{2}}\ldots x_{j_{|w|}}^{\varepsilon_{|w|}}\label{eq:w for def of C(w)}
\end{equation}
in reduced form (as in (\ref{eq:combinatorial-word})). Begin with
$|w|$ disjoint copies of $[0,1]$, and denote the $u$th copy by
$[0,1]_{u}$. Give each interval the orientation from $0$ to $1$.
On each interval choose arbitrarily a map
\[
\gamma_{u}:[0,1]_{u}\to\wedger
\]
such that $\gamma_{u}(0)=\gamma_{u}(1)=o$, $\gamma_{u}:(0,1)_{u}\to S_{x_{j_{u}}}^{1}-\{o\}$
is a diffeomorphism, and the loop in $\wedger$ parameterized by $\gamma_{u}$,
based at $o$, corresponds to $x_{j_{u}}^{\varepsilon_{u}}\in\F_{r}$
at the level of the fundamental group. Now cyclically concatenate
all the intervals and maps together to obtain a circle $C(w)$ and
a map $\gamma_{w}:C(w)\to\wedger$.

Let $v_{w}$ be the initial point $0_{1}\in[0,1]_{1}$ of this circle
$C(w)$. The map $\gamma_{w}$ has the property that $\gamma_{w}(v_{w})=o$
and $(\gamma_{w})_{*}$ maps a generator of $\pi_{1}(C(w),v_{w})$
to $w\in\F_{r}$. We give $C(w)$ the orientation such that the order
of the intervals read, beginning at $v_{w}$ and following the orientation,
matches the left to right order  of (\ref{eq:w for def of C(w)}).
As such, the intervals of $C(w)$ are in one-to-one correspondence
with the letters of $w$.

To clarify and summarize, by definition $w\in\F_{r}=\pi_{1}(o,\wedger)$
corresponds to a homotopy class of a loop based at $o$. What we have
done here is pick a particular representative, $\gamma_{w}$, of this
homotopy class such that the sequence of circles traversed by the
loop is prescribed by $w$ through (\ref{eq:w for def of C(w)}),
and each traversal of a circle is monotone.

\subsubsection{Construction of a map on a surface from words and matchings\label{subsec:Construction-of-amap-on-surface}}

In this section we will describe a construction that takes in the
following
\begin{center}
\textbf{}%
\noindent\fbox{\begin{minipage}[t]{1\columnwidth - 2\fboxsep - 2\fboxrule}%
\textbf{Input.}
\begin{itemize}
\item A tuple $\k=\{\k_{x}\}_{x\in B}$ of non-negative integers.
\item A collection $\m=\{(m_{x,0},\ldots,m_{x,\k_{x}})\}_{x\in B}$ that
contains for each generator $x\in B$, an ordered $(\k_{x}+1)$-tuple
$(m_{x,0},\ldots,m_{x,\k_{x}})$ of elements in $M_{L_{x}}$, where
$2L_{x}$ is the unsigned exponent of $x$ in $\wsl$. We denote by
$\match^{\k}=\match^{\kappa}(\wl)$ the collection of all possible
such $\m$, for fixed $\k$. We write $\match^{*}=\cup_{\k\in\Z_{\geq0}^{B}}\match^{\k}$.
If $\m\in\match^{\k}$ we will say $\kappa(\m)=\kappa$. \textbf{Warning:}\textbf{\emph{
}}our notation $\match$ in this paper is not exactly the same as
in \cite{MP2}\footnote{In \cite{MP2} the matchings are more restrictive: it was assumed
that the matchings $m_{x,i}$ were subordinate to a partition of $[2L_{x}]$
into two blocks, corresponding to the occurrences of $x^{+1}$ and
$x^{-1}$ in $\wl$, whereas here our matchings are arbitrary matchings
of $[2L_{x}]$.}.
\end{itemize}
\end{minipage}}
\par\end{center}

The output of our construction is the following\marginpar{$\protect\match^{\kappa}$ $\protect\match^{*}$}
\begin{center}
\textbf{}%
\noindent\fbox{\begin{minipage}[t]{1\columnwidth - 2\fboxsep - 2\fboxrule}%
\textbf{Output.}
\begin{itemize}
\item A surface $\Sigma_{\m}$ with $\ell$ oriented boundary components
$C(w_{1}),\ldots,C(w_{\ell})$, with a given $CW$-complex structure,
and a marked point $v_{j}$ on each boundary component $C(w_{j})$.
\item A continuous function $f_{\m}:\Sigma_{\m}\to\wedger$ such that $f_{\m}\lvert_{C(w_{j})}$
is the function $\gamma_{w_{j}}$ constructed in $\S$\ref{subsec:map-on-circle-from-word}.
In particular, $f_{\m}(v_{j})=o$ for each $1\leq i\leq\ell$ and
$(f_{\m})_{*}$ maps the generator of $\pi_{1}(C(w_{j}),v_{j})$ specified
by the orientation of $C(w_{j})$ to $w_{j}$.
\end{itemize}
\end{minipage}}
\par\end{center}

Note that the resulting $(\Sigma_{\m},f_{\m})$ is an admissible map
for $\wl$, as in Definition \ref{def:admissible-maps-and-surfaces}.
As will become clearer in the sequel (mostly $\S$\ref{sec:The-transverse-map}),
different collections $\m$ of matchings may result in the same equivalence
class of admissible pairs $\left(\Sigma_{\m},f_{\m}\right)$.

The construction is essentially the same as in \cite[\S 2.2, \S 2.5]{MP2},
except in the current case there are less restrictions on the matchings
present, so the construction can lead to non-orientable surfaces.
The construction of a surface from a single matching per basis element
(so $\kappa_{x}=0$ for all $x\in B$) already appears in \cite{CULLER}.
We give the details now.

\textbf{I. The one-skeleton. }We first perform the construction of
$\S$\ref{subsec:map-on-circle-from-word} for each word $w_{j}$.
This gives us a collection of oriented based circles $(C(w_{j}),v_{w_{j}})$
that are respectively subdivided into $|w_{j}|$ intervals, and maps
$\gamma_{w_{j}}:C(w_{j})\to\wedger$. These oriented circles will
form the oriented boundary
\[
\delta\Sigma_{\m}\stackrel{\mathrm{def}}{=}\bigcup_{j=1}^{\ell}C(w_{j})
\]
of $\Sigma_{\m}$ and the map 
\begin{align*}
\gamma & \stackrel{\mathrm{def}}{=}\cup_{j=1}^{\ell}\gamma_{w_{j}}:\delta\Sigma_{\m}\to\wedger
\end{align*}
will be the restriction of $f_{\m}$ to the boundary of $\Sigma_{\m}$.
We write $v_{j}=v_{w_{j}}$ in the sequel for the marked points on
the $C(w_{j})$.

For each $x\in B$, the preimage 
\[
\gamma^{-1}(S_{x}^{1}-\{o\})
\]
is a disjoint union of $2L_{x}$ open sub-intervals of $\delta\Sigma_{\m}$.
This collection of intervals is denoted by\linebreak{}
\marginpar{$\protect\I_{x}$} $\I_{x}=\I_{x}$($\wsl)$. We identify
$\I_{x}$ with $[2L_{x}]$ in some arbitrary but fixed way. Moreover,
each element of $\I_{x}$ corresponds to an appearance of $x$ or
$x^{-1}$ in a unique $w_{j}$.

Next we add to the one-skeleton of $\Sigma_{\m}$ by gluing some arcs
by their endpoints to $\delta\Sigma_{\m}$. We will call these arcs\marginpar{matching arcs}
\emph{matching arcs. }For each $x\in B$ we mark points $(x,0),\ldots,(x,\k_{x})$
on $S_{x}^{1}$ as in $\S$\ref{subsec:Markings-of-points-on-wedge}.
Since $\gamma$ maps each interval $I\in\I_{x}$ monotonically to
$S_{x}^{1}$, there are uniquely specified distinct points \marginpar{$p_{I}\left(k\right)$}$p_{I}(0),\ldots,p_{I}(\k_{x})$
in $I$ such that $\gamma(p_{I}(k))=(x,k)$ for $0\leq k\leq\k_{x}$.

Now, for every pair of intervals $I$ and $J$ that are matched by
$m_{x,k}$, glue a matching arc to $\delta\Sigma_{\m}$ with endpoints
at $p_{I}(k)$ and $p_{J}(k)$. Carrying out this process for all
generators $x$, now every $p_{I}(k)$ point in $\delta\Sigma_{\m}$
is the endpoint of a unique arc. Call the resulting one dimensional
\emph{CW}-complex $\Sigma_{\m}^{(1)}$. It consists of $\delta\Sigma_{\m}$
together with the matching-arcs described in the current paragraph.
More precisely, we consider $\Sigma_{\m}^{(1)}$ to have the following
0 and 1-cells:
\begin{itemize}
\item The 0-cells (vertices) are just the points $p_{I}(k)$. Since there
are $2L_{x}$ intervals in $\I_{x}$ labeled by $x$ and $\k_{x}+1$
points $p_{I}(k)$ in each such interval, there are $2\sum_{x\in B}L_{x}(\k_{x}+1)$
vertices in the $0$-skeleton.
\item The $1$-cells (edges) are of two different types. Firstly, there
is an edge in $\delta\Sigma_{\m}$ for each connected component of
$\delta\Sigma_{\m}-\{p_{I}(k)\}$. Hence there are $2\sum_{x\in B}L_{x}(\k_{x}+1)$
such edges. Secondly, there is an edge for each matching-arc. Since
the matching arcs give a matching of the points $p_{I}(k)$, there
are $\sum_{x\in B}L_{x}(\k_{x}+1)$ such edges. Hence in total there
are $3\sum_{x\in B}L_{x}(\k_{x}+1)$ edges.
\end{itemize}
Additionally, each component $C(w_{j})$ of $\delta\Sigma_{\m}$ contains
a marked point $v_{j}$. These are not part of the \emph{CW}-complex
structure (nor is any point of $\gamma^{-1}\left(o\right)$). We define
$f_{\m}^{(1)}$ on $\Sigma_{\m}^{(1)}$ by $f_{\m}^{(1)}\lvert_{\delta\Sigma_{\m}}=\gamma$
and by requiring that $f_{\m}^{(1)}$ is constant on all matching
arcs. This completely specifies $f_{\m}^{(1)}$, since the endpoints
of matching arcs are in $\delta\Sigma_{\m}$, so the value of $f_{\m}^{(1)}$
on matching arcs is specified by $\gamma$, and by construction of
the matching arcs, the two endpoints of any two matching arcs have
the same value under $\gamma$.

\textbf{II. The two-skeleton. }Next we complete the construction of
$\Sigma_{\m}$ by gluing in discs that will be the 2-cells of the
\emph{CW}-complex. We glue in different types of discs as follows.

\emph{Type-($x,k)$ discs. }For fixed $x$, if $0\leq k<\kappa_{x}$,
let $R_{x,k}$ be the open interval in $S_{x}^{1}-\{(x,i)\}-\{o\}$
that abuts the points $(x,k)$ and $(x,k+1)$. Let $\overline{R_{x,k}}$
be the closure of this component. The preimage $(f_{\m}^{(1)})^{-1}(\overline{R_{x,k}})$
is a collection of disjoint cycles in $\Sigma_{\m}^{(1)}$ and for
each of these cycles we glue the boundary of a disc simply along the
cycle. These discs are called \emph{type-$(x,k)$ discs.}

\emph{Type-o discs. }Let $R_{o}$ be the star-like connected component
of $\wedger-\{(x,i)\}$ that contains the point $o$ and abuts points
of the form $(x,0)$ and $(x,\kappa_{x})$ (for all $x$). Let $\text{\ensuremath{\overline{R_{o}}}}$
be the closure of this component. The preimage $(f_{\m}^{(1)})^{-1}(\overline{R_{o}})$
is a one-dimensional subcomplex in $\Sigma_{\m}^{(1)}$.\emph{ }We
consider simple cycles $c$ in $(f_{\m}^{(1)})^{-1}(\overline{R_{o}})$
with the property that $f_{\m}^{(1)}\lvert_{c}$ \uline{never traverses
a point} $(x,0)$ or $(x,\kappa_{x})$. Namely, when going along $c$,
whenever $f_{\m}^{\left(1\right)}\left(c\right)$ reaches some $\left(x,0\right)$
or $\left(x,\kappa_{x}\right)$, it then stays at this point for a
while (as $c$ itself travels along a matching-arc) and then leaves
in a backtracking move. For such a cycle we glue the boundary of a
disc simply along the cycle. We call the disc a \emph{type-o }disc.

\emph{The resulting (oriented) surface obtained by gluing in these
discs is $\Sigma_{\m}$. Its CW-complex structure is the CW-complex
we described for $\Sigma_{\m}^{(1)}$ together with the glued discs
as two-cells.}

Note that each disc $D$ of $\Sigma_{\m}$ has its boundary mapped
to a nullhomotopic curve in $\wedger$; for type-$(x,k)$ discs this
is because the boundary is mapped into an interval, and for type-$o$
discs this is due to our condition that $f_{\m}^{(1)}$ never traverses
a point $(x,k)$ when restricted to the boundary of the disc. Hence
we can extend $f_{\m}^{(1)}$ to a continuous function from the disc
to $\overline{R_{o}}$ or $\overline{R_{x,k}}$ respectively, such
that only the points in the preimage of $\{(x,k)\}$ are in $\Sigma_{\m}^{(1)}$.
In other words, the extended function maps the interior of the disc
to either $R_{o}$ or some $R_{x,k}$. We pick such an extension for
each disc and this defines $f_{\m}$ on all of $\Sigma_{\m}$. Note
that the extension at every disk is unique up to homotopy, and therefore
$f_{\m}$ in general is well-defined up to homotopy. \emph{This concludes
the construction of the pair $(\Sigma_{\m},f_{\m})$. }This construction
is illustrates in Figure \ref{fig:Constructing-a-surface}.

\begin{figure}[t]
\centering{}\includegraphics[scale=0.75]{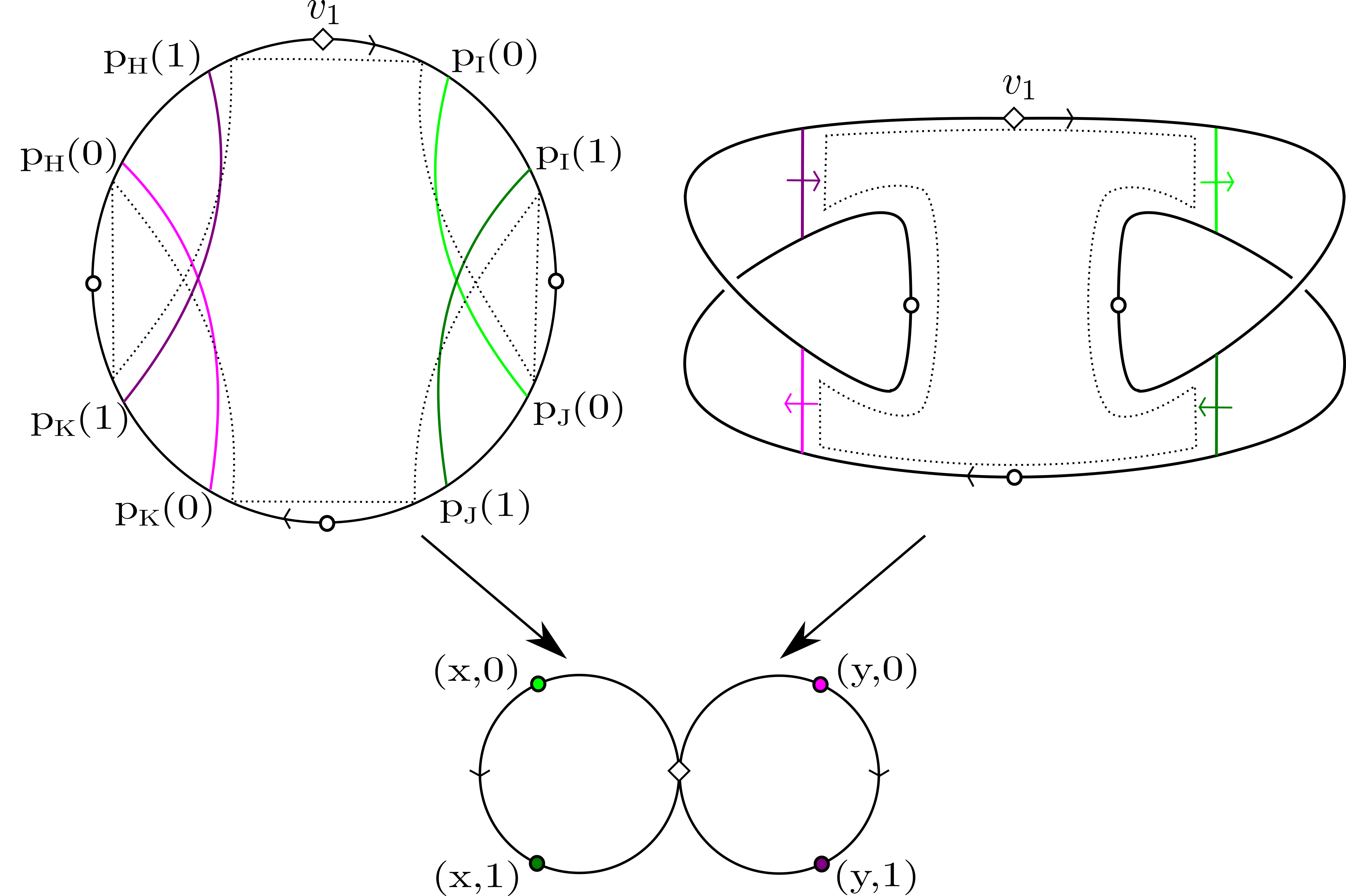}\caption{Constructing a pair $(\Sigma_{\protect\m},f_{\protect\m})$ in the
case that $r=2$, $\ell=1$, and $w_{1}=x^{2}y^{2}$. We write $x,y$
for the generators of $\protect\F_{2}$. Here $\kappa_{x}=\kappa_{y}=1$,
so there are two matchings per generator. Also, $L_{x}=L_{y}=1$ so
$\left|M_{x}\right|=\left|M_{y}\right|=1$ and the matchings are dictated.
The top left picture shows the matching arcs attached to $C(w_{1})$.
The top right picture shows the constructed $\Sigma_{\protect\m}$.
The dotted line in both top pictures follows along the matching arcs
to show how the single type-$o$ disc should be glued to $\Sigma_{\protect\m}^{(1)}$.
We have called the boundary intervals $I,J,K,H$. The colors of the
points marked in $\vee^{2}S^{1}$ in the bottom picture match with
the matching arcs in their preimage under $f_{\protect\m}$. The arrows
in the top right picture mark the normal direction on the matching
arcs coming from the fixed orientation on the petals of the bouquet.\label{fig:Constructing-a-surface}}
\end{figure}

\subsubsection{The Euler characteristic of the constructed surface $\Sigma_{\protect\m}$}

In this section we calculate the Euler characteristic of $\Sigma_{\m}$.
Recall the definition of $\rho$ from (\ref{eq:rho-def}) and that
$\m\in\match^{\k}(\wl)$ is a collection of matchings.
\begin{lem}
\label{lem:Euler-char-calc}We have 
\begin{equation}
\chi(\Sigma_{\m})=-\sum_{x\in B}L_{x}+\#\{\text{type-\ensuremath{o} discs of \ensuremath{\Sigma_{\m}\}}}-\sum_{x\in B,0\leq k<\k_{x}}\rho(m_{x,k},m_{x,k+1}).\label{eq:chi-formula}
\end{equation}
\end{lem}

\begin{proof}
We have
\begin{align}
\chi(\Sigma_{\m}) & =V-E+F\nonumber \\
 & =2\sum_{x\in B}L_{x}(\k_{x}+1)-3\sum_{x\in B}L_{x}(\k_{x}+1)+\#\{\text{type-\ensuremath{o} discs of \ensuremath{\Sigma_{\m}\}}}+\nonumber \\
 & \sum_{x\in B,0\leq k<\k_{x}}\#\{\text{type-\ensuremath{(x,k)} discs of \ensuremath{\Sigma_{\m}\}}}\nonumber \\
 & =-1\sum_{x\in B}L_{x}(\k_{x}+1)+\#\{\text{type-\ensuremath{o} discs of \ensuremath{\Sigma_{\m}\}}}+\sum_{x\in B,0\leq k<\k_{x}}\#\{\text{type-\ensuremath{(x,k)} discs of \ensuremath{\Sigma_{\m}\}}}.\label{eq:chi-sigma-1}
\end{align}
Recall from $\S$\ref{subsec:Matchings-and-permutations} that we
can associate to $m\in M_{k}$ a permutation $\pi_{m}$ in $S_{2k}$
all of whose cycles have length 2. As one follows along the boundary
of some type-$\left(x,k\right)$ disc, namely, along one of the disjoint
cycles in $\left(f_{\m}^{(1)}\right)^{-1}\left(\overline{R_{x,k}}\right)$,
the matching arcs alternate between arcs mapped to $\left(x,k\right)$
and arcs mapped to $\left(x,k+1\right)$, and so the boundary corresponds
to a cycle of $\pi_{m_{x,k}}\pi_{m_{x,k+1}}$. In the other orientation
of the boundary of the same disk we get a second cycle of $\pi_{m_{x,k}}\pi_{m_{x,k+1}}$,
so every type-$(x,k)$ disc corresponds to exactly two cycles of $\pi_{m_{x,k}}\pi_{m_{x,k+1}}$.
On the other hand, every cycle of $\pi_{m_{x,k}}\pi_{m_{x,k+1}}$
corresponds a to unique type-$\left(x,k\right)$ disc. For example,
if $L_{x}=3$ and, through the identification of $\I_{x}$ with $\left[2L_{x}\right]$,
$\pi_{m_{x,k}}=\left(12\right)\left(34\right)\left(56\right)$ and
$\pi_{m_{x,k+1}}=\left(13\right)\left(24\right)\left(56\right)$,
then $\pi_{m_{x,k}}\pi_{m_{x,k+1}}=\left(14\right)\left(23\right)\left(5\right)\left(6\right)$,
and $\left(\Sigma_{\m},f_{\m}\right)$ contains two type $\left(x,k\right)$-discs:
one corresponding to the cycles $\left(14\right)$ and $\left(23\right)$
of $\pi_{m_{x,k}}\pi_{m_{x,k+1}}$, and one corresponding to the cycles
$\left(5\right)$ and $\left(6\right)$.

Since for any $\sigma\in S_{2L_{x}}$ we have $|\sigma|=2L_{x}-\#\{\text{cycles of \ensuremath{\sigma}\}}$,
we have 
\begin{align*}
\rho(m_{x,k},m_{x,k+1}) & =\frac{\left|\pi_{m_{x,k}}\pi_{m_{x,k+1}}\right|}{2}=L_{x}-\frac{\#\{\text{cycles of }\text{\ensuremath{\pi_{m_{x,k}}\pi_{m_{x,k+1}}}\}}}{2}\\
 & =L_{x}-\#\{\text{type-\ensuremath{(x,k)} discs of \ensuremath{\Sigma_{\m}\}.}}
\end{align*}
Therefore $\#\{\text{type-\ensuremath{(x,k)} discs of \ensuremath{\Sigma_{\m}\}}}=L_{x}-\rho(m_{x,k},m_{x,k+1})$
and from (\ref{eq:chi-sigma-1}) we get
\begin{align}
\chi(\Sigma_{\m}) & =-1\sum_{x\in B}L_{x}(\k_{x}+1)+\#\{\text{type-\ensuremath{o} discs of \ensuremath{\Sigma_{\m}\}}}+\sum_{x\in B,0\leq k<\k_{x}}L_{x}-\rho(m_{x,k},m_{x,k+1})\nonumber \\
 & =-\sum_{x\in B}L_{x}+\#\{\text{type-\ensuremath{o} discs of \ensuremath{\Sigma_{\m}\}}}-\sum_{x\in B,0\leq k<\k_{x}}\rho(m_{x,k},m_{x,k+1}).
\end{align}
\end{proof}

\subsection{The anatomy of $\protect\surfaces^{*}(\protect\wl)$}

The passage between the algebraic and topological view on $\surfaces^{*}(\wl)$
stems from the following result of Culler \cite[1.1]{CULLER}:
\begin{lem}[Culler]
\label{lem:culler}An element $w\in\F_{r}$ is a product of $g$
commutators (resp.~$s$ squares) if and only if there exists an admissible
map $(\Sigma,f)$ for $w$ with $\Sigma$ a genus-$g$ orientable
surface (resp.~a connected sum of $s$ copies of $\R P^{2}$) with
a disc removed.
\end{lem}

We first describe when $\surfaces^{*}(\wl)$ is empty.

\begin{lem}
\label{lem:when is surfaces nonempty}The following are equivalent:
\begin{enumerate}
\item $\surfaces^{*}(\wl)$ is non-empty.
\item The unsigned exponent of each $x\in B$ in $\wl$ is even.
\item $w_{1}w_{2}\cdots w_{\ell}$ can be written as a product of squares
in $\F_{r}$.
\end{enumerate}
\end{lem}

\begin{proof}
This is very close to standard facts, in particular the results appearing
in Culler \cite{CULLER}, but we give a proof here for completeness.

We first prove that the first statement implies the second. Consider
the map
\[
h:\pi_{1}(\wedger,o)\to H_{1}(\wedger,\Z/2\Z)
\]
 which induces a map $H:\F_{r}\to(\Z/2\Z)^{|B|}$. Then $H$ maps
$x_{i}\in B$ to the corresponding standard generator $e_{i}$ of
$(\Z/2\Z)^{|B|}$. If $[(\Sigma,f)]\in\surfaces^{*}(\wl)$ and $\delta_{1},\ldots,\delta_{\ell}$
are the boundary components of $\Sigma$ then we have
\[
0=\sum_{i=1}^{\ell}[\delta_{i}]\in H_{1}(\Sigma,\Z/2\Z).
\]
Applying the map induced by $f$ on homology to this equation implies
$\sum_{i=1}^{\ell}H(w_{i})=0$, which means the unsigned exponent
of each $x\in B$ in $\wl$ is even.

Conversely, if the unsigned exponent of each $x\in B$ in $\wl$ is
even, then the construction of $\S$\ref{subsec:Construction-of-maps}
shows that $\surfaces^{*}(\wl)$ is non-empty. This proves that the
second statement implies the first.

The second statement is easily seen to be equivalent to the statement
that the unsigned exponent of each $x\in B$ in $w_{1}w_{2}\ldots w_{\ell}$
is even. By what we have proved, this is equivalent to $\surfaces^{*}(w_{1}w_{2}\ldots w_{\ell})\neq\emptyset$.
Then by Lemma \ref{lem:culler}, this is equivalent to $w_{1}w_{2}\ldots w_{\ell}$
being either the product of commutators or the product of squares.
But since any commutator is a product of squares (see $\S$\ref{subsec:Consequences-of-the}),
this proves that the second statement is equivalent to the third statement.
\end{proof}
For given $\wl$, let \marginpar{$\chi_{\max}$}\label{chi_max}
\[
\chi_{\max}(\wl)\stackrel{\mathrm{def}}{=}\max\left\{ \chi\left(\Sigma\right)\::\:[(\Sigma,f)]\in\surfaces^{*}(\wl)\right\} ,
\]
or $-\infty$ if $\surfaces^{*}(\wl)$ is empty. Note that if $\surfaces^{*}(\wl)$
is non-empty the maximum clearly exists. Indeed, any connected component
of $\Sigma$ with $[(\Sigma,f)]\in\surfaces^{*}(\wl)$ must contain
a boundary component, so $\Sigma$ has at most $\ell$ components.
Moreover, any connected component of $\Sigma$ has Euler characteristic
at most $1$, so $\chi(\Sigma)\leq\ell$.
\begin{lem}
\label{lem:chi-max-non-pos}Assume that all $w_{j}\neq1$. Then $\chi_{\max}(\wl)\leq0$.
Moreover, an admissible map $[(\Sigma,f)]\in\surfaces^{*}(\wl)$ satisfies
$\chi(\Sigma)=0$ if and only if all the connected components of $\Sigma$
are annuli or Möbius bands.
\end{lem}

\begin{proof}
As mentioned above, if $[(\Sigma,f)]\in\surfaces^{*}(\wl)$ then writing
$\Sigma=\cup_{k=1}^{m}\Sigma_{k}$ as a union of connected components,
we have $m\leq\ell$, each $\Sigma_{k}$ has at least one boundary
component, and $\chi(\Sigma_{k})\leq1$. However, $\chi(\Sigma_{k})$
is equal to $1$ only if $\Sigma_{k}$ is a disc, which can only happen
if the $w_{j}$ that marks the boundary component of the disc is $w_{j}=1$.
This is because the boundary component of the disc is nullhomotopic
in the disc. So given all $w_{j}\neq1$, we have $\chi(\Sigma_{k})\leq0$.
Given $\chi(\Sigma)=0$, this implies $\chi$$(\Sigma_{k})=0$ for
$1\leq k\leq m$. Finally, by the classification of surfaces, we have
$\chi(\Sigma_{k})=0$ only if $\Sigma_{k}$ is an annulus or a Möbius
band. Conversely, if all the connected components of $\Sigma$ are
annuli or Möbius bands then obviously $\chi(\Sigma)=0$.
\end{proof}
\begin{lem}
\label{lem:Culler2}If $w\in\F_{r}$ then $\chi_{\max}(w)=1-\min(\sql(w),2\cl(w))$.
\end{lem}

\begin{proof}
Since any $[(\Sigma,f)]\in\surfaces^{*}(w)$ has $\Sigma$ connected,
Culler's Lemma \ref{lem:culler} is easily seen to imply the result,
since the Euler characteristic of a genus $g$ orientable surface
with one boundary component is $1-2g$ and the Euler characteristic
of a connected sum of $s$ copies of $\R P^{2}$ with a disc removed
is $1-s$.
\end{proof}

\subsection{Proof of Theorem \ref{thm:main-theorem} when $\protect\surfaces^{*}(\protect\wl)$
is empty}
\begin{lem}
\label{lem:integral-is-zero}If the total unsigned exponent of some
$x\in B$ in $\wl$ is not even, then $\tr_{\wl}^{G}(n)=0$ for any
$n\geq1$ and for $G=\O,\Sp$.
\end{lem}

\begin{proof}
Suppose, for example, that $x=x_{1}$. Let $u$ be the (odd) total
unsigned exponent of $x_{1}$ in $\wl$. Note that minus the identity,
$-I_{n}$, is in the center of $\O(n)$. By the translation-invariance
of Haar measure, we have
\begin{align*}
\trwl^{\O}(n) & =\int_{\O(n)^{r}}\mathrm{tr}(w_{1}(g_{1},\ldots,g_{r}))\ldots\mathrm{tr}(w_{\ell}(g_{1},\ldots,g_{r}))d\mu_{n}(g_{1})\ldots d\mu_{n}(g_{r})\\
 & =\int_{\O(n)^{r}}\mathrm{tr}(w_{1}(-I_{n}g_{1},\ldots,g_{r}))\ldots\mathrm{tr}(w_{\ell}(-I_{n}g_{1},\ldots,g_{r}))d\mu_{n}(g_{1})\ldots d\mu_{n}(g_{r})\\
 & =(-1)^{u}\int_{\O(n)^{r}}\mathrm{tr}(w_{1}(g_{1},\ldots,g_{r}))\ldots\mathrm{tr}(w_{\ell}(g_{1},\ldots,g_{r}))d\mu_{n}(g_{1})\ldots d\mu_{n}(g_{r})\\
 & =-\int_{\O(n)^{r}}\mathrm{tr}(w_{1}(g_{1},\ldots,g_{r}))\ldots\mathrm{tr}(w_{\ell}(g_{1},\ldots,g_{r}))d\mu_{n}(g_{1})\ldots d\mu_{n}(g_{r})\\
 & =-\trwl^{\O}(n).
\end{align*}
The proof for $G=\Sp$ is the same, using that $-I_{2n}$ is in the
center of $\Sp(n)$.
\end{proof}
\begin{proof}[Proof of Theorem \ref{thm:main-theorem} when $\surfaces^{*}(\wl)$
is empty]
 If $\surfaces^{*}(\wl)$ is empty, then by Lemma \ref{lem:when is surfaces nonempty},
the total unsigned exponent of some $x\in B$ in $\wl$ is not even.
Therefore by Lemma \ref{lem:integral-is-zero}, $\trwl^{\O}(n)=0$
for any $n\geq1$.
\end{proof}

\subsection{Incompressible and almost-incompressible maps}

Certain types of admissible maps $(\Sigma,f)$ will pay a special
role later in the paper (see $\S$\ref{subsec:The--Euler-characteristic-is the usual one}).
These are the \emph{incompressible }and \emph{almost-incompressible}
maps.
\begin{defn}
\label{def:incompressible and almost-incompressible}Let $(\Sigma,f)$
be an admissible map. We say that $(\Sigma,f)$ is \emph{compressible
}if there is a non-nullhomotopic simple closed curve $c\subset\Sigma$
such that $f(c)$ is nullhomotopic in $\wedger$. We call $c$ a \emph{compressing
curve. }We say $(\Sigma,f)$ is \emph{incompressible }if it is not
compressible. If every compressing curve is non-generic in the sense
of Definition \ref{def:generic curve} below, namely, if every compressing
curve is either one-sided\footnote{\label{fn:one-sided curve}A curve is \emph{one-sided} if a thickening
of the curve is a Möbius band, or equivalently, cutting along the
curve results in a surface with only one new boundary component. } or bounds a Möbius band, we say that $\left(\Sigma,f\right)$ is
\emph{almost-incompressible}.
\end{defn}

It is clear that if $(\Sigma,f)$ is (almost-) incompressible and
$(\Sigma',f')\approx(\Sigma,f)$ then $(\Sigma',f')$ is also (almost-)
incompressible. So there is a well-defined notion of an element $[(\Sigma,f)]\in\surfaces^{*}(\wl)$
being (almost-) incompressible.
\begin{lem}
\label{lem:max-chi-implies-incompressible}Let $[(\Sigma,f)]\in\surfaces^{*}(\wl)$.
If $\chi(\Sigma)=\chi_{\max}(\wl)$ then $[(\Sigma,f)]$ is incompressible,
and if $\chi(\Sigma)=\chi_{\max}(\wl)-1$ then $[(\Sigma,f)]$ is
almost-incompressible.
\end{lem}

\begin{proof}
Denote $\chi_{\max}\stackrel{\mathrm{def}}{=}\chi_{\max}\left(\wl\right)$
and assume that $\chi\left(\Sigma\right)\ge\chi_{\max}-1$. If $(\Sigma,f)$
is compressible, let $c\subset\Sigma$ be a compressing curve. We
can cut $\Sigma$ along $c$ to obtain a new surface $\Sigma_{\cut}$
with either \textbf{$b=1$ }or \textbf{$b=2$} new boundary components.
Then we can cap discs on any new boundary components of $\Sigma_{\cut}$
to obtain a new surface $\Sigma'$. Moreover, since $f(c)$ was nullhomotopic,
it is possible to extend $f$ across these discs to obtain a new pair
$(\Sigma',f')$ with $\chi(\Sigma')=\chi(\Sigma)+b$.

If $b=1$, then $(\Sigma',f')$ is admissible, since we must have
cut along a one-sided curve and hence did not create any new connected
components. Since $\chi(\Sigma')=\chi(\Sigma)+1$, we must have $\chi(\Sigma)=\chi_{\max}-1$.

If $b=2$ and every connected component of $\Sigma'$ has boundary,
then the pair $(\Sigma',f')$ is admissible and $\chi\left(\Sigma'\right)=\chi\left(\Sigma\right)+2$,
in contradiction to $\chi\left(\Sigma\right)\ge\chi_{\max}-1$. However,
in $\Sigma'$ we might have created at most one closed component $S$,
in which case the pair $(\Sigma',f')$ is not admissible. In this
case we can simply delete $S$ to obtain an admissible map $(\Sigma'',f'')$.
Since $c$ was non-nullhomotopic, $S$ is not a sphere, hence $\chi(S)\leq1$
and $\chi(\Sigma'')=\chi(\Sigma')-\chi(S)\geq\chi(\Sigma')-1=\chi(\Sigma)+1$
and therefore $\chi\left(\Sigma\right)\le\chi_{\max}-1$. Furthermore,
if $c$ was generic, $S$ is also not $\R P^{2}$, and then $\chi(S)\leq0$
and $\chi\left(\Sigma\right)\le\chi_{\max}-2$, a contradiction.
\end{proof}

\section{A combinatorial Laurent series for $\protect\trwl^{G}(n)$\label{sec:A-combinatorial-formula}}

Since we have proved Theorem \ref{thm:main-theorem} when $\surfaces^{*}(\wl)=\emptyset$,
we assume the contrary for the rest of the paper, and hence in view
of Lemma \ref{lem:when is surfaces nonempty} we assume that the total
unsigned exponent of each $x$ in $\wl$ is even. As in Section \ref{subsec:Construction-of-maps},
we write $2L_{x}$ for the total unsigned exponent of $x$ in $\wl$.

\subsection{The Weingarten calculus}

Denote by $\langle\bullet,\bullet\rangle=\langle\bullet,\bullet\rangle_{\O}$
the standard inner form on $\mathbf{R}^{n}$, and by $\langle\bullet,\bullet\rangle_{\Sp}$
the symplectic form on $\CC^{2n}$ given by 
\[
\langle v,w\rangle_{\Sp}\stackrel{\mathrm{def}}{=}v^{T}Jw
\]
(recall (\ref{eq:J})). We let\marginpar{$e_{i}$} $\{e_{i}\}$ be
the standard basis of $\mathbf{R}^{n}$ or $\CC^{2n}$. The \emph{Weingarten
calculus }allows one to interpret integrals of products of matrix
coefficients in $G(n)$, for $G=\O,\Sp$, in terms of matchings. The
source of the Weingarten calculus is Brauer-Schur-Weyl duality between
$G$ and an appropriate Brauer algebra.

Let $\Q(n)$ denote the ring of rational functions of $n$ with rational
coefficients. The following Theorem was proved for $G=\O$ by Collins
and \'{S}niady in \cite[Cor.~3.4]{CS}. The corresponding theorem
for $G=\Sp$ had its proof outlined by Collins and \'{S}niady in \cite[Thm.~4.1 and following discussion]{CS}.
The theorem was subsequently precisely stated and proved by Collins
and Stolz in \cite[Prop.~3.2]{CoSt}; see also Matsumoto \cite[Thm.~2.4]{MAT}.
Recall that $M_{k}$ marks the set of matchings on $\left[2k\right]$.
\begin{thm}
\label{thm:weingarten-integration}For $G=\O,\Sp$, there are unique,
computable, functions
\[
\wg_{k}^{G}:M_{k}\times M_{k}\to\Q(n)
\]
with the following properties. For $m_{1},m_{2}\in M_{k}$ and $n\in\Z_{\ge1}$,
let us write $\wg_{k}^{G}(m_{1},m_{2};n)$ for the evaluation of the
rational function $\wg_{k}^{G}\left(m_{1},m_{2}\right)$ at $n$.
Assume that $n\geq k$ if $G=\O$ and $2n\geq k$ if $G=\Sp$. We
have
\[
\int_{G(n)}g_{i_{1}j_{1}}\ldots g_{i_{2k}j_{2k}}d\mu_{n}(g)=\sum_{m_{1},m_{2}\in M_{k}}\delta_{{\bf i},m_{1}}^{G}\delta_{{\bf j},m_{2}}^{G}\wg_{k}^{G}(m_{1},m_{2};n).
\]
Here ${\bf i}=(i_{1},\ldots,i_{2k})$, ${\bf j}=(j_{1},\ldots,j_{2k})$,
and
\[
\delta_{\mathbf{i},m}^{G}\stackrel{\mathrm{def}}{=}\prod_{r=1}^{k}\langle e_{i_{m_{(2r-1)}}},e_{i_{m_{(2r)}}}\rangle_{G}
\]
where $m_{(j)}$ are as in (\ref{eq:m-order1}). The integral of the
product of an odd number of matrix coefficients is $0$.
\end{thm}

The function $\wg_{k}^{G}$ is called the \emph{Weingarten function
of $G$}. Although here we take Theorem \ref{thm:weingarten-integration}
as the definition of the Weingarten functions, it is certainly worth
pointing out that the papers \cite{CS,MAT} contain explicit formulas
for the Weingarten functions. It follows from these formulas that
for any $m_{1},m_{2}\in M_{k}$, the rational function $\wg_{k}^{G}(m_{1},m_{2})$
can be computed by an (explicit) algorithm in a finite number of steps.

The following lemma of Matsumoto \cite[\S 2.3.2]{MAT} shows that
the Weingarten functions of $\O$ and $\Sp$ are related in a simple
way.
\begin{lem}
\label{lem:connection-of-Weingarten}For $n\in\Z_{\geq0}$ with $2n\geq k$
and $m_{1},m_{2}\in M_{k}$
\[
\wg_{k}^{\Sp}(m_{1},m_{2};n)=(-1)^{k}\sign(\sigma_{m_{1}}\sigma_{m_{2}}^{-1})\wg_{k}^{\O}(m_{1},m_{2};-2n).
\]
\end{lem}

The following theorem of Collins and \'{S}niady gives the full Laurent
expansion for $\wg_{k}^{\O}(m_{1},m_{2})$ at $n=\infty$ and estimates
the order of vanishing. In view of Lemma \ref{lem:connection-of-Weingarten},
one has analogous results for $G=\Sp$.
\begin{thm}
\label{thm:laurent-exp-of-weingarten}Fix $m_{1},m_{2}\in M_{k}$.
\begin{enumerate}
\item \cite[Lem.~3.12]{CS} We have 
\[
\wg_{k}^{\O}(m_{1},m_{2};n)=n^{-k}\sum_{\ell\geq0}\sum_{\substack{m_{1}=m'_{0},m'_{1},\ldots,m'_{\ell}=m_{2}\in M_{k}\\
m'_{i}\neq m'_{i+1}
}
}\left(-1\right)^{\ell}n^{-\rho(m'_{0},m'_{1})-\ldots-\rho(m'_{\ell-1},m'_{\ell})}.
\]
The sum is absolutely convergent for $n\geq k$.
\item \cite[Thm.~3.13]{CS}
\[
\wg_{k}^{\O}(m_{1},m_{2};n)=O_{n\to\infty}(n^{-k-\rho(m_{1},m_{2})}).
\]
\end{enumerate}
\end{thm}

\subsection{A rational function form of $\protect\trwl^{\protect\O}$\label{subsec:A-rational-function-form-for-orth}}

We wish to find a rational function of $n$ which agrees with the
integral
\begin{equation}
\trwl^{\O}(n)=\int_{\O(n)^{r}}\mathrm{tr}(w_{1}(g_{1},\ldots,g_{r}))\ldots\mathrm{tr}(w_{\ell}(g_{1},\ldots,g_{r}))d\mu_{n}(g_{1})\ldots d\mu_{n}(g_{r})\label{eq:trace-fn-definition-O}
\end{equation}
for sufficiently large $n$, depending only on $\wl$. We denote by
$\k\equiv1$ the tuple that assigns $1$ to each $x\in B$. We now
define \marginpar{$N$}
\begin{equation}
N=N(\wl)\stackrel{\mathrm{def}}{=}\max\{L_{x}\::\:x\in B\}.\label{eq:N0}
\end{equation}

\begin{thm}
\label{thm:orth-wg-exp}For $n\geq N$, we have
\begin{align}
\trwl^{\O}(n) & =\sum_{\m\in\match^{\k\equiv1}}n^{\#\{\text{type-\ensuremath{o} discs of \ensuremath{\Sigma_{\m}}\}}}\prod_{x\in B}\wg_{L_{x}}^{\O}(m_{x,0},m_{x,1};n).\label{eq:orth-wg-exp}
\end{align}
\end{thm}

\begin{proof}
Assume $n\geq N$. We assume each $w_{j}=x_{i_{1}^{j}}^{\varepsilon_{1}^{j}}x_{i_{2}^{j}}^{\varepsilon_{2}^{j}}\ldots x_{i_{|w_{j}|}^{j}}^{\varepsilon_{|w_{j}|}^{j}}$
is written as a reduced word, as in (\ref{eq:combinatorial-word}),
and aim to evaluate (\ref{eq:trace-fn-definition-O}). Hence we have
\begin{equation}
\mathrm{tr}(w_{j}(g_{1},\ldots,g_{r}))=\sum_{q_{1}^{j},\ldots,q_{|w_{j}|}^{j}}\left(g_{i_{1}^{j}}^{\varepsilon_{1}^{j}}\right)_{q_{1}^{j}q_{2}^{j}}\left(g_{i_{2}^{j}}^{\varepsilon_{2}^{j}}\right)_{q_{2}^{j}q_{3}^{j}}\ldots\left(g_{i_{|w_{j}|}^{j}}^{\varepsilon_{|w_{j}|}^{j}}\right)_{q_{|w_{j}|}^{j}q_{1}^{j}}.\label{eq:trace-word-expansion}
\end{equation}
It will be helpful to think about the indices appearing in the above
expression in the following alternative way. Recall from $\S$\ref{subsec:map-on-circle-from-word}
that we constructed a collection of marked circles $\bigcup_{j=1}^{\ell}C(w_{j})$
and a map $\gamma:\bigcup_{j=1}^{\ell}C(w_{j})\to\wedger$. As in
$\S$\ref{subsec:Markings-of-points-on-wedge}, mark two points $(x,0)$
and $(x,1)$ on the circle in $\wedger$ corresponding to $x\in B$
for each such $x$. As in $\S$\ref{subsec:Construction-of-amap-on-surface},
mark points $p_{I}(0)$ and $p_{I}(1)$ on $\bigcup_{j=1}^{\ell}C(w_{j})$.
Recall the collection of intervals $\I_{x}=\I_{x}(\wl)$, and that
we have identified $\I_{x}$ with $[2L_{x}]$ for each $x\in B$.
Let $\I=\cup_{x\in B}\I_{x}$. For $x\in B$ we will let $\I_{x}^{\pm}$
be the collection of intervals that correspond to occurrences of $x^{\pm}$
in $\wl$. We denote $\I^{\pm}=\cup_{x\in B}\I_{x}^{\pm}$.\marginpar{$\protect\I,\protect\I_{x}^{\pm},\protect\I^{\pm}$}

For each $x$ and $k\in\{0,1\}$ we accordingly identify the points
$\{p_{I}(k):I\in\I_{x}\}$ with the set $[2L_{x}]$. This allows us
to think of the choices of $q_{1}^{j},\ldots,q_{|w_{j}|}^{j}$ over
the various $w_{j}$ as an assignment 
\[
\a:\{p_{I}(k)\::\:k=0,1\:\}\to[n]
\]
with the property that two immediately adjacent marked points $p,q$
in $\bigcup_{j=1}^{\ell}C(w_{j})$ that are not of the form $\{p,q\}=\{p_{I}(0),p_{I}(1)\}$
(i.e., internal to some interval) must have $\a(p)=\a(q)$. Write\marginpar{$\protect\A$}
$\A=\A(\wl)$ for the collection of all such assignments $\a$.

To each interval $I\in I_{x}$ we attach the group element $g(I)=g_{i}$
where $x=x_{i}\in B$. Using that $g^{-1}=g^{T}$, we can rewrite
the product over $k$ of the expressions in (\ref{eq:trace-word-expansion})
as
\begin{equation}
\prod_{j=1}^{\ell}\mathrm{tr}(w_{j}(g_{1},\ldots,g_{r}))=\sum_{\mathbf{a}\in\A}\prod_{x\in B}\prod_{I\in\I_{x}}g(I)_{\mathbf{a}(p_{I}(0))\mathbf{\mathbf{a}}(p_{I}(1))}.\label{eq:word-as-matrix-coefs}
\end{equation}
For $\mathbf{a}\in\A$ and $\m=\{(m_{x,0},m_{x,1})\}_{x\in B}\in\match^{\k\equiv1}$
a collection of matchings, we say\marginpar{$\protect\a\vdash\protect\m$}
$\a\vdash\m$ if whenever $m_{x,i}$ matches $p_{I}(i)$ and $p_{J}(i)$,
these points are assigned the same index by $\a$. In this case we
also say that $m_{x,0}$ and $m_{x,1}$ respect $\mathbf{a}$.

If $x=x_{i}$, a direct consequence of Theorem \ref{thm:weingarten-integration}
is 
\[
\int_{g\in\O(n)}\prod_{I\in\I_{x}}g_{\mathbf{a}(p_{I}(0))\mathbf{\mathbf{a}}(p_{I}(1))}d\mu_{n}(g_{i})=\sum_{\substack{\text{matchings \ensuremath{m_{x,0}} of \ensuremath{\{p_{I}(0)}:\ensuremath{I\in\I_{x}\}} that respect \ensuremath{\mathbf{a}}}\\
\text{matchings \ensuremath{m_{x,1}} of \ensuremath{\{p_{I}(1)}:\ensuremath{I\in\I_{x}\}} that respect \ensuremath{\mathbf{a}}}
}
}\wg_{L_{x}}^{\O}(m_{x,0},m_{x,1};n).
\]
Hence
\begin{equation}
\int_{\O(n)^{r}}\prod_{x\in B}\prod_{I\in\I_{x}}g(I)_{\mathbf{a}(p_{I}(0))\mathbf{\mathbf{a}}(p_{I}(1))}d\mu_{n}(g_{1})\ldots d\mu_{n}(g_{r})=\sum_{\m\in\match^{\k\equiv1}\,:\,\a\vdash\m}\prod_{x\in B}\wg_{L_{x}}^{\O}(m_{x,0},m_{x,1};n).\label{eq:integral-of-product-of-matrix-multivar}
\end{equation}
Reordering summation and integration, we obtain
\begin{align}
\trwl^{\O}(n) & =\sum_{\mathbf{a}\in\A}\sum_{\m\in\match^{\k\equiv1}\,:\,\a\vdash\m}\prod_{x\in B}\wg_{L_{x}}^{\O}(m_{x,0},m_{x,1};n).\nonumber \\
 & =\sum_{\m\in\match^{\k\equiv1}}\#\{\a\in\A\,:\,\a\vdash\m\}\prod_{x\in B}\wg_{L_{x}}^{\O}(m_{x,0},m_{x,1};n).\label{eq:marker}
\end{align}
For fixed $\m$, the condition $\a\vdash\m$ holds if and only if
for every type-$o$ disc of $\Sigma_{\m}$, $\a$ is constant on the
set of $p_{I}(k)$ that meet the boundary of that disc. Hence 
\begin{align*}
\#\{\a & \in\A\,:\,\a\vdash\m\}=n^{\#\{\text{type-\ensuremath{o} discs of \ensuremath{\Sigma_{\m}}\}}}.
\end{align*}
This proves the theorem.
\end{proof}
Theorem \ref{thm:orth-wg-exp} has the following easy corollary that
was stated in the Introduction.
\begin{cor}
\label{cor:orth-rational-function}There is a computable rational
function $\overline{\trwl^{\O}}\in\Q(n)$ such that for $n\geq N$,
$\trwl^{\O}(n)$ is given by evaluating $\overline{\trwl^{\O}}$ at
$n$. 
\end{cor}

\begin{proof}
The formula given in Theorem \ref{thm:orth-wg-exp} expresses $\trwl^{\O}$
as a finite sum of computable rational functions since $\match^{\k\equiv1}$
is finite.
\end{proof}

\subsection{First Laurent series expansion at infinity\label{subsec:Power-series-expansion}}

Due to Theorem \ref{thm:orth-symp-relation}, we only need to discuss
$\tr_{\wl}^{\O}(n)$ throughout the rest of the paper. The full Laurent
series expansion of $\tr_{\wl}^{\O}(n)$ at $n=\infty$ involves elements
of $\match^{*}(\wl)$ with extra restrictions. Recall that an element
of $\match^{*}=\match^{*}(\wl)$ is, for some $\k=\{\k_{x}\}_{x\in B}\in\Z_{\geq0}^{B}$,
a collection $\m=\{(m_{x,0},\ldots,m_{x,\k_{x}})\}_{x\in B}$ of tuples
of matchings, where $m_{x,i}$ is a matching of $[2L_{x}]$. For any
$\k$ we write\marginpar{$\protect\matchr^{\protect\k}$} $\matchr^{\k}=\matchr^{\k}(\wl)$
for the elements $\m$ of $\match^{\k}$ with the additional constraint
that $m_{x,i}\neq m_{x,i+1}$ for each $0\leq i<\k_{x}$, and similarly
define\marginpar{$\protect\matchr^{*}$} $\matchr^{*}$. We will also
write\marginpar{$|\protect\k|$} $|\k|=\sum_{x\in B}\k_{x}$.
\begin{prop}
\label{prop:First-Laurent-Expansion}For $n\geq N$, $\tr_{\wl}^{\O}(n)$
is given by the following absolutely convergent series:
\begin{align}
\tr_{\wl}^{\O}(n) & =\sum_{\m\in\matchr^{*}}(-1)^{|\k(\m)|}n^{\chi(\Sigma_{\m})}.\label{eq:first-laurent}
\end{align}
\end{prop}

\begin{proof}
Putting the power series expansion in $n^{-1}$ for the orthogonal
Weingarten function given in Theorem \ref{thm:laurent-exp-of-weingarten}
into Theorem \ref{thm:orth-wg-exp} gives
\begin{eqnarray*}
\trwl^{\O}\left(n\right) & = & \sum_{\m\in\match^{\k\equiv1}}n^{\#\{\text{\ensuremath{o}-discs of \ensuremath{\Sigma_{\m}}\}}}\prod_{x\in B}\wg_{L_{x}}^{\O}\left(m_{x,0},m_{x,1};n\right)\\
 & = & \sum_{\m=\{(m_{x,0},m_{x,1})\}_{x\in B}}n^{\#\{\text{\ensuremath{o}-discs of \ensuremath{\Sigma_{\m}}\}}}\prod_{x\in B}n^{-L_{x}}\cdot\\
 &  & \cdot\left\{ \sum_{\kappa_{x}\geq0}\sum_{{\scriptstyle \substack{m_{x,0}=m'_{x,0},\ldots,m'_{x,\k_{x}}=m_{x,1}\in M_{L_{x}}\\
m'_{x,i}\neq m'_{x,i+1}
}
}}\left(-1\right)^{\k_{x}}n^{-\rho\left(m'_{x,0},m'_{x,1}\right)-\ldots-\rho\left(m'_{x,\k_{x}-1},m'_{x,\k_{x}}\right)}\right\} \\
 & = & n^{-\sum_{x\in B}L_{x}}\sum_{(\k_{x})_{x\in B}\in\Z_{\geq0}^{B}}(-1)^{\sum_{x\in B}\k_{x}}\cdot\\
 &  & \cdot\left\{ \sum_{\m'\in\matchr^{\k}}n{}^{\#\left\{ \text{\ensuremath{o}-discs of \ensuremath{\Sigma_{\m'}}\}}\right\} }\prod_{x\in B}n^{-\rho\left(m'_{x,0},m'_{x,1}\right)-\ldots-\rho\left(m'_{x,\k_{x}-1},m'_{x,\k_{x}}\right)}\right\} 
\end{eqnarray*}
Here we used the fact that the number of type-$o$ discs of $\Sigma_{\m'}$
are the same as those of $\Sigma_{\m}$, since they only depend on
the outer matchings $(m'_{x,0},m'_{x,\k_{x}})_{x\in B}$, that are
the same as in $\m$. Also note that each of the expressions for $\wg^{\O}$
are absolutely convergent when $n\geq N$, and we only used finitely
many such expressions, corresponding to the finitely many choices
of $\m\in\match^{\k\equiv1}(\wl)$. Using Lemma \ref{lem:Euler-char-calc},
the above can be rewritten as
\[
\sum_{\m'\in\matchr^{*}}(-1)^{|\k(\m')|}n^{\chi(\Sigma_{\m'})}.
\]
\end{proof}

\subsection{Signed matchings and a new Laurent series expansion\label{subsec:Signed-matchings}}

In the section we modify our previous definitions to get a new combinatorial
Laurent expansion for $\tr_{\wl}^{O}(n)$ in the shifted parameter
$(n-1)^{-1}$, or equivalently, an expansion for $\tr_{\wl}^{O}(n+1)$
in the parameter $n^{-1}$. The reason for doing this is that we want
to add into our Laurent expansion additional surfaces that are constructed
from discs and Möbius bands. This has no analog in \cite{MP2}. The
resulting marked surfaces are the ones that are not stabilized by
any non-trivial element of the mapping class group of the surface;
see Lemma \ref{lem:isotropy-groups} for the precise statement. The
introduction of these extra surfaces is essential in allowing us to
give clean expressions for the coefficients as in Theorem \ref{thm:main-theorem}.
We formalize this as follows.
\begin{defn}
\label{def:smatch}Let\marginpar{${\scriptstyle \protect\sm^{*}}$}
$\sm^{*}=\sm^{*}(\wl)$ be the collection of pairs $(\m,\varepsilon)$
where
\begin{itemize}
\item $\m\in\match^{*}(\wl)$,
\item $\varepsilon$ is a function from the 2-cells of $\Sigma_{\m}$ to
$\{-1,1\}$,
\item if $m_{x,i}=m_{x,i+1}$ then at least one type-$(x,i)$ disc of $\Sigma_{\m}$
must be assigned $-1$ by $\varepsilon$.
\end{itemize}
Let $\kappa(\m,\varepsilon)\stackrel{\mathrm{def}}{=}\kappa(\m)$.
We call a pair $(\m,\varepsilon)$ a \emph{signed matching.}
\end{defn}

\begin{defn}
\label{def:surface-from-signed-matching}Given $(\m,\varepsilon)\in\sm^{*}$,
we construct a new pair $(\Sigma_{\m,\varepsilon},f_{\m,\varepsilon})$
where $\Sigma_{\m,\varepsilon}$ is a surface and $f_{\m,\varepsilon}:\Sigma_{\m,\varepsilon}\to\wedger$
as follows:
\begin{itemize}
\item Let $\Sigma_{\m,\varepsilon}$ be the surface obtained by connected
summing a real projective plane $\R P^{2}$ onto each 2-cell of $\Sigma_{\m}$
that is assigned $-1$ by $\varepsilon$.
\item On the neighborhood of each disc that was cut out to perform a connected
sum, homotope $f_{\m}$ to be a constant other than $\{o\}\cup\{(x,k)\}$,
while maintaining the property that the only points in the preimage
of $\{(x,k)\}$ are in $\Sigma_{\m}^{(1)}$. This is possible because
$f_{\m}$ maps each open 2-cell of $\Sigma_{\m}$ to a contractible
piece of $\wedger$. Now extend the function by the constant to the
added $\R P^{2}$. Performing this homotopy then extension for each
$\R P^{2}$ added to $\Sigma_{\m}$ yields $f_{\m,\varepsilon}$.
See Figure \ref{fig:example of a signed matching}.
\end{itemize}
The resulting pair $(\Sigma_{\m,\varepsilon},f_{\m,\varepsilon})$
is an admissible map in the sense of Definition \ref{def:admissible-maps-and-surfaces}. 
\end{defn}

\begin{rem}
Note that $\Sigma_{\m,\varepsilon}$ is no longer a $CW$-complex;
rather it is a $CW$-complex with some 2-cells replaced by Möbius
bands.
\end{rem}

\begin{figure}[t]
\includegraphics[scale=0.65]{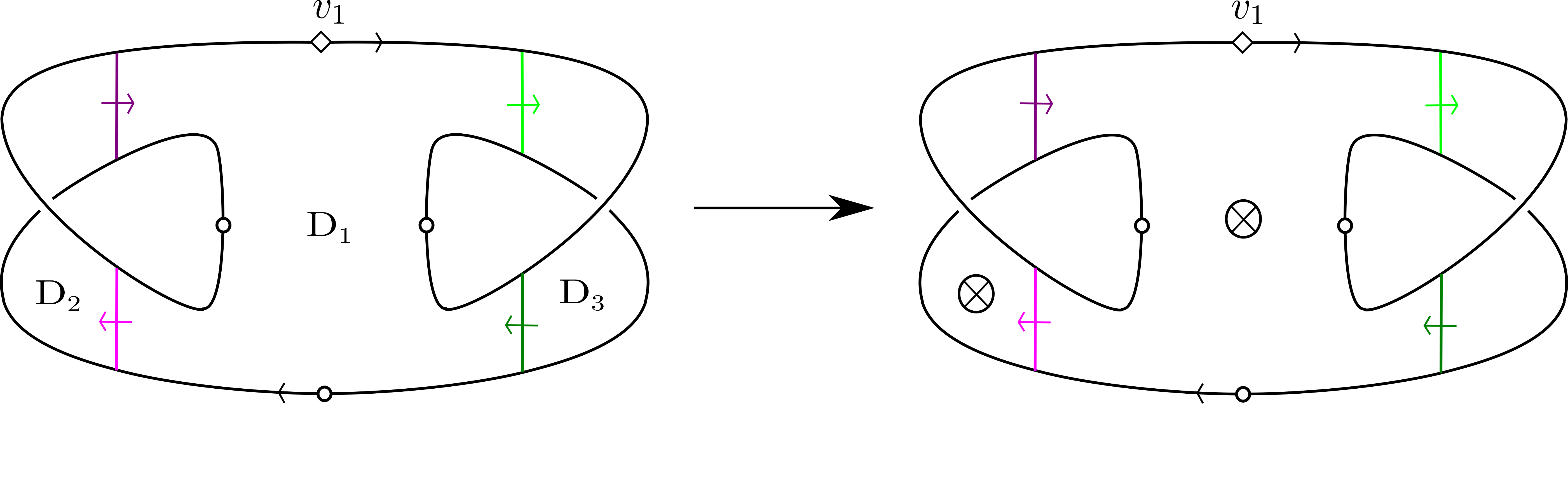}\caption{In this example $\protect\m$ is the same matching data as from Figure
\ref{fig:Constructing-a-surface}. On the left is $\Sigma_{\protect\m}$
with $2$-cells $D_{1},D_{2},D_{3}$ labeled. Here, $\protect\e$
is given by $\protect\e(D_{1})=\protect\e(D_{2})=-1$ and $\protect\e(D_{3})=1$.
The resulting $\Sigma_{\protect\m,\protect\e}$ is drawn on the right,
where we draw a $\otimes$ on the surface to mean an $\protect\R P^{2}$
has been connected summed there.\label{fig:example of a signed matching}}

\end{figure}

The following Lemma is an obvious consequence of the fact that $\chi(\R P^{2})=1$.
\begin{lem}
The Euler characteristic of $\Sigma_{\m,\varepsilon}$ is 
\[
\chi(\Sigma_{\m,\varepsilon})=\chi(\Sigma_{\m})-|\varepsilon^{-1}(\{-1\})|.
\]
\end{lem}

\begin{defn}
There is a map $\forget:\sm^{*}\to\matchr^{*}$ as follows. Given
$(\m,\varepsilon)$, let $\forget(\m,\varepsilon)\in\matchr^{*}$
be the set of matching data obtained by repeatedly replacing pairs
of the form $m_{x,i}=m_{x,i+1}$ by $m_{x,i}$, and re-indexing, until
there are no consecutive duplicate matchings. Note that because of
this removal of duplicates, $\forget$ \uline{does not} respect
the $\Z^{B}$ gradings of $\sm^{*}$ and $\matchr^{*}$ by $\kappa$.
\end{defn}

We now define \marginpar{$M$}
\begin{equation}
M=M(\wl)=\frac{\max\{L_{x}\::\:x\in B\}}{\log2}=\frac{N}{\log2}\label{eq:Mdef}
\end{equation}
where $\log$ denotes the natural logarithm. This $M$ is the quantity
that appears in Theorem \ref{thm:main-theorem}. Note that $M>N$.
The reason for this choice of parameter will be explained in the proof
of the next lemma.
\begin{lem}
\label{lem:local-shift}For $n>M$, for any $\m\in\matchr^{*}$ we
have 
\begin{equation}
(-1)^{|\kappa(\m)|}(n+1)^{\chi(\Sigma_{\m})}=\sum_{(\m',\varepsilon):\forget(\m',\varepsilon)=\m}(-1)^{|\kappa(\m',\varepsilon)|}n^{\chi(\Sigma_{\m',\varepsilon})},\label{eq:generating-function-shift}
\end{equation}
where the right hand side is absolutely convergent.
\end{lem}

\begin{proof}
Write $\m=\{(m_{x,0},\ldots,m_{x,\k_{x}})\}_{x\in B}$. To obtain
$(\m',\varepsilon)$ as in the right hand side of (\ref{eq:generating-function-shift})
from $\m$ we make the following choices, that we split into two types.

\textbf{A.} For each $x\in B$ and $0\leq k\leq\k_{x}$ we choose
$d(x,k)\in\Z_{\geq0}$ and replace $m_{x,k}$ with $d(x,k)+1$ repeats
of $m_{x,k}$. Let $\m'=\{(m'_{x,0},\ldots,m'_{x,\k'_{x}})\}_{x\in B}$
be the resulting new tuples of matchings. Let $J_{x,k}$ be the collection
of $k'$ such that $m'_{x,k'}=m'_{x,k'+1}$ and $m_{x,k'}$ was formed
by duplicating $m_{x,k}$. Hence $|J_{x,k}|=d(x,k)$. For each $k'\in J_{x,k}$
we furthermore have to choose $\varepsilon_{x,k'}$ on the type-$(x,k')$
discs of $\Sigma_{\m'}$, such that $\varepsilon_{x,k'}$ assigns
$-1$ to at least one of these discs. Note that all of these type-$(x,k')$
discs are rectangles, and there are $L_{x}$ of them.

\textbf{B. }Independently of the above, we choose some $\varepsilon_{0}$
on the 2-cells of $\Sigma_{\m}$, since these correspond bijectively
to the two cells of $\Sigma_{\m'}$ that were not created by the previous
step.

Consider the generating function 
\begin{equation}
G(n)=\sum_{(\m',\varepsilon):\forget(\m',\varepsilon)=\m}(-1)^{|\kappa(\m',\varepsilon)|-|\k(\m)|}n^{\chi(\Sigma_{\m',\varepsilon})-\chi(\Sigma_{\m})}.\label{eq:G-def}
\end{equation}
Whatever choices we make in the two steps above, they affect both
$\chi(\Sigma_{\m',\varepsilon})-\chi(\Sigma_{\m})$ and $|\kappa(\m',\varepsilon)|-|\k(\m)|$
independently of one another. Therefore the generating function $G$
splits as a product over $m_{x,k}$ (type \textbf{A }above) and the
discs of $\Sigma_{\m}$ (type \textbf{B }above).

We explain the contribution from the choices of type \textbf{A. }Since
the effect of the choice made for a given $m_{x,k}$ is to contribute
$d(x,k)$ to $|\k(\m',\varepsilon)|-|\k(\m)|$, and for each $k'\in J_{x,k}$
the contribution of $\varepsilon_{x,k'}$ to $\chi(\Sigma_{\m',\varepsilon})-\chi(\Sigma_{\m})$
is $-|\varepsilon_{x,k'}^{-1}(\{-1\})|$, the multiplicative contribution
from a fixed $m_{x,k}$ to $G(n)$ is 
\begin{align*}
\sum_{d(x,k)\geq0}(-1)^{d(x,k)}\prod_{k'\in J_{x,k}}\sum_{\varepsilon_{x,k'}\not\equiv1}n^{-|\varepsilon_{x,k'}^{-1}(\{-1\})|}= & \sum_{d(x,k)\geq0}(-1)^{d(x,k)}\prod_{k'\in J_{x,k}}\left((1+n^{-1})^{L_{x}}-1\right)\\
= & \sum_{d(x,k)\geq0}(-1)^{d(x,k)}\left((1+n^{-1})^{L_{x}}-1\right)^{d(x,k)}\\
= & \frac{1}{1+(1+n^{-1})^{L_{x}}-1}=\frac{1}{(1+n^{-1})^{L_{x}}}.
\end{align*}
All the sums are absolutely convergent when $\left|(1+n^{-1})^{L_{x}}-1\right|<1$
that holds when $n>M$ (this is the reason for the choice of $M$).
Multiplying all these contributions together over all $m_{x,k}$ the
total multiplicative contribution is
\begin{equation}
\prod_{x\in B}\prod_{0\leq k\leq\k_{x}}\frac{1}{(1+n^{-1})^{L_{x}}}=\frac{1}{(1+n^{-1})^{\sum_{x\in B}(\k_{x}+1)L_{x}}}.\label{eq:edges-G}
\end{equation}

Now we explain the contribution from choices of type \textbf{B. }The
choices made contribute $0$ to $|\k(\m',\varepsilon)|-|\k(\m)|$
and $-|\varepsilon_{0}^{-1}(\{-1\})|$ to $\chi(\Sigma_{\m',\varepsilon})-\chi(\Sigma_{\m})$.
Hence the multiplicative contribution of these choices to $G(n)$
is more simply 
\begin{equation}
\sum_{\varepsilon_{0}}n^{-|\varepsilon_{0}^{-1}(\{1\})|}=(1+n^{-1})^{\#\{\text{discs of \ensuremath{\Sigma_{\m}\}}}}.\label{eq:vertices-G}
\end{equation}
Multiplying (\ref{eq:edges-G}) and (\ref{eq:vertices-G}) together
and using (\ref{eq:chi-sigma-1}) we obtain
\begin{align}
G(n) & =(1+n^{-1})^{\#\{\text{discs of \ensuremath{\Sigma_{\m}\}}}-\sum_{x\in B}(\k_{x}+1)L_{x}}\nonumber \\
 & =(1+n^{-1})^{\chi(\Sigma_{\m})}=\left(\frac{n+1}{n}\right)^{\chi(\Sigma_{\m})}.\label{eq:G-formula}
\end{align}
Equating (\ref{eq:G-def}) and (\ref{eq:G-formula}) and rearranging
gives the result.
\end{proof}
\begin{prop}
\label{prop:Second-Laurent-Expansion}For $n>M$, $\tr_{\wl}^{\O}(n+1)$
is given by the following absolutely convergent series:
\begin{align}
\tr_{\wl}^{\O}(n+1) & =\sum_{(\m,\varepsilon)\in\sm^{*}}(-1)^{|\k(\m,\varepsilon)|}n^{\chi(\Sigma_{\m,\varepsilon})}.\label{eq:second-Laurent}
\end{align}
\end{prop}

\begin{proof}
Assume $n>M$. Since $M>N$, we have by Proposition \ref{prop:First-Laurent-Expansion}
and Lemma \ref{lem:local-shift}
\begin{align*}
\tr_{\wl}^{\O}(n+1) & =\sum_{\m\in\matchr^{*}}(-1)^{|\k(\m)|}(n+1)^{\chi(\Sigma_{\m})}\\
 & =\sum_{\m\in\matchr^{*}}\sum_{(\m',\varepsilon):\forget(\m',\varepsilon)=\m}(-1)^{|\kappa(\m',\varepsilon)|}n^{\chi(\Sigma_{\m',\varepsilon})},
\end{align*}
where the right hand side is absolutely convergent. This clearly gives
(\ref{eq:second-Laurent}).
\end{proof}
\begin{cor}
\label{cor:finitely-many-matchings-of-fixed-chi0}For fixed $\wl$
and fixed $\chi_{0}\in\Z$, there are only finitely many elements
$(\m,\varepsilon)$ of $\sm^{*}(\wl)$ with $\chi(\Sigma_{\m,\varepsilon})=\chi_{0}$. 
\end{cor}

\begin{proof}
This could be proven by a direct combinatorial argument similarly
to \cite[Claim 2.10]{MP2}. It is also a direct consequence of the
sum (\ref{eq:second-Laurent}) in Proposition \ref{prop:Second-Laurent-Expansion}
being absolutely convergent for $n>M$. Indeed, if there were infinitely
many $\m\in\sm^{*}$ with $\chi(\Sigma_{\m})=\chi_{0}$ then there
would be infinitely many summands in (\ref{eq:second-Laurent}) with
absolute value $n^{\chi_{0}}$.
\end{proof}
Since every $(\m,\varepsilon)\in\sm^{*}$ gives rise to an admissible
map $(\Sigma_{\m,\varepsilon},f_{\m,\varepsilon})$, it makes sense
to partition elements of $\sm^{*}$ according to the equivalence class
of $(\Sigma_{\m,\varepsilon},f_{\m,\varepsilon})$ in $\surfaces^{*}(\wl)$.
Thus, given $[(\Sigma,f)]\in\surfaces^{*}(\wl)$ we define\marginpar{${\scriptscriptstyle \protect\sm^{*}(\Sigma,f)}$}
\begin{eqnarray*}
\sm^{*}\left(\Sigma,f\right) & \defi & \sm^{*}\left(\wl;\Sigma,f\right)\stackrel{\mathrm{def}}{=}\\
 & \defi & \left\{ (\m,\varepsilon)\in\sm^{*}(\wl)\::\:\left(\Sigma_{\m,\varepsilon},f_{\m,\varepsilon}\right)\approx(\Sigma,f)\right\} .
\end{eqnarray*}
Then we can rewrite Proposition \ref{prop:Second-Laurent-Expansion}
as,
\begin{cor}
\label{cor:Laurent-with-surfaces}For $n>M$,
\[
\tr_{\wl}^{\O}(n+1)=\sum_{[(\Sigma,f)]\in\surfaces^{*}(\wl)}n^{\chi(\Sigma)}\sum_{(\m,\varepsilon)\in\sm^{*}(\Sigma,f)}(-1)^{|\k(\m,\varepsilon)|}.
\]
\end{cor}

Corollary \ref{cor:Laurent-with-surfaces} reduces the proof of our
main theorem (Theorem \ref{thm:main-theorem}), when all unsigned
exponents of $x\in B$ in $\wl$ are even, to the following.
\begin{thm}
\label{thm:Formula-for-l2-euler-char}For $[(\Sigma,f)]\in\surfaces^{*}(\wl)$,
the $L^{2}$-Euler characteristic $\chi^{(2)}(\MCG(f))$ is well-defined,
and given by

\begin{equation}
\chi^{(2)}(\MCG(f))=\sum_{(\m,\varepsilon)\in\sm^{*}(\Sigma,f)}(-1)^{|\k(\m,\varepsilon)|}.\label{eq:euler-char-mcgf-formula-1}
\end{equation}
\end{thm}

The proof of Theorem \ref{thm:Formula-for-l2-euler-char} will be
the subject of $\S$\ref{sec:The-transverse-map} and $\S$\ref{sec:The-action-of}.

\subsection{\label{subsec:Proof-of-Corollary-decay}Proof of Corollary \ref{cor:decay-rate}}
\begin{proof}[Proof of Corollary \ref{cor:decay-rate}]
By Lemma \ref{lem:Culler2}, if $1-\min(\sql(w),2\cl(w))=-\infty,$
then $\surfaces^{*}(w)$ is empty, and Lemmas \ref{lem:when is surfaces nonempty}
and \ref{lem:integral-is-zero} tell us $\trw^{\O}(n)=0$ for all
$n$, which gives the result.

If $1-\min(\sql(w),2\cl(w))$ is finite, then by Proposition \ref{prop:First-Laurent-Expansion},
(\ref{eq:first-laurent}) is a convergent Laurent series in $n^{-1}$
with positive radius of convergence, and the order of the zero at
$\infty$ is at least
\[
-\max_{\m\in\match^{*}(w)}\chi(\Sigma_{\m}).
\]
Note that we only have a bound here, since there could be cancellations
between the coefficients of $n^{\chi}$ for any given $\chi$. On
the other hand, every $\m\in\match^{*}$ gives an admissible map $(\Sigma_{\m},f_{\m})$
so $\max_{\m\in\match^{*}(w)}\chi(\Sigma_{\m})\leq\chi_{\max}(w).$
This implies
\[
\trw^{\O}(n)=O\left(n^{\chi_{\max}(w)}\right)
\]
 as $n\to\infty$. Finally, Lemma \ref{lem:Culler2} tells us we can
replace this by $O(n^{1-\min(\sql(w),2\cl(w))})$ as stated in Corollary
\ref{cor:decay-rate}.
\end{proof}
\begin{rem}
In fact, every admissible map $\left(\Sigma,f\right)\in\surfaces^{*}\left(\wl\right)$
of maximal Euler characteristic (so $\chi\left(\Sigma\right)=\chi_{\max}\left(\wl\right)$)
can be constructed from a suitable $\m\in\match^{*}\left(\wl\right)$,
so we have 
\[
\max_{\m\in\match^{*}(\wl)}\chi(\Sigma_{\m})=\chi_{\max}(\wl).
\]
Indeed, this is a simple generalization of \cite[Thm.~1.5]{CULLER}.
The leading exponent of $\trwl^{O}\left(n\right)$ is strictly smaller
than $\chi_{\max}(\wl)$ only if the coefficients $\chi^{\left(2\right)}\left(\MCG\left(f\right)\right)$
of the maps $\left[\Sigma,f\right]\in\sur^{*}\left(\wl\right)$ with
$\chi\left(\Sigma\right)=\chi_{\max}\left(\wl\right)$ sum up to zero.
\end{rem}

\section{The transverse map complex\label{sec:The-transverse-map}}

In this section we follow \cite[\S 3]{MP2} closely. Our goal here
is to extend the results of (\emph{loc.~cit.}) to surfaces that might
be non-orientable. Since our aim is to prove Theorem \ref{thm:Formula-for-l2-euler-char},
we now fix an admissible map $(\Sigma,f)$ for $\wl$. Recall the
notation from Definition \ref{def:admissible-maps-and-surfaces}.
It will be useful to mark an additional set of points $V_{o}$ on
the boundary of $\Sigma$ with the following properties:
\begin{itemize}
\item $V_{o}$ contains the original marked points $\{v_{j}\}\subset\delta\Sigma$.
\item $V_{o}\subset f^{-1}(\{o\})$, and $V_{o}$ is finite.
\item $\left|\delta_{j}\cap V_{o}\right|=\left|w_{j}\right|$ and so $\delta_{j}-V_{o}$
consists of $|w_{j}|$ intervals. Ordering these intervals according
to the orientation of $\delta_{j}$, beginning at $v_{j}$, the $u$th
interval, directed according to $\delta_{j}$, maps under $f$ to
a loop in $\wedger$ that corresponds to $x_{j_{u}^{k}}^{\varepsilon_{u}^{k}}\in\pi_{1}(\wedger,o)\cong\F_{r}$. 
\end{itemize}
We fix this choice of $V_{o}$ henceforth.

\subsection{Transverse maps on possibly non-orientable surfaces}

We use the terminology \emph{arc} to refer to an embedding of a closed
interval in a compact surface such that the endpoints of the arc are
in the boundary of the surface and these are the only points of the
arc in the boundary. We use the terminology \emph{curve }to refer
to an embedding of a circle in a surface, disjoint from the boundary
of the surface. Note that our notion of curve is what is usually referred
to as a simple closed curve.
\begin{defn}
Let $\Sigma$ be a compact surface. A continuous function $f:\Sigma\to\wedger$
is said to be \emph{transverse} to a point $p\in\wedger-\{o\}$ if
$f^{-1}(\{p\})$ is a disjoint union of arcs and curves, and every
arc or curve in the preimage has a tubular neighborhood that is cut
into two halves by the arc or curve, and the two halves map under
$f$ to two different (local) sides of $p$ in $\wedger$.
\end{defn}

Note that this definition prevents one-sided curves in $f^{-1}(\{p\})$
(see Footnote \ref{fn:one-sided curve} for the definition of a one-sided
curve). 
\begin{defn}
A \emph{transverse map }on $\Sigma$ is a tuple $\k=\{\k_{x}\}_{x\in B}$,
a choice for each $x\in B$ of $\k_{x}+1$ distinct \emph{transversion
points} $(x,0),\ldots,(x,\k_{x})$ in $S_{x}^{1}$, ordered according
to the orientation of $S_{x}^{1}$, and a continuous function $g:\Sigma\to\wedger$
that is transverse to all the points $\{(x,j):x\in B,0\leq j\leq\k_{x}\}$,
such that $g^{-1}\{o\}\cap\delta\Sigma=V_{o}$.

We say that a transverse map $g$ \emph{realizes} $(\Sigma,f)$ if
$g$ is homotopic to $f$ relative to $V_{o}$.

Let $J_{x,j}$ be the connected component of $\wedger-\{(x,j):x\in B,0\leq j\leq\k_{x}\}$
that is bounded by the points $(x,j)$ and $(x,j+1)$. Let $U_{o}$
be the connected component that contains $o$. We call a connected
component of $g^{-1}(U_{o})$ an $o$-zone of $g$ and a connected
component of $g^{-1}(J_{x,j})$ an $(x,j)$-zone of $g$, or if we
do not care about $x$ and $j$, simply an $x$-zone. We say that
a transverse map is \emph{filling} if all its zones are topological
discs. We say the map is \emph{almost-filling} if all its zones are
discs or Möbius bands.

Two transverse maps $g_{1}$ and $g_{2}$ on $\Sigma$ are said to
be \emph{isotopic} if they are homotopic through transverse maps with
the same parameters $\k$. In this homotopy, the points $(x,j)$ are
allowed to vary continuously in $S_{x}^{1}-\{o\}$.
\end{defn}

We refer to transverse maps realizing $(\Sigma,f)$ simply as\emph{
transverse maps. }As in \cite[\S 3]{MP2}, we think of isotopy classes
of transverse maps as isotopy classes of colored arcs and curves with
assigned normal direction: the $(x,j)$-colored arcs and curves are
the components of $g^{-1}\{(x,j)\}$ and the normal direction to the
curve is given by the order in which the two local sides of the arc
and curve map to two local sides of $(x,j)$ in $S_{x}^{1}$, with
the order coming from the fixed orientation of $S_{x}^{1}$.

The following definition is the same as in \cite[\S 3]{MP2}.
\begin{defn}[Loose and strict transverse maps]
We say a transverse map $g$ is \emph{loose }if it satisfies
\begin{description}
\item [{Restriction~1}] There are no $o$-zones or $z$-zones containing
no element of $V_{o}$ with the property that all the bounding arcs
and curves of the zone are pointing inwards, or all pointing outwards,
and all the bounding arcs and curves have the same color. Note this
rules out the possibility that there is a zone that is bounded by
one curve, e.g.~a disc or a Möbius band.
\item [{Restriction~2}] Any segment of the boundary of $\Sigma$ that
is bounded by two same colored endpoints of arcs, that are both directed
inwards or both outwards, must contain an element of $V_{o}$ and
hence be part of an $o$-zone.
\end{description}
The transverse map $g$ is called \emph{strict }if it also satisfies
\begin{description}
\item [{Restriction~3}] For every $x\in B$ and $0\leq j<\k_{x}$ there
must be an $(x,j)$-zone that is neither a rectangle (bounded by two
arcs and two boundary segments) nor an annulus bounded by two curves.
\end{description}
\end{defn}

\begin{rem}
Note that \textbf{Restriction 2} implies that if $g$ is a transverse
map with parameters $\{\k_{x}\}_{x\in B}$, any connected component
sub-interval of $\delta\Sigma-V_{o}$ contains for some $x\in B$
exactly $\k_{x}+1$ points that for some order of the sub-interval
map to $(x,0),\ldots,(x,\k_{x})$ respectively.
\end{rem}

\begin{example}
Consider the pairs $(\Sigma_{\m},f_{\m})$ constructed in $\S$\ref{subsec:Construction-of-amap-on-surface}.
Each of these are admissible maps. Each $f_{\m}$ is a loose transverse
map on $\Sigma_{\m}$, and it is, furthermore, strict, if and only
if $\m\in\matchr^{*}$. In this case, $V_{o}$ are the endpoints of
the intervals used in $\S\ref{subsec:map-on-circle-from-word}$. The
zones of $f_{\m}$ are the 2-cells of $\Sigma_{\m}$, hence $f_{\m}$
is filling. 
\end{example}

\begin{example}
\label{exa:fme-is-a-transverse-map}Consider now the pairs $(\Sigma_{\m,\varepsilon},f_{\m,\varepsilon})$
constructed in Definition \ref{def:surface-from-signed-matching}.
Each of these are admissible maps, and $f_{\m,\varepsilon}$ is a
strict transverse map on $\Sigma_{\m,\varepsilon}$. Now, the zones
of $f_{\m,\varepsilon}$ may be either discs or Möbius bands, depending
on $\varepsilon$. In this case, $f_{\m,\varepsilon}$ is almost-filling.
\end{example}

\subsection{Polysimplicial complexes of transverse maps}
\begin{defn}
The \emph{poset of transverse maps realizing $(\Sigma,f)$, }denoted
$(\T,\preceq)$, has underlying set $\T=\T(\Sigma,f)$ of isotopy
classes $[g]$ of \emph{strict} transverse maps $g$ realizing $(\Sigma,f)$.
The partial order $\preceq$ is defined by $[g_{2}]\preceq[g_{1}]$
if $g_{2}$ is obtained from $g_{1}$ by forgetting transversion points.
(After we forget transversion points we re-index the remaining $(x,i_{0}),\ldots,(x,i_{r})\mapsto(x,0),\ldots,(x,r)$.)
\end{defn}

As in \cite[\S 3]{MP2} we have the following lemmas. 
\begin{lem}
\label{lem:strict-transverse-maps-are-downwards-closed}If $g$ is
a strict transverse map realizing $(\Sigma,f)$ and $g'$ is a transverse
map obtained from $g$ by forgetting points of transversion then $g'$
is also a strict transverse map realizing $(\Sigma,f)$.
\end{lem}

\begin{proof}
Same as \cite[Lem.~3.7]{MP2}.
\end{proof}
\begin{lem}
\label{lem:T-nonempty}$\T=\T(\Sigma,f)$ is not empty.
\end{lem}

\begin{proof}
Same as \cite[Lem.~3.8]{MP2}.
\end{proof}
A polysimplex is a subset of $\R^{k}$ of the form $\Delta_{k_{1}}\times\Delta_{k_{2}}\times\ldots\times\Delta_{k_{r}}$
where $\sum_{j=1}^{r}k_{j}=k$ and the $\Delta_{k_{j}}$ are standard
simplices in $\R^{k_{j}}.$ The polysimplex $\Delta_{k_{1}}\times\Delta_{k_{2}}\times\ldots\times\Delta_{k_{r}}$
has dimension $k$. A \emph{polysimplicial complex }is the natural
generalization of a simplicial complex that allows cells to be polysimplices.
\begin{defn}[Complex of transverse maps]
The \emph{complex of transverse maps realizing $(\Sigma,f)$} is
the polysimplicial complex with a polysimplex $\ps([g])\cong\prod_{x\in B}\Delta_{\k_{x}}$
for each element $[g]$ of $\T(\Sigma,f)$ with associated parameters
$\{\k_{x}\}_{x\in B}$. The faces of $\ps([g])$ are $\ps([g'])$
where $[g']\preceq[g]$. The resulting polysimplicial complex is denoted
$|\T|_{\ps}=\left|\T(\Sigma,f)\right|_{\ps}$. It can be naturally
identified with a closed subset of Euclidean space and is given the
subspace topology.
\end{defn}

\begin{rem}
Lemma \ref{lem:strict-transverse-maps-are-downwards-closed} implies
that the face relations of $|\T|_{\ps}$ make sense: the property
of being a strict transverse map is preserved under passing to sub-faces,
and it is obvious that if $g_{1}$ is a transverse map realizing $(\Sigma,f)$
and $g_{2}$ is obtained from $g_{1}$ by forgetting transversion
points, then $g_{1}$ and $g_{2}$ have the same underlying map and
hence $g_{2}$ realizes $(\Sigma,f)$.

Also note that \textbf{Restriction 3 }implies that any minimal element
$\left[g\right]$ of $\T$ corresponds to exactly one vertex of any
given polysimplex of $|\T|_{\ps}$ containing $\left[g\right]$, so
$|\T|_{\ps}$ is really a polysimplicial complex.
\end{rem}

The poset $(\T,\preceq)$ also gives rise to a simplicial complex
called the \emph{order complex} and denoted by $|\T|$. The $k-$simplices
of $|\T|$ are chains 
\[
[g_{0}]\precneqq[g_{1}]\precneqq\ldots\precneqq[g_{k}]
\]
in $(\T,\preceq)$, and passing to sub-faces corresponds to deleting
elements from chains.
\begin{fact}
\cite[Claim 3.10]{MP2} $|\T|$ is the barycentric subdivision of
$|\T|_{\ps}$. In particular, $|\T|$ and $|\T|_{\ps}$ are homeomorphic.
\end{fact}

The proof of this fact has nothing to do with the issue of whether
$\Sigma$ is orientable, so carries over to the current situation.
\begin{lem}
The complex $|\T|_{\ps}$ is finite dimensional with $\dim\left(|\T|_{\ps}\right)\leq\frac{\ell}{2}-\chi(\Sigma)$.
\end{lem}

\begin{proof}
The proof is along the same lines as the proof of \cite[Lem.~3.12]{MP2}.
The point of the proof is that given a transverse map $g$
\begin{equation}
\chi(\Sigma)=\sum_{\Sigma'}\left(\chi(\Sigma')-\frac{1}{2}\#\{\text{arcs of \ensuremath{g} in the boundary of \ensuremath{\Sigma'\}} }\right)\label{eq:chi-sigma-split}
\end{equation}
where the sum is over zones of $g$ and an arc is counted twice for
$\Sigma'$ if it meets $\Sigma'$ on both sides. 

First we note that because of \textbf{Restriction 1 }$\chi(\Sigma')-\frac{1}{2}\#\{\text{arcs of \ensuremath{g} in the boundary of \ensuremath{\Sigma'\}} }$
is positive only when $\Sigma'$ is a disc that meets exactly one
arc, and on one side. This is still true after dropping the assumption
that $\Sigma$ is orientable, using the classification of surfaces.
Each such zone must be an $o$-zone containing a point $v_{j}$, and
this can only happen if $w_{j}$ is not cyclically reduced. Thus each
of these zones contributes $1/2$ to (\ref{eq:chi-sigma-split}) hence
the contribution of such zones to $\chi(\Sigma)$ is at most $\ell/2$.

As in \cite{MP2}, the zones $\Sigma'$ that contribute $0$ to $\chi(\Sigma)$
include annuli bounded by two curves and rectangles. As $\Sigma$
is not necessarily orientable, there is now the extra possibility
of a Möbius band bounded by a curve. However, this is forbidden by
\textbf{Restriction 1.}

Every $x$-zone $\Sigma'$ not considered thus far contributes at
most $-1$ to $\chi(\Sigma)$. Indeed,\linebreak{}
$\chi(\Sigma')-\frac{1}{2}\#\{\text{arcs of \ensuremath{g} in the boundary of \ensuremath{\Sigma'\}} }$
is an integer, since every $x$-zone meets an even number of arcs,
and we have classified the zones that contribute $\geq0$. Moreover,
for each $x\in B$ and $0\leq k<\k_{x}$, there is an $(x,k)$-zone
contributing at most $-1$ to (\ref{eq:chi-sigma-split}) by \textbf{Restriction
3}. Hence $\dim([g])=\sum_{x\in B}\k_{x}\leq\ell/2-\chi(\Sigma)$.
\end{proof}
The main goal of this $\S$\ref{sec:The-transverse-map} is to record
the following theorem.
\begin{thm}
\label{thm:The-polysimplicial-complex-is=00003Dcontractible}The polysimplicial
complex $|\T|_{\ps}$ is contractible.
\end{thm}

The motivation for this theorem is that $|\T|_{\ps}$ carries an action
of $\MCG(f)$ that will allow us to calculate $\chi^{(2)}(\MCG(f))$
in terms of the orbits of $\MCG(f)$ on $\T$. On the other hand,
these orbits can be related to the terms in (\ref{eq:euler-char-mcgf-formula-1})
(see Lemma \ref{lem:properties-of-m}).

The proof of Theorem \ref{thm:The-polysimplicial-complex-is=00003Dcontractible}
is the same as the proof of \cite[Thm.~3.14]{MP2}. However, there
is one minor point that needs adjusting. Here we refer to terminology
of \cite{MP2} to explain the adjustment for the sake of completeness.
In the classification of maximal null-arc systems on pages 384-385
of \emph{(ibid.)}, it is argued that any component of the complement
of a maximal system $\Omega$ of null-arcs that has one boundary component
consisting of a closed null-arc, contains a pair of pants disjoint
from the curves of $g$, where $g$ is an auxiliary transverse map
with $\k_{x}=0$ for all $x$ and such that the arcs and curves of
$g$ are disjoint from $\Omega$. This should be replaced by the following
analysis. Let $\Sigma'$ be a component of the complement of $\Omega$
that is bounded by a single closed null-arc. If $\Sigma'$ is orientable
then it contains a pair of pants disjoint from the curves of $g$,
and this contradicts the maximality of $\Omega$ as in \cite[pg.~385]{MP2}.
If $\Sigma'$ is not orientable, then $\Sigma'$ contains a simple
closed curve $\gamma$ that bounds a Möbius band, both of which are
disjoint from the arcs and curves of $g$. On the Möbius band consider
the waist curve. We can add a new null-arc that is disjoint from the
old ones and essentially crosses the waist curve of the Möbius band
and is hence not homotopic to any null-arc in $\Omega$. This contradicts
the maximality of $\Omega$.

\section{\label{sec:The-action-of}The action of $\protect\MCG(f)$ on the
transverse map complex.}

\subsection{$L^{2}$-invariants\label{subsec:L2invariants}}

For a discrete group $G$, the $L^{2}$-Euler characteristic $\chi^{(2)}(G)$
is defined as follows. First of all, we make the following definition.
\begin{defn}
We say that $X$ is a $G$-$CW$-complex if $X$ is a $CW$-complex,
with a cellular action of $G$, such that if $g\in G$ preserves an
open cell of $X$, then $g$ acts as the identity on that cell.
\end{defn}

For a discrete group $G$, the \emph{group von Neumann algebra $\N(G)$
}is the algebra of $G$-equivariant bounded operators on $\ell^{2}(G)$.
Let $X$ be a $G$-$CW$-complex, and let $C_{*}(X)$ be the singular
chain complex of $X$. Since $C_{*}(X)$ is a complex of left $\Z[G]$-modules,
we can form the chain complex
\[
\ldots\to\N(G)\otimes_{\Z[G]}C_{p+1}(X)\xrightarrow{d_{p+1}}\N(G)\otimes_{\Z[G]}C_{p}(X)\xrightarrow{d_{p}}\N(G)\otimes_{\Z[G]}C_{p-1}(X)\to\ldots.
\]
This is a complex of left Hilbert $\N(G)$-modules, following \cite[Def.~1.15]{L},
and the boundary maps are bounded $G$-equivariant operators between
Hilbert spaces. We define

\[
H_{p}^{(2)}(X;G)\stackrel{\mathrm{def}}{=}\frac{\ker(d_{p})}{\mathrm{closure}(\mathrm{image}(d_{p+1}))}.
\]
Each of these homology groups is also a Hilbert $\N(G)$-module and
hence has an associated \emph{von Neumann dimension \cite[Def.~6.20]{L}}

\[
b_{p}^{(2)}(X;G)\stackrel{\mathrm{def}}{=}\dim_{\N(G)}H_{p}^{(2)}(X;G)\in[0,\infty].
\]
Following \cite[Def.~6.79]{L}, let
\[
\chi^{(2)}(X;G)\stackrel{\mathrm{def}}{=}\sum_{p\in\Z_{\geq0}}(-1)^{p}b_{p}^{(2)}(X,G)
\]
if the sum is absolutely convergent. Note this assumes at the very
least that all the $b_{p}^{(2)}(X;G)$ are finite. If $EG$ is a contractible
$G$-$CW$-complex with a free action of $G$, and the sum defining
$\chi^{(2)}(EG;G)$ is absolutely convergent, then we define
\[
b_{p}^{(2)}(G)\stackrel{\mathrm{def}}{=}b_{p}^{(2)}(EG;G),\quad\chi^{(2)}(G)\stackrel{\mathrm{def}}{=}\chi^{(2)}(EG;G).
\]
These quantities do not depend on $EG$, so give invariants of $G$
(when they are defined). The reason for this is that $b_{p}^{(2)}(EG;G)$
is invariant under $G$-equivariant homotopy equivalence of $EG$
\cite[Thm.~6.54]{L}, and $EG$ always exists and is unique up to
such homotopy equivalences \cite{tD1,tD2}.

In this paper we will calculate $\chi^{(2)}(\MCG(f))$, for $[(\Sigma,f)]\in\surfaces^{*}(\wl)$,
by other means, which make use of the following definition.
\begin{defn}
Let $\B_{\infty}$ be the class of discrete groups such that all $b_{p}^{(2)}(G)$
are defined and equal to $0$ for all $p\in\Z_{\geq0}$.
\end{defn}

If $G$ is a discrete group and $X$ a $G$-$CW$-complex, and $\sigma$
is a cell of $X$, then we define the \emph{isotropy group $G_{\sigma}$
}to be the stabilizer of $\sigma$ in $G$. We use the convention
that $\frac{1}{|G_{\sigma}|}=0$ if $G_{\sigma}$ is infinite. We
will use the following theorem to calculate $\chi^{(2)}(\MCG(f))$.
\begin{thm}
\label{thm:l2-chi-formula}Let $G$ be a discrete group, and $X$
be a $G$-$CW$-complex with the following properties
\begin{itemize}
\item $X$ is acyclic.
\item All the isotropy groups of $G_{\sigma}$ are either infinite and in
the class $\B_{\infty}$, or finite.
\item We have 
\begin{equation}
\sum_{[\sigma]\in G\backslash X}\frac{1}{|G_{\sigma}|}<\infty.\label{eq:finiteness-condition}
\end{equation}
\end{itemize}
Then $\chi^{(2)}(G)$ is well-defined and given by
\[
\chi^{(2)}(G)=\sum_{[\sigma]\in G\backslash X}(-1)^{\dim(\sigma)}\frac{1}{|G_{\sigma}|}.
\]
\end{thm}

Theorem \ref{thm:l2-chi-formula} is a synthesis of results in Lück
\cite[Thm.~6.80(1), Ex.~6.20]{L}.

As in \cite[\S 4]{MP2}, we use the following theorem, essentially
due to Cheeger and Gromov (cf. \cite[Cor.~0.6]{CG}), as a source
of groups lying in $\B_{\infty}$. The precise statement we need can
be deduced from \cite[Thm.~7.2, items (1) and (2)]{L}. Recall that
a discrete group is called \emph{amenable} if it has a finitely additive
left invariant probability measure.
\begin{thm}[Cheeger-Gromov]
\label{thm:infinite-normal-amenable-subgroup-in-B_infty}If $G$
is a discrete group containing a normal infinite amenable subgroup
then $G\in\B_{\infty}$.
\end{thm}

\subsection{Proof of Theorem \ref{thm:Formula-for-l2-euler-char}\label{subsec:Proof-of-Theorem-MAIN}}

When we apply Theorem \ref{thm:l2-chi-formula} to prove Theorem \ref{thm:Formula-for-l2-euler-char},
we will take $X=|\T|_{\ps}$. Recall that $\T=\T(\Sigma,f)$ for $[(\Sigma,f)]\in\surfaces^{*}(\wl)$.
Let $\Gamma=\MCG(f)$. We now prepare the necessary inputs for Theorem
\ref{thm:l2-chi-formula} in the following lemmas.
\begin{lem}
\label{lem:g-cw-complex}The action of $\Gamma$ on $\T$ makes $|\T|_{\ps}$
into a $\Gamma$-$CW$-complex.
\end{lem}

\begin{proof}
The same as the proof of \cite[Lem.~4.5]{MP2}.
\end{proof}
Analogously to \cite[\S 4]{MP2}, we let $(\T_{\infty},\preceq)$\marginpar{$\protect\T_{\infty}$}
denote the subposet of $\T$ consisting of isotopy classes of transverse
maps that are not almost-filling. Recall a transverse map is almost-filling
if its zones are discs or Möbius bands. This definition marks an essential
departure from \cite{MP2}, where the presence of Möbius bands is
not possible due to the surfaces under consideration being orientable.
In fact, this difference is what is responsible for the shift by the
Jack parameter in our main theorem (Theorem \ref{thm:main-theorem}).
The reason discs and Möbius bands are singled out here is because
these are precisely the type of zones that have trivial mapping class
group:
\begin{lem}
\label{lem:The-mapping-class-group-of-disc-or-mobius}The mapping
class group\footnote{Mapping classes fix the boundary pointwise.}
of a disc or a Möbius band is trivial.
\end{lem}

\begin{proof}
The statement for a disc is the Alexander Lemma \cite[Lem.~2.1]{FM}.
The statement for a Möbius band can be found in Epstein \cite[Thm.~3.4]{EPSTEIN}.
\end{proof}
\begin{defn}
\label{def:generic curve}We say that a two-sided simple closed curve
in $\Sigma$ is \emph{generic }if it is not homotopic to a boundary
component, and does not bound either a disc or a Möbius band.
\end{defn}

We also need the following proposition that appears in \cite[Prop.~4.4]{Stukow}.
\begin{prop}
\label{prop:dehn-twists-generate-free-abelian}If $\Sigma$ is any
surface with boundary, and $c_{1},\ldots,c_{r}$ are a collection
of disjoint, pairwise non-isotopic, generic two-sided simple closed
curves in $\Sigma$, then the Dehn twists in $c_{1},\ldots,c_{r}$
generate a subgroup of $\MCG(\Sigma)$ that is isomorphic to $\Z^{r}$.
\end{prop}

\begin{lem}
\label{lem:isotropy-groups}The isotropy groups $\Gamma_{[g]}$ of
the action of $\Gamma$ on $\T$ can be classified as follows
\begin{itemize}
\item $\Gamma_{[g]}=\{\mathrm{id}\}$ if $[g]\in\T-\T_{\infty}$,
\item $\Gamma_{[g]}$ is infinite and in the class $\B_{\infty}$ if $[g]\in\T_{\infty}$.
\end{itemize}
\end{lem}

\begin{proof}
The proof of the first statement (when $[g]$ is almost-filling) is
similar to the proof of \cite[Lem.~4.7]{MP2}, but incorporating Lemma
\ref{lem:The-mapping-class-group-of-disc-or-mobius} instead of simply
the Alexander Lemma.

The proof of the statement given when $[g]\in\T_{\infty}$ is similar
to the proof of \cite[Lem.~4.8]{MP2}. Given $[g]\in\T_{\infty}$,
we create a list of simple closed curves $c_{1},\ldots,c_{k}$ as
follows. For every boundary component of any zone of $g$, add to
the list the simple closed curve that follows close to the boundary
component inside the zone. After doing so, remove repeats of isotopic
curves (e.g.~if a zone of $g$ is bounded by a simple closed curve,
then in the previous step isotopic curves were created on both sides).
Also remove any curves that are not generic.

Note that by construction the $c_{i}$ are pairwise non-isotopic,
disjoint, two-sided, and generic. We should check that the collection
of $c_{i}$ is not empty. Indeed, since $[g]\in\T_{\infty}$, some
zone $Z$ of $g$ is not a Möbius band or a disc. Hence $Z$ must
have a boundary component that gave rise to a $c_{i}$ that is generic. 

Any Dehn twist $D_{c_{i}}$ in one of the $c_{i}$ is in $\Gamma_{[g]}$,
since $c_{i}$ is disjoint from the arcs and curves of $g$, and $[g]$
is determined by these. Since mapping classes in $\Gamma_{[g]}$ have
representatives that respect the zones of $g$, elements of $\Gamma_{[g]}$
permute the isotopy classes of the $c_{i}$. Hence the group generated
by the $D_{c_{i}}$ (we choose one Dehn twist for each $c_{i}$) is
a normal subgroup of $\Gamma_{[g]}$. By Proposition \ref{prop:dehn-twists-generate-free-abelian},
this subgroup is isomorphic to $\Z^{d}$ with $d\geq1$. Since $\Z^{d}$
is amenable by a result of von Neumann \cite{VN}, we deduce from
Theorem \ref{thm:infinite-normal-amenable-subgroup-in-B_infty} that
$\Gamma_{[g]}$ is in $\B_{\infty}$.
\end{proof}
Recall the sets of signed matchings $\sm^{*}=\sm^{*}(\wl)$ from $\S$\ref{subsec:Signed-matchings}.
Let $[g]\in\T-\T_{\infty}$. We will now describe how $[g]$ naturally
defines an element $(\m([g]),\varepsilon([g]))$ of $\sm^{*}$.

The points $V_{o}$ cut $\delta\Sigma$ into intervals, which by design,
are naturally identified with the sub-intervals of $\cup_{j=1}^{\ell}C(w_{j})$
that were used in their construction. Let $\k=\{\k_{x}\}_{x\in B}$
be the parameters of $[g]$. Consider, for each $x\in B$ and $0\leq k\leq\k_{x}$,
the collection $A(x,k)$ of arcs of $g$ that are in the preimage
of $(x,k)$. These arcs naturally give a matching $m_{x,k}(g)$ of
$\I_{x}(\wl)$. This matching does not change under isotopy of $g$.
Hence reading off all the matchings as $x$ and $k$ vary, we obtain
a tuple of matchings
\[
\m([g])=\{(m_{x,0}(g),\ldots,m_{x,\k_{x}}(g))\}_{x\in B}\in\match^{*}
\]
with $\kappa(\m([g]))=\kappa$. Note that \textbf{Restriction 3, }together
with $[g]\in\T-\T_{\infty}$\textbf{ }implies that if $m_{x,i}=m_{x,i+1}$
then at least one of the $(x,i)$-zones of $[g]$ is a Möbius band. 

By \textbf{Restriction 1}, together with $\left[g\right]\in\T-\T_{\infty}$,
we have that $g^{-1}\left(\left\{ \left(x,i\right)\right\} \right)$
contains no curves but only arcs. By construction, the matching arcs
of $\Sigma_{\m}$ corresponding to $m_{x,i}$ are in one-to-one correspondence
with the connected components of $g^{-1}(\{(x,i)\})$, and any zone
of $g$ corresponds to a 2-cell of $\Sigma_{\m}$ by matching up the
matching arcs on the boundary of the 2-cell. However, a zone of $g$
that is a Möbius band may correspond to a 2-cell of $\Sigma_{\m}$
that is a disc. To record this discrepancy, we define $\varepsilon=\varepsilon([g])$
to assign $1$ to each 2-cell of $\Sigma_{\m}$ that corresponds to
a zone of $g$ that is a disc, and define $\varepsilon$ to assign
$-1$ to any 2-cell of $\Sigma_{\m}$ that corresponds to a zone of
$g$ that is a Möbius band. 

Thus we have defined a map
\[
(\m,\varepsilon):\T-\T_{\infty}\to\sm^{*}.
\]

\begin{lem}
\label{lem:same-matching-implies-differ-by-homeo}Let $[(\Sigma_{i},f_{i})]\in\surfaces^{*}(\wl)$
and $[g_{i}]\in\T(\Sigma_{i},f_{i})$ for $i=1,2$. Then $(\m([g_{1}]),\varepsilon([g_{1}]))=(\m([g_{2}]),\varepsilon([g_{2}]))$
if and only if there is a homeomorphism 
\[
\phi:\Sigma_{1}\to\Sigma_{2}
\]
that respects all markings of the boundaries of the two surfaces and
such that
\[
[g_{2}\circ\phi]=[g_{1}]
\]
(as isotopy classes of transverse maps).
\end{lem}

\begin{proof}
First suppose that $(\m([g_{1}]),\varepsilon([g_{1}]))=(\m([g_{2}]),\varepsilon([g_{2}]))$.
By design of the map $(\m,\varepsilon)$, there are homeomorphisms
\[
\phi_{i}:\Sigma_{\m([g_{i}]),\varepsilon([g_{i}])}\to\Sigma_{i}
\]
such that for $i=1,2$, $f_{\m([g_{i}]),\varepsilon([g_{i}])}\circ\phi_{i}^{-1}$
is a transverse map on $\Sigma_{i}$ that is isotopic to $g_{i}$,
and the $\phi_{i}$ respect the boundary markings of the surfaces.
The map $\phi=\phi_{2}\circ\phi_{1}^{-1}$ satisfies $[g_{2}\circ\phi]=[g_{1}]$.

On the other hand, if $\phi$ is as in the statement of the lemma,
then it is not hard to see that $(\m([g_{1}]),\varepsilon([g_{1}]))=(\m([g_{2}]),\varepsilon([g_{2}]))$.
\end{proof}
\begin{lem}
\label{lem:properties-of-m}The map $(\m,\varepsilon)$ has the following
properties:
\begin{enumerate}
\item $(\m,\varepsilon)$ is invariant under $\Gamma$, that is, for $\gamma\in\Gamma=\MCG(f)$
and $[g]\in\T-\T_{\infty}$,\\
 $(\m(\gamma[g]),\varepsilon(\gamma[g]))=(\m([g]),\varepsilon([g]))$.
\item The image of $(\m,\varepsilon)$ is $\sm^{*}(\Sigma,f)$.
\item $(\m,\varepsilon)$ descends to a bijection $\widehat{(\m,\varepsilon)}:\Gamma\backslash(\T-\T_{\infty})\to\sm^{*}(\Sigma,f)$
that respects the $\Z_{\geq0}^{B}$ gradings of the two sets given
by the two incarnations of $\kappa$ (on transverse maps and signed
matchings).
\end{enumerate}
\end{lem}

\begin{proof}
\emph{Part 1. }This is a special case of one of the implications of
Lemma \ref{lem:same-matching-implies-differ-by-homeo}.

\emph{Part 2. }This follows from the fact that given any $(\m,\varepsilon)\in\sm^{*}(\Sigma,f)$,
by definition $(\Sigma_{\m,\varepsilon},f_{\m,\varepsilon})\approx(\Sigma,f)$.
So there is a homeomorphism $h:\Sigma_{\m,\varepsilon}\to\Sigma$
such that $f_{\m,\varepsilon}\circ h^{-1}$ is homotopic to $f$ relative
to $V_{o}$. On the other hand, $f_{\m,\varepsilon}\circ h^{-1}$
is a strict transverse map realizing $(\Sigma,f)$, all of whose zones
are discs or Möbius bands by construction. Hence $[f_{\m,\varepsilon}\circ h^{-1}]\in\T(\Sigma,f)-\T_{\infty}(\Sigma,f)$
with $\m([f_{m,\varepsilon}\circ h^{-1}])=\m$ and $\varepsilon(f_{\m,\varepsilon}\circ h^{-1})=\varepsilon$
(by Part 1).

\emph{Part 3. }The fact that $(\m,\varepsilon)$ descends to a surjective
map $\widehat{(\m,\varepsilon)}:\Gamma\backslash(\T-\T_{\infty})\to\sm^{*}(\Sigma,f)$
follows from Parts 1 and 2. We need to prove $\widehat{(\m,\varepsilon)}$
is injective, in other words, if $(\m([g_{1}]),\varepsilon([g_{1}]))=(\m([g_{2}]),\varepsilon([g_{2}]))$
then there is some $\gamma\in\MCG(f)$ such that $\gamma([g_{1}])=[g_{2}]$.
The needed $\gamma$ is furnished by Lemma \ref{lem:same-matching-implies-differ-by-homeo}
(taking $\Sigma_{1}=\Sigma_{2}=\Sigma$).
\end{proof}
\begin{cor}
\label{cor:quotient-is-finite.}$\Gamma\backslash(\T-\T_{\infty})$
is finite.
\end{cor}

\begin{proof}
By Lemma \ref{lem:properties-of-m}, Part 3, $\Gamma\backslash(\T-\T_{\infty})$
has the same cardinality as $\sm^{*}(\Sigma,f)$, which is finite
by Corollary \ref{cor:finitely-many-matchings-of-fixed-chi0}, applied
with $\chi_{0}=\chi(\Sigma)$.
\end{proof}

\begin{proof}[Proof of Theorem \ref{thm:Formula-for-l2-euler-char}]
Combining Theorems \ref{thm:The-polysimplicial-complex-is=00003Dcontractible}
and \ref{thm:l2-chi-formula} (with $X=|\T|_{\ps}$ and $G=\Gamma$)
together with Lemmas \ref{lem:g-cw-complex} and \ref{lem:isotropy-groups}
and Corollary \ref{cor:quotient-is-finite.} shows that $\chi^{(2)}(\Gamma)$
is well-defined and given by
\[
\chi^{(2)}(\Gamma)=\sum_{[[g]]\in\Gamma\backslash(\T-\T_{\infty})}(-1)^{|\kappa([g])|}.
\]
Finally, Lemma \ref{lem:properties-of-m}, Part 3, shows that the
above sum can be replaced by 
\[
\chi^{(2)}(\Gamma)=\sum_{(\m,\varepsilon)\in\sm^{*}(\Sigma,f)}(-1)^{|\kappa(\m,\varepsilon)|}
\]
that gives the formula stated in Theorem \ref{thm:Formula-for-l2-euler-char}.
\end{proof}

\subsection{\label{subsec:The--Euler-characteristic-is the usual one}The $L^{2}$-Euler
characteristic is the usual one for almost-incompressible maps}

Recall Definition \ref{def:incompressible and almost-incompressible}
of incompressible and almost-incompressible maps.
\begin{thm}
\label{thm:tame-euler-char}If $[(\Sigma,f)]$ is an almost-incompressible
element of $\surfaces^{*}(\wl)$ then there exists a finite $CW$-complex
$X(f)$ such that $X(f)$ is an Eilenberg-Maclane space of type $K(\MCG(f),1)$
and 
\[
\chi^{(2)}(\MCG(f))=\chi(X(f))
\]
where the right hand side is the usual topological Euler characteristic.
\end{thm}

\begin{rem}
Note that by Lemma \ref{lem:max-chi-implies-incompressible}, Theorem
\ref{thm:tame-euler-char} applies to all $[(\Sigma,f)]$ with\linebreak{}
$\chi(\Sigma)\ge\chi_{\max}(\wl)-1$. 
\end{rem}

\begin{rem}
Since such $X(f)$ are unique up to weak homotopy equivalence, $\chi(X(f))$
is an invariant of $\MCG(f)$ usually simply denoted by $\chi(\MCG(f))$.
\end{rem}

The proof of Theorem \ref{thm:tame-euler-char} relies on the following
lemma.
\begin{lem}
\label{lem:incompressible-have-filling-elements-only}Let $[(\Sigma,f)]\in\surfaces^{*}(\wl)$
be almost-incompressible and $\T=\T\left(\Sigma,f\right)$. Then all
$\left[g\right]\in\T$ are almost-filling, namely, $\T_{\infty}\left(\Sigma,f\right)$
is empty.
\end{lem}

\begin{proof}
Assume that $[g]\in\T$ is not almost-filling. Then there must be
a generic simple closed curve $c$ in $\Sigma$ that is disjoint from
the arcs and curves of $g$. Indeed, if $g$ contains any curves then
we can take $c$ to be parallel to one of these curves, and $c$ would
then be generic by \textbf{Restriction 1}. Otherwise, if there is
a zone of $g$ that is not a topological disc nor a Möbius band, then
we can take $c$ to be any generic simple closed curve in this zone.
Since $c$ lives in only one zone of $g$, $g(c)$ is confined to
a contractible region of $\wedger$, hence is nullhomotopic. Hence
$f(c)$ is also nullhomotopic, since by assumption $g$ is homotopic
to $f$. This contradicts our assumption, hence all $[g]\in\T(\Sigma,f)$
are almost-filling.
\end{proof}
\begin{proof}[Proof of Theorem \ref{thm:tame-euler-char}]
Let $X(f)=\MCG(f)\backslash|\T|_{\ps}$. This is a finite $CW$-complex
by Corollary \ref{cor:quotient-is-finite.} and Lemma \ref{lem:incompressible-have-filling-elements-only}.
Since by Lemma \ref{lem:isotropy-groups} the action of $\MCG(f)$
on $|\T|_{\ps}$ is free and by Theorem \ref{thm:The-polysimplicial-complex-is=00003Dcontractible}
$|\T|_{\ps}$ is contractible, standard arguments show that $X(f)$
is connected, $\pi_{1}(X(f))\cong\MCG(f)$, and the higher homotopy
groups $\pi_{k}(X(f))=0$ for $k\geq2$. This is the statement that
$X(f)$ is an Eilenberg-Maclane space of type $K(\MCG(f),1)$.

For the equality of Euler characteristics, we may use Theorem \ref{thm:l2-chi-formula}
to obtain
\[
\chi^{(2)}(\MCG(f))=\sum_{[\sigma]\in\MCG(f)\backslash\T}(-1)^{\dim(\sigma)}=\sum_{\tau\text{ a cell of \ensuremath{X(f)}}}(-1)^{\dim(\tau)}=\chi(X(f)).
\]
Note that the use of Theorem \ref{thm:l2-chi-formula} is valid since
the sum in (\ref{eq:finiteness-condition}) is finite by finiteness
of $X(f)$.
\end{proof}
\begin{cor}
\label{cor:finitely many almost incompressible}For fixed $\wl$,
there are only finitely many almost-incompressible elements in $\surfaces^{*}(\wl)$.
\end{cor}

\begin{proof}
Suppose $[(\Sigma,f)]\in\surfaces^{*}(\wl)$ is almost-incompressible.
By Lemmas \ref{lem:T-nonempty} and \ref{lem:incompressible-have-filling-elements-only},
$\T(\Sigma,f)$ is non-empty, and all its elements are almost-filling.
By Lemma \ref{lem:strict-transverse-maps-are-downwards-closed}, there
is $[g]\in\T(\Sigma,f)$ with $\kappa\left(\left[g\right]\right)=\left(0,\ldots,0\right)$.
Recalling the maps $(\m,\varepsilon)$ from $\S$\ref{subsec:Proof-of-Theorem-MAIN}
and relying on Lemma \ref{lem:same-matching-implies-differ-by-homeo},
the pair $\left(\m\left(\left[g\right]\right),\varepsilon\left(\left[g\right]\right)\right)$
can only be obtained under the map $\left(\m,\varepsilon\right)$
on $\T(\Sigma,f)$ for one particular $[(\Sigma,f)]\in\surfaces^{*}(\wl)$.
Therefore, the cardinality of the almost-incompressible maps in $\surfaces^{*}(\wl)$
is at most the cardinality of the set 
\[
\left\{ (\m,\varepsilon)\in\sm^{*}(\wl)\,\middle|\,\k(\m)=(0,0,\ldots,0)\right\} .
\]
The latter set is clearly finite, since its elements are obtained
by choosing a finite number of matchings $\m$ of a finite set and
then a finite number of possible maps $\varepsilon$ from the discs
of $\Sigma_{\m}$ to $\left\{ \pm1\right\} $.
\end{proof}
\begin{rem}
The counting yielding the upper bound in the proof of the last corollary
is much redundant. First, if $\left(\m,\varepsilon\right)\in\sm^{*}\left(\wl\right)$
and $\left(\Sigma_{\m,\varepsilon},f_{\m,\varepsilon}\right)$ is
almost-filling, then $\varepsilon$ assigns $+1$ to each $2$-cell
in $\Sigma_{\m}$ except for, possibly, at most one $2$-cell in every
connected component of $\Sigma_{\m}$. Indeed, if there were two zones
in the same connected component of $\Sigma_{\m,\varepsilon}$ which
are Möbius bands, then a simple closed curve tracing the boundary
of one of these zones, continuing to the second zone, tracing its
boundary and going back to the first zone along a parallel path (thus
creating a kind of a barbell-shape) would be a generic, compressing
curve. Second, it is not hard to see that moving a single Möbius band
from one disc of $\Sigma_{\m}$ to another disc in the same connected
component, does not alter the element in $\sur^{*}\left(\wl\right)$.
Therefore, a given matching $\m$ corresponds to at most $2^{\#~\mathrm{connected~components~of}~\Sigma_{\m}}$
almost-incompressible maps in $\sur^{*}\left(\wl\right)$.
\end{rem}

\section{Remaining proofs and some examples\label{sec:Remaining-proofs-and}}

\subsection{\label{subsec:Proof-of-Corollary-limit-counting}Proof of Corollaries
\ref{cor:limit-counting-formula}, \ref{cor:power-limits} and \ref{cor:moment-convergence}}

We begin with the following lemma.
\begin{lem}
\label{lem:annuli-and-mobius-chi-equal-1}If all $w_{j}\neq1$, $[(\Sigma,f)]\in\surfaces^{*}(\wl)$,
and all the connected components of $\Sigma$ are annuli or Möbius
bands, then $\MCG\left(f\right)$ is trivial and, thus, $\chi^{(2)}(\MCG(f))=1$.
\end{lem}

\begin{proof}
Note that $\MCG(f)$ is the product of the mapping class groups of
$f$ restricted to the various connected components of $\Sigma,$
so it is sufficient to prove this when $\ell=2$ and $\Sigma$ is
an annulus or $\ell=1$ and $\Sigma$ is a Möbius band. In the latter
case, the whole mapping class group is trivial as stated in Lemma
\ref{lem:The-mapping-class-group-of-disc-or-mobius} (due to \cite[Thm.~3.4]{EPSTEIN}),
and so, in particular, $\MCG(f)=\{1\}$.

Finally, suppose that $\Sigma$ is an annulus. The mapping class group
of $\Sigma$ is isomorphic to $\Z$ and generated by a Dehn twist
$D$ in a curve parallel to the boundary. Consider a directed arc
$\beta$ connecting $v_{1}$ to $v_{2}$ ($v_{i}$ is the marked point
on the boundary component $\delta_{i}$). Then $f(\beta)$ is a loop
in $\wedger$ based at $o$, and we write $f_{*}(\beta)\in\pi_{1}(\wedger,o)=\F_{r}$
for the class of this loop. For every $n\in\Z$, $D^{n}(\beta)$ is
also an arc in $A$ with the same endpoints. If one boundary component
of $\Sigma$ is labeled $w_{1}$, then $f_{*}\left(D^{n}(\beta)\right)=w_{1}^{\pm n}f_{*}(\beta)\neq f_{*}(\beta)$
for all $0\ne n\in\Z$ since $w_{1}\neq1$. Hence $D^{n}\notin\MCG(f)$
and so $\MCG(f)=\{1\}$.
\end{proof}
\begin{lem}
\label{lem:counting-annuli}Let $w_{1}$ and $w_{2}$ be two words
in $\F_{r}$, both $\neq1$\@. Let $d$ be the maximal integer such
that $w_{1}=u^{d}$ with $u\in\F_{r}$. The number of elements $[(\Sigma,f)]\in\surfaces^{*}(w_{1},w_{2})$
such that $\Sigma$ is an annulus is
\[
\begin{cases}
d & \mathrm{if}~w_{1}~\mathrm{is~conjugate~to~either}~w_{2}~\mathrm{or}~w_{2}^{-1},\\
0 & \mathrm{otherwise}.
\end{cases}
\]
\end{lem}

\begin{proof}
Assume that $\left[\left(\Sigma,f\right)\right]\in\sur^{*}\left(w_{1},w_{2}\right)$
is an annulus. Then, $f:\Sigma\to\wedger$ defines a free homotopy
between $f\left(\delta_{1}\right)$ and $f\left(\delta_{2}\right)$.
Since free homotopy classes of oriented curves in $\wedger$ correspond
to conjugacy classes in $\F_{r}$, this shows that $w_{1}$ must be
conjugate to either $w_{2}$ or $w_{2}^{-1}$, depending on whether
the orientations of the two boundary components of $\Sigma$ agree
or not.

Since no non-identity element of $\F_{r}$ is conjugate to its inverse,
$w_{1}$ cannot be conjugate to both $w_{2}$ and $w_{2}^{-1}$. Without
loss of generality, we assume from now on that $w_{1}$ is conjugate
to $w_{2}$. Let $\beta$ be a directed arc as in the proof of Lemma
\ref{lem:annuli-and-mobius-chi-equal-1}, connecting $v_{1}$ to $v_{2}$.
Denote $b=f_{*}\left(\beta\right)\in\F_{r}$, and note that $w_{1}=bw_{2}b^{-1}$.
Also note that as $\beta$ cuts $\Sigma$ into a disc, the map $f$
is completely determined, up to homotopy, by $f_{*}\left(\beta\right)$.
As in the proof of Lemma \ref{lem:annuli-and-mobius-chi-equal-1},
\[
\left\{ \left(\left[f\right]\circ\left[\rho\right]\right)_{*}\left(\beta\right)\,\middle|\,\left[\rho\right]\in\mcg\left(\Sigma\right)\right\} =\left\{ w_{1}^{n}f_{*}\left(\beta\right)\right\} _{n\in\Z}.
\]
But as the centralizer of $w_{1}$ in $\F_{r}$ is $\left\langle u\right\rangle $,
with $u\in\F_{r}$ the $d$-th root of $w_{1}$ as in the statement
of the lemma, we have $\left\{ c\in\F_{r}\,\middle|\,w_{1}=cw_{2}c^{-1}\right\} =\left\{ u^{n}b\,\middle|\,n\in\Z\right\} $.
Thus, there are exactly $d$ distinct orbits of possible values of
$f_{*}\left(\beta\right)$ under the action of $\mcg\left(\Sigma\right)$,
and therefore exactly $d$ classes of annuli in $\sur^{*}\left(w_{1},w_{2}\right)$.
\end{proof}
\begin{lem}
\label{lem:counting-Mobius}Let $w\neq1$, $w\in\F_{r}$. The number
of elements $[(\Sigma,f)]\in\surfaces^{*}(w)$ such that $\Sigma$
is a Möbius band is 0 if $w$ is not a square and $1$ if $w$ is
a square in $\F_{r}$.
\end{lem}

\begin{proof}
If $w$ is not a square in $\F_{r}$ then there are no $[(\Sigma,f)]\in\surfaces^{*}(w)$
with $\Sigma$ a Möbius band by Lemma \ref{lem:culler}. So suppose
that $w$ is a square in $\F_{r}$. Then by Lemma \ref{lem:culler},
there is at least one $[(\Sigma,f)]\in\surfaces^{*}(w)$ with $\Sigma$
a Möbius band. Let $[(\Sigma,f)]$ be of this form. Then up to homotopy,
there is a unique arc $\alpha(\Sigma)$ in $\Sigma$ joining $v_{1}$
to itself and not separating $\Sigma$ \cite[Proof of Thm.~3.4]{EPSTEIN}.
We have $f^{*}(a)^{2}=w$, which uniquely specifies $f_{*}(\alpha)$.
Let $u$ be the unique solution of $u^{2}=w$ in $\F_{r}$.

Let $(M_{1},f_{1})$ and $(M_{2},f_{2})$ be admissible maps for $w$
with the $M_{i}$ Möbius bands. For $i=1,2$ let $\alpha_{i}$ be
an embedded directed arc from $v_{1}$ to itself in $M_{i}$ that
does not separate $M_{i}$. Since by the previous paragraph $(f_{1})_{*}(\alpha_{1})=(f_{2})_{*}(\alpha_{2})=u$,
the homeomorphism $h$ from $M_{1}$ to $M_{2}$ that preserves the
markings on boundaries and maps $\alpha_{1}$ to $\alpha_{2}$, has
the property that $f_{2}\circ h$ is homotopic to $f_{1}$ and this
shows $[(M_{1},f_{1})]=[(M_{2},f_{2})].$ \emph{Hence there is exactly
one element $[(\Sigma,f)]\in\surfaces^{*}(w)$ with $\Sigma$ a Möbius
band.}
\end{proof}
\begin{proof}[Proof of Corollary \ref{cor:limit-counting-formula}]
 We assume all $w_{j}\neq1$ and examine the expansion given in Theorem
\ref{thm:main-theorem}. The limit $\lim_{n\to\infty}\trwl^{\O}(n)$
exists, since there are no $[(\Sigma,f)]\in\surfaces^{*}(\wl)$ with
$\chi(\Sigma)>0$ by Lemma \ref{lem:chi-max-non-pos}. Moreover, Lemma
\ref{lem:chi-max-non-pos} gives that 
\[
\surfaces_{0}^{*}(\wl)\stackrel{\mathrm{def}}{=}\{[(\Sigma,f)]\in\surfaces^{*}(\wl)\,:\,\chi(\Sigma)=0\}
\]
consists precisely of surfaces all the connected components of which
are annuli or Möbius bands. Thus 

\begin{eqnarray*}
\lim_{n\to\infty}\trwl^{O}\left(n\right) & = & \sum_{[(\Sigma,f)]\in\surfaces_{0}^{*}(\wl)}\chi^{(2)}\left(\MCG(f)\right)=\sum_{[(\Sigma,f)]\in\surfaces_{0}^{*}(\wl)}1\\
 & = & \left|\surfaces_{0}^{*}(\wl)\right|,
\end{eqnarray*}
the second equality following from Lemma \ref{lem:annuli-and-mobius-chi-equal-1}.
This proves the first statement of the corollary.

Every admissible pair $\left[\left(\Sigma,f\right)\right]\in\sur^{*}\left(\wl\right)$
induces a partition on $\left\{ \wl\right\} $ where every block consists
of the words associated with one connected component of $\Sigma$.
The algebraic characterization given in the statement of Corollary
\ref{cor:limit-counting-formula} follows from the geometric part
using these partitions and Lemmas \ref{lem:counting-annuli} and \ref{lem:counting-Mobius}.
\end{proof}

\begin{proof}[Proof of Corollary \ref{cor:power-limits}]
 Assume $w=u^{d}$, where $u\ne1$ is not a proper power. The first
statement of this corollary follows readily from the algebraic characterization
of $\lim_{n\to\infty}\trwl^{\O}\left(n\right)$ in Corollary \ref{cor:limit-counting-formula}:
the valid partitions of the words $w^{j_{1}},\ldots,w^{j_{\ell}}$
depend only on $j_{1},\ldots,j_{\ell}$, and the weight of every partition
depends only on $d$ and not on $u$. The collection of limits determines
$d$ using, for example, the following two equalities:
\[
\lim_{n\to\infty}\trw^{\O}\left(n\right)=\begin{cases}
1 & \mathrm{if}~d~\mathrm{is~even},\\
0 & \mathrm{if}~d~\mathrm{is~odd,}
\end{cases}~~~~\mathrm{and}~~~~\lim_{n\to\infty}\tr_{w,w}^{\O}\left(n\right)=\begin{cases}
d+1 & \mathrm{if}~d~\mathrm{is~even},\\
d & \mathrm{if}~d~\mathrm{is~odd.}
\end{cases}
\]
The analogous result for $\Sp\left(n\right)$ now follows from Theorem
\ref{thm:orth-symp-relation}. 
\end{proof}

\begin{proof}[Proof of Corollary \ref{cor:moment-convergence}]
 Corollary \ref{cor:power-limits} shows that the joint moments of
$T_{n}\left(w\right),\ldots,T_{n}\left(w^{\ell}\right)$ converge
to the same values as the joint moments of $T_{n}\left(x\right),\ldots,T_{n}\left(x^{\ell}\right)$
for some $x\in B$, as long as $w\ne1$ and is not a proper power.
By the method of moments one can now deduce the corollary, using that
Diaconis and Shahshahani have shown in \cite[\S 3]{Diaconis1994}
that these limits of moments are precisely those of the multivariate
normal distribution described in the statement. 
\end{proof}

\subsection{An example: non-orientable surface words\label{sub:frob-schur-example}}

Fix $s\geq1$. Let $w_{s}=x_{1}^{2}\cdots x_{s}^{2}\in\F_{s}$. Here
we analyze $\surfaces^{*}(w_{s})$. Note that $\chi_{\max}\left(w_{s}\right)=1-\sql\left(w_{s}\right)=1-s$,
so there is no admissible map $\left(\Sigma,f\right)\in\sur^{*}\left(w_{s}\right)$
with $\chi\left(\Sigma\right)>1-s$.
\begin{claim}
\emph{\label{claim:admissible surfaces of non-orientable surface words}Let
$t\in\Z_{\ge0}$ be a non-negative integer. Then there is exactly
one $\left[\left(\Sigma_{t},f_{t}\right)\right]\in\surfaces^{*}\left(w_{s}\right)$
with $\chi(\Sigma_{t})=1-s-t$ that can be realized by an almost-filling
strict transverse map. In particular, there is at most one $\left[\left(\Sigma_{t},f_{t}\right)\right]\in\surfaces^{*}\left(w_{s}\right)$
with $\chi^{(2)}(\MCG(f_{t}))\ne0$ and $\chi(\Sigma_{t})=1-s-t$.
In fact,
\begin{equation}
\chi^{\left(2\right)}\left(\mcg\left(f_{t}\right)\right)=\begin{cases}
\delta_{t,0} & \mathrm{if}~s=1,\\
(-1)^{t}{t+s-2 \choose s-2} & \mathrm{if}~s\ge2,
\end{cases}\label{eq:L2ec for surface words}
\end{equation}
where $\delta_{t,0}$ is the Kronecker delta.}
\end{claim}

Note that for all $t\ge0$, $\Sigma_{t}$ is the non-orientable surface
of genus $s+t$ with one boundary component (namely, the connected
sum of $s+t$ copies of $\R P^{2}$, with a disc removed). When $t=0$,
$\mcg\left(f_{0}\right)$ is trivial and $\chi^{\left(2\right)}\left(\mcg\left(f_{0}\right)\right)=\chi\left(\mcg\left(f_{0}\right)\right)=1$.
When $t=1$, a simple analysis gives that $\mcg\left(f_{1}\right)\cong\pi_{1}\left(\Sigma_{0}\right)\cong\F_{s}$,
in which case $\chi^{\left(2\right)}\left(\mcg\left(f_{1}\right)\right)=\chi\left(\mcg\left(f_{1}\right)\right)=\chi\left(\F_{s}\right)=1-s$.
This agrees with the $t=1$ case in (\ref{eq:L2ec for surface words}).
It intrigues us to wonder whether $\MCG(f_{t})$ is related to some
well-known group when $t\ge2$.
\begin{proof}[Proof of Claim \ref{claim:admissible surfaces of non-orientable surface words}]
 If $\left[\left(\Sigma,f\right)\right]\in\sur^{*}\left(w_{s}\right)$
satisfies $\chi^{\left(2\right)}\left(\mcg\left(f\right)\right)\ne0$,
then it must be realized by some almost-filling strict transverse
map, by Theorem \ref{thm:Formula-for-l2-euler-char} and Lemma \ref{lem:properties-of-m}.
So the second statement of the claim follows from the first one.

Now fix $t\ge0$ and assume that $\left[\left(\Sigma,f\right)\right]\in\sur^{*}\left(w_{s}\right)$
satisfies $\chi\left(\Sigma\right)=1-s-t$ and is realized by some
almost-filling strict transverse map $g$. As $g$ is almost-filling,
it has only arcs and no curves. Note that each letter $x_{i}$ appears
exactly twice in $w_{s}=x_{1}^{2}\cdots x_{s}^{2}$, so there is only
one possible matching for every letter $x_{i}$, matching the two
occurrences of $x_{i}$ in $w_{s}$. Therefore, there is a single
$\left(x_{i},j\right)$-zone for every $1\le i\le s$ and $0\le j<\kappa_{x_{i}}\left(\left[g\right]\right)$
which must be a Möbius band. It is easy to check that there is one
$o$-zone in $g$, which may be a disc or a Möbius band. A simple
Euler characteristic calculation shows that exactly $t$ of the zones
of $g$ are Möbius bands. In particular, $\left|\kappa\left(\left[g\right]\right)\right|\in\left\{ t-1,t\right\} $.

Now we modify $[g]$ by forgetting all points of transversion $(x,i)$
with $i\geq1$. Let $[h]\in\T(\Sigma,f)$ be the resulting transverse
map. This $\left[h\right]$ has exactly one zone, it is an $o$-zone,
and by the previous paragraph, this zone was obtained by gluing $t$
Möbius bands together, and is, thus, necessarily the non-orientable
surface of genus $t$ with one boundary component. Because the matchings
in $\left[h\right]$ are dictated and so is the topological type of
its sole zone, any $\left[h'\right]$ obtained in the same way from
some other $\left[\left(\Sigma',f'\right)\right]$ with the same properties,
would be equivalent to $\left[h\right]$. Namely, we could find $[h']\in\T(\Sigma',f')$
with $[h']=[h\circ\phi]$ for $\phi:\Sigma'\to\Sigma$ a homeomorphism
respecting boundary markings. The same $\phi$ shows $(\Sigma',f')\approx(\Sigma,f)$.
Hence there is exactly one \emph{$\left[\left(\Sigma_{t},f_{t}\right)\right]\in\surfaces^{*}\left(w_{s}\right)$
with $\chi(\Sigma_{t})=1-s-t$ which can be realized by an almost-filling
strict transverse map.}

It is left to prove the equality (\ref{eq:L2ec for surface words}).
We prove it in two different ways. First, as mentioned above, any
almost-filling $\left[g\right]\in\T\left(\Sigma_{t},f_{t}\right)$
satisfies $\left|\kappa\left(\left[g\right]\right)\right|\in\left\{ t-1,t\right\} $,
and the same analysis shows that there is exactly one $\mcg\left(f\right)$-orbit
of almost-filling strict transverse maps in $\T\left(\Sigma_{t},f_{t}\right)$
for every valid choice of $\kappa$. There are $\binom{t+s-1}{t}$
possible $\kappa\in\left(\Z_{\ge0}\right)^{s}$ with $\left|\kappa\right|=t$,
each contributing $\left(-1\right)^{t}$ to (\ref{eq:euler-char-mcgf-formula-1}),
and $\binom{t-1+s-1}{t-1}$ possible $\kappa\in\left(\Z_{\ge0}\right)^{s}$
with $\left|\kappa\right|=t-1$, each contributing $\left(-1\right)^{t-1}$
to (\ref{eq:euler-char-mcgf-formula-1}). The total sum is precisely
the one specified in (\ref{eq:L2ec for surface words}).

The second proof uses the Frobenius-Schur type formula (\ref{eq:frobenius}),
by which for $s=1$
\[
\tr_{x^{2}}^{\O}\left(n+1\right)=1,
\]
and for $s\ge2$,
\[
\tr_{w_{s}}^{\O}\left(n+1\right)=\frac{1}{(n+1)^{s-1}}=\frac{1}{n^{s-1}}\left(1-\frac{1}{n}+\frac{1}{n^{2}}-\ldots\right)^{s-1}=\frac{1}{n^{s-1}}\sum_{t=0}^{\infty}(-1)^{t}\binom{t+s-2}{s-2}\frac{1}{n^{t}}.
\]
Combining these two expressions with Theorem \ref{thm:main-theorem},
we see that $\chi^{\left(2\right)}\left(\mcg\left(f_{t}\right)\right)$
is given by (\ref{eq:L2ec for surface words}).
\end{proof}

\subsection{More examples\label{subsec:More-examples}}

We give here some more details about the examples from Table \ref{tab:examples}.
We elaborate on the exact contributions, in the language of Theorem
\ref{thm:main-theorem}, to the two leading terms with exponents $\chi_{\max}$
and $\chi_{\max}-1$. The data is summarized in Table \ref{tab:examples-elaborated}.
The analysis of these examples was carried out with the help of a
SageMath script, and using various observations and considerations.
We do not describe the analysis here as we do not see it as crucial
-- we only aim to give a sense of how our main theorem plays out
in concrete examples.

The fourth column of Table \ref{tab:examples-elaborated} specifies
the rational expressions for $\trwl^{\O}\left(n+1\right)$ (unlike
the expression in Table \ref{tab:examples} which gave the expression
for $\trwl^{\O}\left(n\right)$), as well as the coefficients of $n^{\chi_{\max}}$
and of $n^{\chi_{\max}-1}$ in the Laurent expansion. The fifth column
is the same as the fifth one in Table \ref{tab:examples}, while the
sixth column lists the equivalence classes of maps $\left[\left(\Sigma,f\right)\right]$
in $\sur^{*}\left(\wl\right)$ with $\chi\left(\Sigma\right)=\chi_{\max}-1$.
Note that by Lemma \ref{lem:max-chi-implies-incompressible}, when
$\chi\left(\Sigma\right)=\chi_{\max}$ the maps are incompressible,
and when $\chi\left(\Sigma\right)=\chi_{\max}-1$, the maps are always
almost-incompressible and sometimes even incompressible (in the table
we point out specifically the cases where the stronger condition holds).

Moreover, by Corollary \ref{cor:finitely many almost incompressible},
there are finitely many such equivalence classes in\linebreak{}
$\sur^{*}\left(\wl\right)$, so we can indeed list them all. By Theorem
\ref{thm:tame-euler-char}, in all these cases, we get concrete, finite
$CW$-complexes of type $K\left(\mcg\left(f\right),1\right)$ for
these maps, which means we can understand the groups pretty well.
Indeed, we were able to compute the exact isomorphism type of the
groups $\mcg\left(f\right)$ in all cases mentioned in the table.
The fact all groups but one are free is probably only due to the fact
that the words in these examples are rather short, which means the
complexes associated with them tend to have low dimensions.
\begin{center}
\begin{table*}[!t]
\begin{centering}
\begin{tabular}{|c|>{\centering}p{1.7cm}|>{\centering}p{0.7cm}|>{\centering}p{4.2cm}|>{\raggedright}p{2.5cm}|>{\raggedright}p{5cm}|}
\hline 
\noalign{\vskip\doublerulesep}
$\ell$ & $\wl$ & $\chi_{\max}$ & $\tr_{\wl}^{\O}\left(n+1\right)$ and two leading terms & Admissible maps with $\chi\left(\Sigma\right)=\chi_{\max}$ & Admissible maps with $\chi\left(\Sigma\right)=\chi_{\max}-1$\tabularnewline[\doublerulesep]
\hline 
\hline 
\noalign{\vskip\doublerulesep}
\multirow{6}{*}{1} & $x^{2}y^{2}$ & $-1$ & {\Large{}$\frac{1}{n+1}=\frac{1}{n}-\frac{1}{n^{2}}+\ldots$} & one $P_{2,1}$ w. ${\scriptstyle \mcg\left(f\right)=\left\{ 1\right\} }$ & one $P_{3,1}$ w.~$\mcg\left(f\right)\cong\F_{2}$\tabularnewline[\doublerulesep]
\cline{2-6} \cline{3-6} \cline{4-6} \cline{5-6} \cline{6-6} 
\noalign{\vskip\doublerulesep}
 & $x^{4}y^{4}$ & $-1$ & {\Large{}$\frac{1}{n+1}=\frac{1}{n}-\frac{1}{n^{2}}+\ldots$} & one $P_{2,1}$ w. ${\scriptstyle \mcg\left(f\right)=\left\{ 1\right\} }$ & two incompr.~$P_{3,1}$ w. $\mcg\left(f\right)\cong\Z$; one $P_{3,1}$
w.~$\mcg\left(f\right)\cong\F_{2}$\tabularnewline[\doublerulesep]
\cline{2-6} \cline{3-6} \cline{4-6} \cline{5-6} \cline{6-6} 
\noalign{\vskip\doublerulesep}
 & $[x,y]^{2}$ & $0$ & {\Large{}$\frac{n^{3}+4n^{2}+3n-4}{\left(n+1\right)(n+3)n}$ $=1+\frac{0}{n}+\ldots$} & one $P_{1,1}$ w. ${\scriptstyle \mcg\left(f\right)=\left\{ 1\right\} }$ & one $P_{2,1}$ w.~$\mcg\left(f\right)\cong\Z$\tabularnewline[\doublerulesep]
\cline{2-6} \cline{3-6} \cline{4-6} \cline{5-6} \cline{6-6} 
\noalign{\vskip\doublerulesep}
 & $xy^{3}x^{-1}y^{-1}$ & $-2$ & {\Large{}$0=\frac{0}{n^{2}}+\frac{0}{n^{3}}+\ldots$} & one $P_{3,1}$ w. ${\scriptstyle \mcg\left(f\right)\cong\mathbb{Z}}$ & one $P_{4,1}$ w.~${\scriptstyle \mcg\left(f\right)}\cong\Z\times\F_{3}$\tabularnewline[\doublerulesep]
\cline{2-6} \cline{3-6} \cline{4-6} \cline{5-6} \cline{6-6} 
\noalign{\vskip\doublerulesep}
 & $xy^{4}x^{-1}y^{-2}$ & $-1$ & {\Large{}$\frac{1}{n+1}=\frac{1}{n}-\frac{1}{n^{2}}+\ldots$} & one $P_{2,1}$ w. ${\scriptstyle \mcg\left(f\right)=\left\{ 1\right\} }$ & three incompr.~$P_{3,1}$ w. ${\scriptstyle \mcg\left(f\right)}\cong\Z$;
one $P_{3,1}$ w. ${\scriptstyle \mcg\left(f\right)}\cong\F_{2}$\tabularnewline[\doublerulesep]
\cline{2-6} \cline{3-6} \cline{4-6} \cline{5-6} \cline{6-6} 
\noalign{\vskip\doublerulesep}
 & $xyx^{2}yx^{3}y^{2}$ & $-2$ & {\Large{}$\frac{3n+5}{\left(n+1\right)(n+3)n}={\scriptstyle \frac{3}{n^{2}}-\frac{7}{n^{3}}+\ldots}$} & three $P_{3,1}$ w. ${\scriptstyle \mcg\left(f\right)=\left\{ 1\right\} }$ & one $P_{4,1}$ w.~$\mcg\left(f\right)\cong\F_{3}$;

one $P_{4,1}$ w.~$\mcg\left(f\right)\cong\F_{6}$\tabularnewline[\doublerulesep]
\hline 
\noalign{\vskip\doublerulesep}
\multirow{2}{*}{2} & $w,w$ for $w=x^{2}y^{2}$ & $0$ & {\Large{}$\frac{n^{3}+4n^{2}+7n+8}{\left(n+1\right)(n+3)n}$ $=1+\frac{0}{n}+\ldots$} & one $A$ w. ${\scriptstyle \mcg\left(f\right)=\left\{ 1\right\} }$ & one $P_{1,2}$ w.~$\mcg\left(f\right)\cong\Z$\tabularnewline[\doublerulesep]
\cline{2-6} \cline{3-6} \cline{4-6} \cline{5-6} \cline{6-6} 
\noalign{\vskip\doublerulesep}
 & $w,w$ for $w=x^{2}y$ & $0$ & {\Large{}$1=1+\frac{0}{n}+\ldots$} & one $A$ w. ${\scriptstyle \mcg\left(f\right)=\left\{ 1\right\} }$ & one $P_{1,2}$ w.~$\mcg\left(f\right)\cong\Z$\tabularnewline[\doublerulesep]
\hline 
\noalign{\vskip\doublerulesep}
3 & $w,w,w$ for $w=x^{2}y^{2}$ & $-1$ & {\Large{}$\frac{3(n^{4}+7n^{3}+13n^{2}+15n+24)}{(n-1)n\left(n+1\right)(n+3)(n+5)}$
$=\frac{3}{n}-\frac{3}{n^{2}}+\ldots$} & three $A\sqcup P_{2,1}$ w.~${\scriptstyle \mcg\left(f\right)=\left\{ 1\right\} }$ & three $A\sqcup P_{3,1}$ w.~${\scriptstyle \mcg\left(f\right)\cong\F_{2}}$;
three $P_{1,2}\sqcup P_{2,1}$ w.~${\scriptstyle \mcg\left(f\right)\cong\Z}$\tabularnewline[\doublerulesep]
\hline 
\end{tabular}
\par\end{centering}
\caption{This table gives more details about the examples from Table \ref{tab:examples}.
Here $A$ denotes an annulus, and as in Table \ref{tab:examples},
$P_{g,b}$ denotes the non-orientable surface of genus $g$ with $b$
boundary components (so $\chi\left(P_{g,b}\right)=2-g-b$).\label{tab:examples-elaborated}}
\end{table*}
\par\end{center}

\appendix

\section{Proof of Theorem \ref{thm:orth-symp-relation}: relationship between
$\protect\O$ and $\protect\Sp$\label{sec:symplectic formula}}

Here we prove Theorem \ref{thm:orth-symp-relation}. Throughout this
appendix, fix $n$ with $2n\ge N\left(\wl\right)$, the latter defined
in (\ref{eq:N0}). For $i\in\left[2n\right]$ denote 
\[
\hat{i}\stackrel{\mathrm{def}}{=}\begin{cases}
i+n & \mathrm{if}~1\le i\le n,\\
i-n & \mathrm{if}~n+1\le i\le2n,
\end{cases}
\]
and
\[
\xi\left(i\right)\stackrel{\mathrm{def}}{=}\sign\left(n+\frac{1}{2}-i\right)=\begin{cases}
1 & \mathrm{if}~1\le i\le n,\\
-1 & \mathrm{if}~n+1\le i\le2n.
\end{cases}
\]
Recall that we think of $\Sp\left(n\right)$ as a subgroup of $\mathrm{GL}_{2n}\left(\C\right)$,
and that the matrix $J$ was defined in (\ref{eq:J}). The following
lemma follows easily from (\ref{eq:isom of two forms of Sp(n)}) and
the fact that $A^{-1}=J^{T}A^{T}J$ for $A\in\Sp\left(n\right)$.
\begin{lem}
\label{lem:inverse-formula}If $A\in\Sp\left(n\right)$ and $i,j\in\left[2n\right]$,
then 
\begin{equation}
\left(A^{-1}\right)_{i,j}=\xi\left(i\right)\xi\left(j\right)A_{\hat{j},\hat{i}}.\label{eq:A^-1 of Sp}
\end{equation}
\end{lem}

Our first goal is to obtain an analog of Theorem \ref{thm:orth-wg-exp}
for $\Sp\left(n\right)$, namely, to obtain a formula for $\trwl^{\Sp}\left(n\right)$
as a finite sum over systems of matchings, only with an additional
sign associated with every such system; see Proposition \ref{prop:word integrals on Sp(n) - first step}
for the precise statement.

We recall some of the notation we use here. Let $2L=2\sum_{x\in B}L_{x}=\sum_{j=1}^{\ell}\left|w_{j}\right|$
denote the total number of letters in $\wl$. The $j$th boundary
component of every surface in $\sur^{*}\left(\wl\right)$ is subdivided
to $\left|w_{j}\right|$ intervals corresponding to the letters of
$w_{j}$, and we denoted by $\I,\I^{+},\I^{-}$ the sets of all $2L$
intervals, the subset of intervals corresponding to positive letters
and its complement, respectively. Likewise, we denote by $\I_{x},\I_{x}^{+},\I_{x}^{-}$
the analogous sets of intervals corresponding to the instances of
$x\in B$. We again identify $\I_{x}$ with the set $[2L_{x}]$, for
each $x\in B$, in the same way as in Section \ref{subsec:Construction-of-amap-on-surface}.
Similarly to the notation from Section \ref{subsec:Construction-of-amap-on-surface},
we denote by $\A=\A(\wl)$ the set of index assignments 
\[
\a\colon\left\{ p_{I}\left(k\right)\,\middle|\,I\in\I,k\in\left\{ 0,1\right\} \right\} \to\left[2n\right],
\]
where for every two immediately adjacent marked points $p,q$ in $\cup_{j=1}^{\ell}C\left(w_{j}\right)$
that belong to different intervals in $\I$ we have $\a\left(p\right)=\a\left(q\right)$.
(Note the range here is $\left[2n\right]$ and not $\left[n\right]$
as in Section \ref{subsec:Construction-of-amap-on-surface}). Given
$\a\in\A$, let $\hat{\a}$ be the assignment obtained after applying
(\ref{eq:A^-1 of Sp}), namely,

\[
\hat{\a}(p_{I}(i))=\begin{cases}
\a(p_{I}(i)) & \text{if \ensuremath{I\in\I^{+}}}\\
\widehat{\a(p_{I}(i))} & \text{if \ensuremath{I\in\I^{-}}.}
\end{cases}
\]
As we shall use Theorem \ref{thm:weingarten-integration} for evaluating
$\trwl^{\Sp}\left(n\right)$, we need the following expression which
gathers the total sign contribution for a given system of matchings
$\m=\{(m_{x,0},m_{x,1})\}_{x\in B}\in\match^{\k\equiv1}$ and an assignment
$\a$. Recall the notation $\delta_{\mathbf{i},m}^{\Sp}$ from Theorem
\ref{thm:weingarten-integration}. We let 
\begin{eqnarray*}
\Delta\left(\a,\m\right) & \stackrel{\mathrm{def}}{=} & \left[\prod_{I\in\I^{-}}\xi\left(\mathbf{a}(p_{I}(0))\right)\xi\left(\mathbf{a}(p_{I}(1))\right)\right]\cdot\prod_{\substack{{\scriptscriptstyle x\in B}\\
k=0,1
}
}\delta_{\hat{\a}|_{\left\{ p_{I}\left(k\right)\,\middle|\,I\in\I_{x}\right\} },m_{x,k}}^{\Sp}\\
 & = & \left[\prod_{I\in\I^{-}}\xi\left(\mathbf{a}(p_{I}(0))\right)\xi\left(\mathbf{a}(p_{I}(1))\right)\right]\cdot\prod_{\substack{{\scriptscriptstyle x\in B}\\
k=0,1
}
}\prod_{\substack{{\scriptstyle {\scriptstyle \left(p_{I}(k),p_{J}(k)\right)}}\\
{\scriptstyle \mathrm{matched~by}~\ensuremath{m_{x,k}}}
}
}\left\langle e_{\hat{\a}(p_{I}(k))},e_{\hat{\a}(p_{J}(k))}\right\rangle _{\Sp},
\end{eqnarray*}
where in the innermost product, each matched pair appears once and
is given its predetermined order. Note that for $i,j\in\left[2n\right]$,
we have 
\begin{equation}
\left\langle e_{i},e_{j}\right\rangle _{\Sp}=e_{i}^{T}Je_{j}=\delta_{\hat{i},j}\xi\left(i\right),\label{eq:expression for symplectic form}
\end{equation}
where here $\delta$ is the Kronecker delta. Also notice that $\Delta\left(\a,\m\right)\in\left\{ -1,0,1\right\} $.
We say $\a\vdash^{*}\m$ if $\Delta(\a,\m)\neq0$. Therefore,
\begin{equation}
\Delta\left(\a,\m\right)=\mathbf{1}_{\a\vdash^{*}\m}\cdot\left[\prod_{I\in\I^{-}}\xi\left(\mathbf{a}(p_{I}(0))\right)\xi\left(\mathbf{a}(p_{I}(1))\right)\right]\cdot\prod_{\substack{{\scriptstyle x\in B}\\
{\scriptstyle k=0,1}
}
}\prod_{\substack{{\scriptstyle {\scriptstyle \left(p_{I}(k),p_{J}(k)\right)}}\\
{\scriptstyle \mathrm{matched~by}~\ensuremath{m_{x,k}}}
}
}\xi\left(\hat{\a}\left(p_{I}\left(k\right)\right)\right).\label{eq:Delta-as-signs}
\end{equation}

\begin{defn}
\label{def:positive letters and orientable matching arcs}Let $\m\in\match^{\k\equiv1}$.
Call a matching arc of $\m$ \emph{orientable }if it pairs an interval
in $\I^{\pm}$ with an interval in $\I^{\mp}$, and\emph{ non-orientable}
otherwise. Let $m$ be one of the matchings in $\m$. In every pair
$\left(m_{(2t-1)},m_{(2t)}\right)$ we think of the corresponding
matching arc in $\Sigma_{\m}$ as directed from its \emph{origin}
-- the interval corresponding to $m_{(2t-1)}$, to its \emph{terminus}
-- the interval associated with $m_{(2t)}$. Let $D$ be a type-$o$
disc of $\Sigma_{\m}$. Every interval in $\I$ that meets $\partial D$
has an orientation coming from the given orientation of $\partial\Sigma_{\m}$.
We say that two intervals that meet $\partial D$ are \emph{co-oriented
}(relative to \emph{$D$})\emph{ }if their orientation induces the
same orientation on $\partial D$, and \emph{counter-oriented} otherwise.
Note that a matching arc is orientable if and only if it matches two
co-oriented intervals meeting $\delta D$.
\end{defn}

In the computation of $\Delta\left(\a,\m\right)$, we attribute every
sign that appears in (\ref{eq:Delta-as-signs}) to one of the type-$o$
discs of $\Sigma_{\m}$. Indeed, every matching arc is at the boundary
of exactly one type-$o$ disc, and every $p_{I}(k)$ also belongs
to exactly one type-$o$ disc.
\begin{lem}[Computation of $\Delta(\a,\m)$]
\label{lem:lemma on signs} ~Assume $\a,\a_{1},\a_{2}\in\A$ and
$\m\in\match^{\k\equiv1}$, and let $D$ be a type-$o$ disc in $\Sigma_{\m}$.
\begin{enumerate}
\item \label{enu:The-number-of-non-orientable matching arcs is even}The
number of non-orientable matching arcs along $\partial D$ is even.
\item \label{enu:three types of contribs to Delta}If $\a\vdash^{*}\m$,
the total sign contribution of $D$ to $\Delta(\a,\m)$ is the product
of:\\
$\left(i\right)$ the sign of the index\footnote{Here and elsewhere in this appendix, the ``sign of an index'' $i$
is $\xi\left(i\right)$.} given by $\a$ at the origin of every matching arc with origin in
$\I^{+}$,\\
$\left(ii\right)$ the sign of the index given by $\a$ at the terminus
of every matching arc with terminus in $\I^{-}$, and\\
$\left(iii\right)$ $\left(-1\right)$ for every matching arc with
origin in $\I^{-}$.
\item \label{enu:same index mod n along o-disc}If $\a\vdash^{*}\m$ and
$p_{I}(k),p_{J}(k)$ are matched by any $m_{x,k}$ then $\a(p_{I}(k))\equiv\a(p_{J}(k))\bmod n$,
and moreover, $\a(p_{I}(k))=\a(p_{J}(k))$ if and only if $m_{x,k}$
corresponds to an orientable matching arc.
\item \label{enu:number of assignments}For fixed $\m$, the number of $\a$
with $\a\vdash^{*}\m$ is $\left(2n\right)^{\#\{\text{type-\ensuremath{o} discs of \ensuremath{\Sigma_{\m}}\}}}$.
\item \label{enu:Delta depends only on m}If $\a_{1},\a_{2}\vdash^{*}\m$,
then $\Delta(\a_{1},\m)=\Delta(\a_{2},\m)$.
\end{enumerate}
\end{lem}

The final statement allows us to define:
\begin{defn}
\label{def:Delta m}For every $\m\in\match^{\kappa\equiv1}$ we let
$\Delta(\m)\stackrel{\mathrm{def}}{=}\Delta(\a,\m)$, defined by any
$\a\vdash^{*}\m$.
\end{defn}

\begin{proof}[\emph{Proof of Lemma \ref{lem:lemma on signs}}]
\emph{ }We prove part by part.

\paragraph*{Part \ref{enu:The-number-of-non-orientable matching arcs is even}.\emph{ }}

The first point is due to the fact that the boundary components of
$\Sigma$ have built-in orientation, and along the boundary of $D$,
the orientation of intervals meeting $\partial\Sigma$ is preserved
when going along an orientable matching arc, and flipped along a non-orientable
matching are. But $\partial D$ is a loop, so the number of orientation
flips must be even.

\paragraph*{Part \ref{enu:three types of contribs to Delta}.\emph{ }}

This follows from (\ref{eq:Delta-as-signs}) by checking case by case
over all possibilities.

\paragraph*{Part \ref{enu:same index mod n along o-disc}.\emph{ }}

We have $\a\vdash^{*}\m$ if and only if for all ordered matched pairs
$p_{I}(k),p_{J}(k)$ of any $m_{x,k}$ 
\begin{equation}
\a\left(p_{I}(k)\right)=\begin{cases}
\widehat{\a\left(p_{J}(k)\right)} & \text{if \ensuremath{I} and \ensuremath{J} are in the same set \ensuremath{\I^{\pm}}},\\
\a\left(p_{J}(k)\right) & \text{if \ensuremath{I} and \ensuremath{J} are in \ensuremath{\I^{\pm}} and \ensuremath{\I^{\mp}}, respectively.}
\end{cases}\label{eq:what-happens}
\end{equation}
This means that when $\a\in\A$ and $\a\vdash^{*}\m$, there is a
constraint on the values of $\a$ at every pair of points that are
adjacent on the boundary of some type-$o$ disc of $\Sigma_{\m}$.
This is similar to the situation for the orthogonal group, but the
constraints are more complicated now. The constraint implies that
the values of $\a$ on the points $p_{I}(k)$ in the boundary of a
fixed type-$o$ disc $D$ of $\Sigma_{\m}$ are determined by the
value at any fixed point $p_{D}$ on the boundary of that disc. The
values of $\a$ are constant along segments of $\partial D$, except
for segments that are matching arcs joining intervals in the same
set $I_{x}^{\pm}$, across which the value of $\a$ jumps by $n\bmod2n$.
These are the non-orientable matching arcs defined in Definition \ref{def:positive letters and orientable matching arcs}. 

\paragraph*{Part \ref{enu:number of assignments}.\emph{ }}

It now follows from Parts \ref{enu:The-number-of-non-orientable matching arcs is even}
and \ref{enu:same index mod n along o-disc} that if for each type-$o$
disc $D$ of $\Sigma_{\m}$, we choose $\a(p)$ for some $p$ in $\partial D$,
then there exists a unique $\a\in\A$ with these prescribed values
and such that $\a\vdash^{*}\m$. Hence, for any $\m$, there are $(2n)^{\#\{\text{type-\ensuremath{o} discs of \ensuremath{\Sigma_{\m}}\}}}$
elements of $\A$ with $\a\vdash^{*}\m$.

\paragraph*{Part \ref{enu:Delta depends only on m}.\emph{ }}

We need to show that for $\m$ fixed, all the $\Delta(\a,\m)$ have
the same sign. Indeed, we collect the contribution to the sign of
every $o$-disc $D$ separately, and show it does not depend on the
particular assignment of indices along $\partial D$. There are two
options for the signs of these indices, where one is a complete negation
of the other. Recall that the sign of $\Delta(\a,\m)$ splits up into
three types of contributions according to Part 2. The contribution
from $\left(iii\right)$ clearly does not depend on $\a$. Now consider
the $\left(i\right)$- and $\left(ii\right)$-type contributions.
\begin{itemize}
\item If $\alpha$ is an orientable matching arc, its $\left(i\right)$-
and $\left(ii\right)$-type contributions to $\Delta(\a_{i},\m)$
are always 1 in total. This is surely the case if $\alpha$ is directed
from $I\in I_{x}^{-}$ to $J\in I_{x}^{+}$. But it is also the case
when $\alpha$ is directed the other way round, as the signs of both
indices at its endpoints are identical. Hence the contributions of
type $(i)$ and type $(ii)$ of orientable matching arcs to either
$\Delta(\a_{1},\m)$ or $\Delta(\a_{2},\m)$ is equal to $1$.
\item Note from the discussion in the proof of Part 3, that $\a_{1}$ and
$\a_{2}$ are related by a sequence of the following type of \emph{flip-moves}:
choose a type-$o$ disc $D$ of $\Sigma_{\m}$, and modify $\a_{1}$
by adding $n$ to $\a(p_{I}(k))$ modulo $2n$, for every $p_{I}(k)$
that meets $\partial D$. Now for any given non-orientable matching
arc $\alpha$, its $\left(i\right)$- and $\left(ii\right)$-type
contribution is the sign of one of the endpoints. Hence the effect
of a flip-move on $\a_{1}$ at a disc $D$ is to change the type $(i)$
and $(ii)$ contributions to $\Delta(\a_{1},\m)$ by $(-1)^{\#\text{non-orientable matching arcs of \ensuremath{\m} meeting \ensuremath{D}}}$.
On the other hand, by Part 1, the total number of non-orientable matching
arcs of $\m$ meeting $D$ is even.
\end{itemize}
This concludes the proof of Lemma \ref{lem:lemma on signs}.
\end{proof}
We can now prove the analog of Theorem \ref{thm:orth-wg-exp} for
$\trwl^{\Sp}\left(n\right)$.
\begin{prop}
\label{prop:word integrals on Sp(n) - first step}For $2n\geq N$
\[
\trwl^{\Sp}\left(n\right)=\sum_{\m\in\match^{\k\equiv1}}\left(2n\right)^{\#\{\text{type-\ensuremath{o} discs of \ensuremath{\Sigma_{\m}}\}}}\Delta\left(\m\right)\prod_{x\in B}\wg_{L_{x}}^{\Sp}(m_{x,0},m_{x,1};n),
\]
with $\Delta\left(\m\right)\in\left\{ 1,-1\right\} $ as defined in
Definition \ref{def:Delta m}.
\end{prop}

\begin{proof}
Let $g(I)$ be as in $\S$\ref{subsec:A-rational-function-form-for-orth}.
Assume $2n\geq N$. By the same arguments that led to (\ref{eq:marker}),
incorporating (\ref{eq:A^-1 of Sp}) and using Theorem \ref{thm:weingarten-integration},
we have

\[
\trwl^{\Sp}(n)=\sum_{\a\in\A(\wl)}\sum_{\m\in\match^{\k\equiv1}}\Delta(\a,\m)\prod_{x\in B}\wg_{L_{x}}^{\Sp}(m_{x,0},m_{x,1};n).
\]
\emph{This formula was the original motivation for introducing $\Delta(\a,\m)$.}
Now using Lemma \ref{lem:lemma on signs}, Parts 4 and 5, and interchanging
the sums over $\a$ and $\m$ gives

\begin{align*}
\trwl^{\Sp}(n) & =\sum_{\m\in\match^{\k\equiv1}}\left(\prod_{x\in B}\wg_{L_{x}}^{\Sp}(m_{x,0},m_{x,1};n)\right)\left(\sum_{\a\in\A(w)}\text{\ensuremath{\Delta}(\ensuremath{\a},\ensuremath{\m})}\right)\\
 & =\sum_{\m\in\match^{\k\equiv1}}\left(2n\right)^{\#\{\text{type-\ensuremath{o} discs of \ensuremath{\Sigma_{\m}}\}}}\Delta\left(\m\right)\prod_{x\in B}\wg_{L_{x}}^{\Sp}(m_{x,0},m_{x,1};n)
\end{align*}
as required.
\end{proof}
We can now prove the main result of this subsection and show that
$\trwl^{\Sp}\left(n\right)=\left(-1\right)^{\ell}\cdot\trwl^{\O}\left(-2n\right)$
for large $n$.
\begin{proof}[Proof of Theorem \ref{thm:orth-symp-relation}]
 It follows from Proposition \ref{prop:word integrals on Sp(n) - first step}
and Lemma \ref{lem:connection-of-Weingarten} that
\begin{eqnarray}
\trwl^{\Sp}\left(n\right) & = & \sum_{\m\in\match^{\k\equiv1}}\left(2n\right)^{\#\{\text{\ensuremath{o}-discs of \ensuremath{\Sigma_{\m}}\}}}\Delta\left(\m\right)\prod_{x\in B}\wg_{L_{x}}^{\Sp}(m_{x,0},m_{x,1};n)\nonumber \\
 & = & \sum_{\m\in\match^{\k\equiv1}}\left(2n\right)^{\#\{\text{\ensuremath{o}-discs of \ensuremath{\Sigma_{\m}}\}}}\nonumber \\
 &  & ~~~~~~~~~~~~~~~~~~~~\cdot\prod_{x\in B}\left(-1\right)^{L_{x}}\cdot\sign\left(\sigma_{m_{x,0}}^{-1}\cdot\sigma_{m_{x,1}}\right)\cdot\wg_{L_{x}}^{\O}(m_{x,0},m_{x,1};-2n)\cdot\Delta\left(\m\right)\nonumber \\
 & = & \sum_{{\scriptscriptstyle \m\in\match^{\k\equiv1}}}\left(-2n\right)^{\#\{\text{\ensuremath{o}-discs of \ensuremath{\Sigma_{\m}}\}}}\prod_{x\in B}\wg_{L_{x}}^{\O}\left(m_{x,0},m_{x,1};-2n\right)\Xi\left(\wl;\m\right)\label{eq:using Chi}
\end{eqnarray}
where 
\[
\Xi\left(\wl;\m\right)\stackrel{\mathrm{def}}{=}\left(-1\right)^{\#\{\text{type-\ensuremath{o} discs of \ensuremath{\Sigma_{\m}}\}}}\cdot\left(-1\right)^{L}\cdot\Delta\left(\m\right)\cdot\prod_{x\in B}\sign\left(\sigma_{m_{x,0}}^{-1}\sigma_{m_{x,1}}\right).
\]
We now show that $\Xi\left(\wl;\m\right)$ is independent of $\m$
and equal to $\left(-1\right)^{\ell}$. This will complete the proof
by combining (\ref{eq:using Chi}) with Theorem \ref{thm:orth-wg-exp}.

Our strategy for proving that $\Xi\left(\wl;\m\right)\equiv\left(-1\right)^{\ell}$
consists of three parts:
\begin{enumerate}
\item The fact there are $r=|B|$ different types of letters in $\wl$ can
be ignored, and all letters may be considered as identical.
\item If $\I=\I^{+}$, there is one particular set of matchings $\m$ for
which $\Xi\left(\wl;\m\right)=\left(-1\right)^{\ell}$.
\item The value of $\Xi\left(\wl;\m\right)$ does not change if we make
small local changes: $\left(a\right)$ flipping the direction of one
matching arc, $\left(b\right)$ exchanging the termini of two matching
arcs, or $\left(c\right)$ flipping the orientation of one of the
letters in the word from positive to negative or vice versa.
\end{enumerate}
During this proof we consider the matchings $m_{x,i}$ of $\m$ as
matchings of the letters of the words $\wl$, this is possible since
the letters are in one-to-one correspondence with the intervals $\I$.

\paragraph*{Part I: Consider all letters as identical}

\noindent First, recall the definition from $\S$\ref{subsec:Matchings-and-permutations}
of the permutation $\sigma_{m}\in S_{2k}$ associated with the matching
$m$ belonging to $\m$, and note that the order of the pairs in $m$
does not affect the sign of $\sigma_{m}$, nor $\Delta\left(\m\right)$,
so we ignore it here. (In contrast, the order within each pair does
affect these quantities.) As a result, we can treat all matchings
$\left\{ m_{x,0}\right\} _{x\in B}$ as a single matching $m_{0}\in M_{L}$
of the whole collection of intervals $\I$, where we keep track of
the order within each pair, namely, of which endpoint is the origin
and which the terminus of every matching arc. Similarly, we replace
$\left\{ m_{x,1}\right\} _{x\in B}$ with a single matching $m_{1}\in M_{L}$
of $\I$. The corresponding permutations $\sigma_{m_{0}}$ and $\sigma_{m_{1}}$
lie in $S_{2L}$. From every pair of matchings $m_{0},m_{1}\in M_{L}$,
we can construct a corresponding surface $\Sigma_{m_{0},m_{1}}$ as
in $\S$\ref{subsec:Construction-of-maps}. We define $\Delta\left(m_{0},m_{1}\right)$
accordingly. It is thus enough to show that for every $m_{0},m_{1}\in M_{L}$,
\begin{equation}
\Xi\left(\wl;m_{0},m_{1}\right)\stackrel{\mathrm{def}}{=}\left(-1\right)^{\#\{\text{type-\ensuremath{o} discs of \ensuremath{\Sigma_{m_{0},m_{1}}}\}}}\cdot\left(-1\right)^{L}\cdot\Delta\left(m_{0},m_{1}\right)\cdot\sign\left(\sigma_{m_{0}}^{-1}\sigma_{m_{1}}\right)=\left(-1\right)^{\ell}.\label{eq:Chi with labeless letters}
\end{equation}

\paragraph*{Part II: Particular matchings $m_{0},m_{1}$}

\noindent Next, we show that the equality (\ref{eq:Chi with labeless letters})
holds for a particular pair $m_{0},m_{1}\in M_{L}$, when all letters
in $\wl$ are positive, namely, when $\I=\I^{+}$. The pair will satisfy
$m_{0}=m_{1}$, and so $\sign\left(\sigma_{m_{0}}^{-1}\sigma_{m_{1}}\right)=1$.
We partition the words $\wl$ into singletons of even-length words
and pairs of odd-length words. The matchings $m_{0}$ and $m_{1}$
will only pair letters of words in the same block of this partition.
It is enough to prove (\ref{eq:Chi with labeless letters}) for every
connected component of $\Sigma_{m_{0},m_{1}}$ separately.

First consider the case of a single, even-length word $w$ (which,
by abuse of notation, has length $2L$). Let each of $m_{0},m_{1}$
pair the first interval to the second, the third to the fourth, and
so forth. It is easy to check that in this case, there is exactly
one type-$o$ disc, with $2L$ non-orientable matching arcs at its
boundary, all directed, say, clockwise. In every compatible assignment
of indices $\a\vdash^{*}\left(m_{0},m_{1}\right)$, the sign $\xi$
flips along every matching arc, and as all letters are positive, exactly
half of the matching arcs contribute $\left(-1\right)$ (see Lemma
\ref{lem:lemma on signs}, Part 2), so $\Delta\left(m_{0},m_{1}\right)=\left(-1\right)^{L}$
in this case. Hence the left hand side of (\ref{eq:Chi with labeless letters})
is $\left(-1\right)^{1}\cdot\left(-1\right)^{L}\cdot\left(-1\right)^{L}\cdot1=\left(-1\right)$,
which is the desired outcome as $\ell=1$. See the left hand side
of Figure \ref{fig:particular matchings}.
\begin{figure}[t]
\centering{}\includegraphics[scale=0.8]{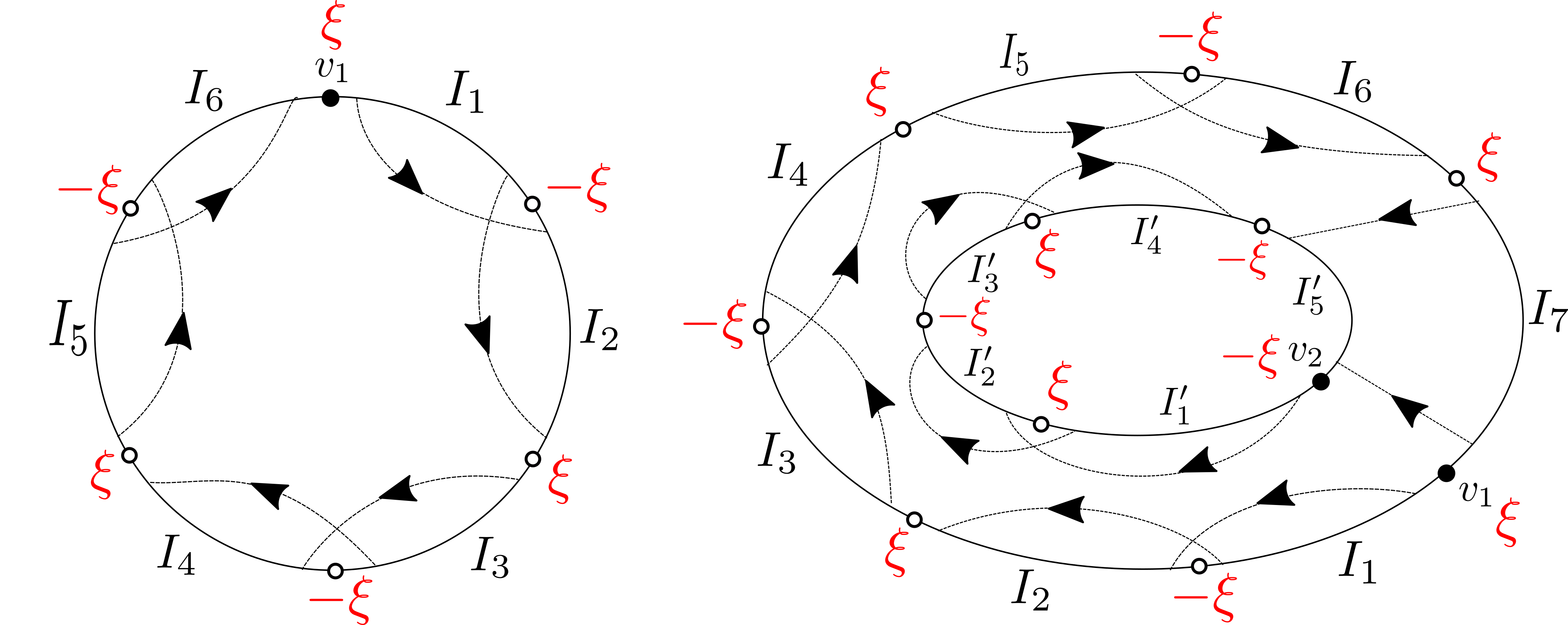}\caption{On the left hand side, there is one word of even length ($6$ in this
example) with all letters positive, and $m_{0}=m_{1}$ match $I_{1}\to I_{2}$,
$I_{3}\to I_{4}$ and $I_{5}\to I_{6}$. An easy computation gives
that $\Xi\left(w;m_{0},m_{1}\right)=-1$ in this case. On the right
hand side, there are two words of odd length each ($7$ and $5$ in
this example) with all letters positive. Here, $m_{0}=m_{1}$ match
$I_{1}\to I_{2}$, $I_{3}\to I_{4}$, $I_{5}\to I_{6}$, $I'_{1}\to I'_{2}$,
$I'_{3}\to I'_{4}$ and $I_{7}\to I'_{5}$. An easy analysis gives
that $\Xi\left(w_{1},w_{2};m_{0},m_{1}\right)=1$ in this case.\label{fig:particular matchings}}
\end{figure}

Second, consider the case of a pair of odd-length words $w_{1},w_{2}$,
of total length $2L$. Let each of $m_{0},m_{1}$ pair the first interval
of each word with the second one, the third with the fourth and so
on, and pair the last interval of $C\left(w_{1}\right)$ with the
last interval of $C\left(w_{2}\right)$. Again, it is easy to verify
there is a single type-$o$ disc in $\Sigma_{m_{0},m_{1}}$, with
$2L$ non-orientable matching arcs at its boundary. At the boundary
of the type-$o$ disc there are $\left|w_{1}\right|$ successive $o$-points
of $C\left(w_{1}\right)$, and then $\left|w_{2}\right|$ o-points
of $C\left(w_{2}\right)$, where the matching arcs separating these
two sequences are the two matchings arcs connecting the last interval
of $w_{1}$ with the last interval of $w_{2}$. In every compatible
assignment $\a\vdash^{*}\left(m_{0},m_{1}\right)$, the signs $\xi$
alternate, and so every pair of matching arcs connecting the same
two intervals of the same word contributes $\left(-1\right)$ to $\Delta\left(m_{0},m_{1}\right)$.
However, both matching arcs connecting the last intervals have the
same sign at their origins, and so their contribution is 1. This shows
that $\Delta\left(m_{0},m_{1}\right)=\left(-1\right)^{L-1}$ in this
case. Hence the left hand side of (\ref{eq:Chi with labeless letters})
is $\left(-1\right)^{1}\cdot\left(-1\right)^{L}\cdot\left(-1\right)^{L-1}\cdot1=1$,
which is the desired outcome as $\ell=2$. See the right hand side
of Figure \ref{fig:particular matchings}.

\paragraph*{Part III: $\Xi$ is invariant under local modifications}

\noindent Finally, we show that the three local modifications we specified
above do not alter the value of $\Xi\left(\wl;m_{0},m_{1}\right)$.
As applying suitable steps of all three types leads from the instance
described in part II of this proof to any given pair of matchings
$m_{0},m_{1}$ and to any orientation of the $2L$ letters (positive/negative),
this will complete the proof. Note that none of these changes affect
the total number of letters, $L$, so we ought to show that they do
not alter the product $\left(-1\right)^{\#\{\text{type-\ensuremath{o} discs of \ensuremath{\Sigma_{m_{0},m_{1}}}\}}}\cdot\Delta\left(m_{0},m_{1}\right)\cdot\sign\left(\sigma_{m_{0}}^{-1}\sigma_{m_{1}}\right)$.

We begin with flipping the direction of one matching arc. Obviously,
this does not change the number of type-$o$ discs. It does change
the sign of one of $\sigma_{m_{0}}$ or $\sigma_{m_{1}}$, and therefore
the sign of $\sigma_{m_{0}}^{-1}\sigma_{m_{1}}$, but it also changes
the contribution of this matching arc to $\Delta\left(m_{0},m_{1}\right)$:
this follows from a simple case-by-case analysis of whether the origin
of the matching arc is in $\I^{+}$ or in $\I^{-}$, and likewise
the terminus of the arc. The analysis is based on Lemma \ref{lem:lemma on signs}
and (\ref{eq:what-happens}).

Next, consider a switch between the termini of two matching arcs $\alpha_{1}$
and $\alpha_{2}$ of, say, $m_{0}$. This switch changes the sign
of $\sigma_{m_{0}}$ and therefore of $\sigma_{m_{0}}^{-1}\sigma_{m_{1}}$.
We distinguish between three cases and show that in each one of them,
there is one more sign change that cancels with the change in $\mathrm{sign}(\sigma_{m_{0}}^{-1}\sigma_{m_{1}})$:
\begin{itemize}
\item Assume that $\alpha_{1}$ and $\alpha_{2}$ both belong to the same
type-$o$ disc $D$ and are directed along the same orientation of
$\partial D$. Then switching the termini splits $D$ into two discs,
and so the sign of $\left(-1\right)^{\#\{\text{type-\ensuremath{o} discs of \ensuremath{\Sigma_{m_{0},m_{1}}}\}}}$
flips. Any compatible assignment $\a$ before the switch remains compatible
after it, and the combined sign contribution of the two arcs (as in
Lemma \ref{lem:lemma on signs}, Part 2) remains unchanged. See Figure
\ref{fig:switching ends of arcs 1}\\
\begin{figure}[t]
\centering{}\includegraphics{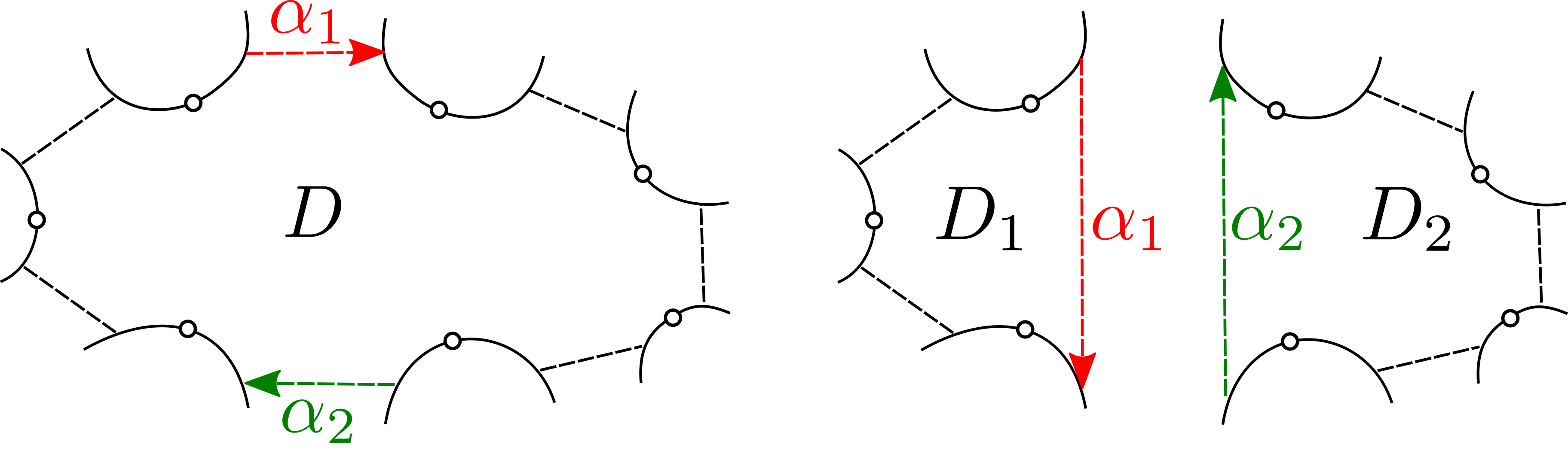}\caption{The left part depicts a type-$o$ disc $D$ with two co-directed matching
arcs $\alpha_{1}$ and $\alpha_{2}$ at its boundary. Switching their
termini results in splitting $D$ into two separate type-$o$ discs:
$D_{1}$ and $D_{2}$, as in the right hand side. This move flips
both $\mathrm{sign}\left(\sigma_{m_{0}}^{-1}\sigma_{m_{1}}\right)$
and $\left(-1\right)^{\#\{\text{type-\ensuremath{o} discs of \ensuremath{\Sigma_{m_{0},m_{1}}}\}}}$,
but leaves $\Delta\left(m_{0},m_{1}\right)$, and therefore also $\Xi\left(\protect\wl;m_{0},m_{1}\right)$,
unchanged.\label{fig:switching ends of arcs 1}}
\end{figure}
\item Assume that $\alpha_{1}$ and $\alpha_{2}$ both belong to the same
type-$o$ disc $D$ and are directed along different orientations
of $\partial D$. Of the two components of $\partial D\setminus\left(\alpha_{1}\cup\alpha_{2}\right)$,
one, denoted $C_{o}$, has the origins of $\alpha_{1}$ and $\alpha_{2}$
as endpoints, and the other, denoted $C_{t}$, has the two termini
as endpoints. Switching the termini corresponds to reflecting $C_{t}$
-- see Figure \ref{fig:switching ends of arcs 2}. By the definition
of compatible assignments, every piece of \texttt{$\partial D\cap\partial\Sigma$}
is assigned a well-defined index in $\left[2n\right]$, and by Lemma
\ref{lem:lemma on signs}, Part 3, two different pieces of \texttt{$\partial D\cap\partial\Sigma$}
are assigned the same index if and only if the corresponding orientations
induced by $\partial\Sigma$ induce, in turn, the same orientation
on $\partial D$. This means that if we preserve the assignment along
$C_{o}$, the signs along $C_{t}$ must be flipped. The number of
type-$o$ discs is preserved. In the terminology of Lemma \ref{lem:lemma on signs},
type-$\left(iii\right)$ contributions to $\Delta\left(m_{0},m_{1}\right)$
do not change. The sign contributions of the matching arcs along $C_{o}$
do not change. Also, the contribution of orientable arcs along $C_{t}$
does not change, nor does the type-$\left(i\right)$ contribution
of $\alpha_{1}$ and $\alpha_{2}$. However, $\Delta\left(m_{0},m_{1}\right)$
does flip. To see this, denote by $\partial_{1},\partial_{2}$ the
two connected components of $\partial D\cap\partial\Sigma$ which
contain (as endpoints) the two termini $t\left(\alpha_{1}\right)$
and $t\left(\alpha_{2}\right)$, respectively ($\partial_{1}$ and
$\partial_{2}$ may be equal). Let $i_{1}$ and $i_{2}$ be the indices
corresponding to $\partial_{1}$ and $\partial_{2}$ in some compatible
assignment (before the flip of $\alpha_{1}$ and $\alpha_{2}$).\\
\begin{figure}[t]
\includegraphics[scale=0.9]{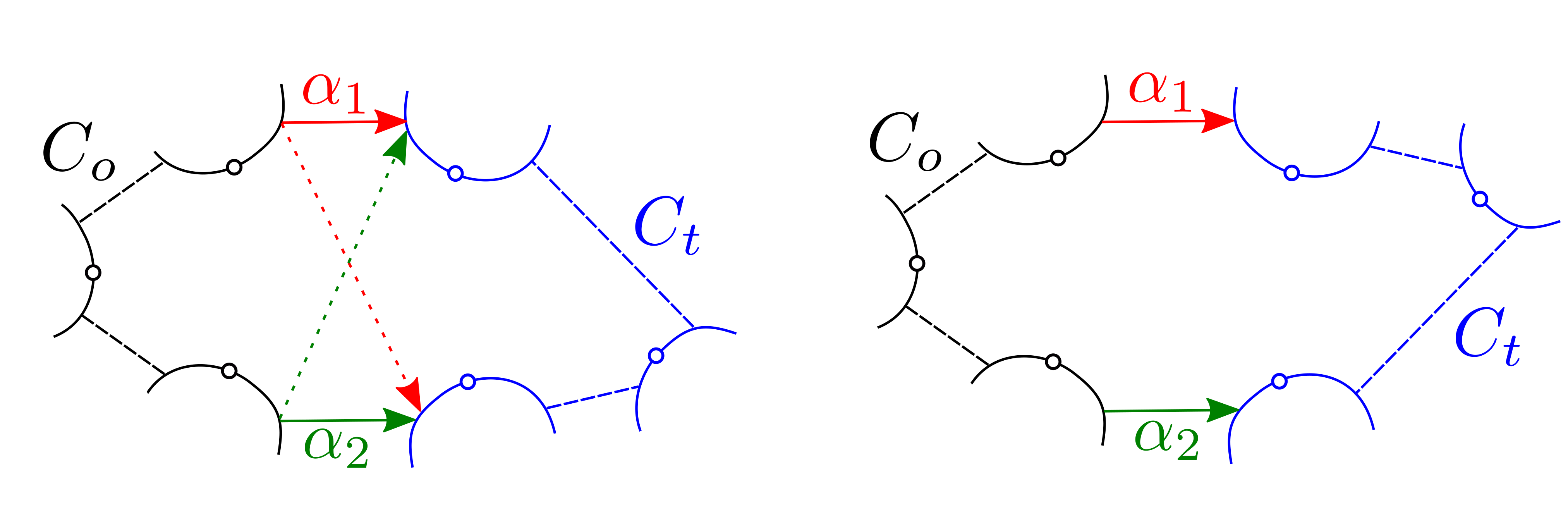}\caption{The left part depicts a type-$o$ disc $D$ with two counter-directed
matching arcs $\alpha_{1}$ and $\alpha_{2}$ at its boundary. The
two connected components of $D\setminus\left(\alpha_{1}\cup\alpha_{2}\right)$
are denoted $C_{o}$ and $C_{t}$. Switching the termini of $\alpha_{1}$
and $\alpha_{2}$ results in reflecting $C_{t}$, as in the right
hand side. This move flips both $\mathrm{sign}\left(\sigma_{m_{0}}^{-1}\sigma_{m_{1}}\right)$
and $\Delta\left(m_{0},m_{1}\right)$, but leaves $\left(-1\right)^{\#\{\text{type-\ensuremath{o} discs of \ensuremath{\Sigma_{m_{0},m_{1}}}\}}}$,
and therefore also $\Xi\left(\protect\wl;m_{0},m_{1}\right)$, unchanged.\label{fig:switching ends of arcs 2}}
\end{figure}
\\
If the orientation of $\partial\Sigma$ along $\partial_{1}$ and
$\partial_{2}$ induces the same orientation on $\partial D$, then
$i_{1}=i_{2}$ and of the two intervals at the termini of $\alpha_{1}$
and $\alpha_{2}$, one is in $\I^{+}$ and the other in $\I^{-}$.
Thus the total type-$\left(ii\right)$ contribution of $\alpha_{1}$
and $\alpha_{2}$ flips. As $C_{t}$ contains an even number of non-orientable
matching arcs in this case, the total sign contribution of the non-orientable
arcs along $C_{t}$ is preserved (as in the proof of Lemma \ref{lem:lemma on signs},
Part 5).\\
If the orientation of $\partial\Sigma$ along $\partial_{1}$ and
$\partial_{2}$ induces different orientations on $\partial D$, then
$i_{2}=\widehat{i_{1}}$ and the two letters at the termini are both
positive or both negative. In this case, the total type-$\left(ii\right)$
contribution of $\alpha_{1}$ and $\alpha_{2}$ is unchanged, the
total type-$\left(i\right)$ and type-$\left(ii\right)$ contribution
of every orientable arc alongs $C_{t}$ is unchanged, but the same
contribution of every non-orientable arc along $C_{t}$ is flipped,
and the total number of non-orientable arcs along $C_{t}$ is odd
(by our assumption about $\partial_{1}$ and $\partial_{2}$).
\item The third and last case is the one where $\alpha_{1}$ and $\alpha_{2}$
belong to different type-$o$ discs. Switching their termini then
leads to merging the two discs into one. In the united type-$o$ disc,
the two arcs are ``co-oriented'', so this case is the reverse of
the first one, and $\Delta\left(m_{0},m_{1}\right)$ remains unchanged.
\end{itemize}
The final small change we consider is that of flipping some letter
from being positive to negative, namely, of flipping an interval in
some $C\left(w_{j}\right)$ from $\I^{\pm}$ to $\I^{\mp}$. Here,
$\sign\left(\sigma_{m_{0}}^{-1}\sigma_{m_{1}}\right)$ is unchanged.
By the first local modification in this part of the proof, we may
assume without loss of generality that this letter is at the termini
of two matching arcs, $\alpha_{1}$ and $\alpha_{2}$. A similar argument
as in the previous paragraph would show that:
\begin{itemize}
\item Assume that $\alpha_{1}$ and $\alpha_{2}$ belong to the same type-$o$
disc $D$ with the same orientation. The flip of the letter then splits
$D$ into two type-$o$ discs. Denote by $\partial_{1}$ and $\partial_{2}$
the pieces of $\partial D\cap\partial\Sigma$ at the termini of $\alpha_{1}$
and $\alpha_{2}$. They must be counter-oriented. We may preserve
the same assignment of indices as before the flip of the letter, but
then the type-$\left(ii\right)$ contribution of both arcs flips when
the letter is flipped. No other change in sign contributions occurs.
\item Assume that $\alpha_{1}$ and $\alpha_{2}$ belong to the same type-$o$
disc $D$ with opposite orientations. The flip of the letter preserves
the number of type-$o$ discs and corresponds to reflecting $C_{t}$.
Here $\partial_{1}$ and $\partial_{2}$ are co-oriented and the signs
along $C_{t}$ must be flipped. There is no change to $\Delta\left(m_{0},m_{1}\right)$:
the total type-$\left(ii\right)$ contribution of $\alpha_{1}$ and
$\alpha_{2}$ is $1$ before and after the flip, and the number of
non-orientable arcs along $C_{t}$ is even.
\item If $\alpha_{1}$ and $\alpha_{2}$ belong to different type-$o$ disc,
the flip is the reverse of the first case.
\end{itemize}
\noindent This completes the proof of Theorem \ref{thm:orth-symp-relation}.
\end{proof}
We now have the analog of Corollary \ref{cor:orth-rational-function}
for $G=\Sp$:
\begin{cor}
\label{cor:symp-rational-function}There is a rational function $\overline{\trwl^{\Sp}}\in\Q(n)$
such that for $2n\ge\max\{L_{x}\::\:x\in B\}$, $\trwl^{\Sp}(n)$
is given by evaluating $\overline{\trwl^{\Sp}}$ at $n$.
\end{cor}

\begin{proof}
Theorem \ref{thm:orth-symp-relation} shows that we can obtain $\overline{\trwl^{\Sp}}$
by switching $n$ with $-2n$ in $\overline{\trwl^{\O}}$ (the rational
function from Corollary \ref{cor:orth-rational-function}) and multiplying
by $(-1)^{\ell}$.
\end{proof}
\bibliographystyle{alpha}
\bibliography{database_united}

\noindent Michael Magee, Department of Mathematical Sciences, Durham
University, Lower Mountjoy, DH1 3LE Durham, United Kingdom

\noindent \texttt{michael.r.magee@durham.ac.uk}\\

\noindent Doron Puder, School of Mathematical Sciences, Tel Aviv University,
Tel Aviv, 6997801, Israel\\
\texttt{doronpuder@gmail.com}
\end{document}